\documentclass{amsart}

\usepackage[hypertexnames=false,pdftex,pdfpagemode=UseNone,breaklinks=true,extension=pdf,colorlinks=true,linkcolor=blue,citecolor=blue,urlcolor=blue]{hyperref}

\usepackage[margin=1in]{geometry}
\usepackage{mathtools} 
\usepackage{graphicx} 
\usepackage[svgnames]{xcolor}
\usepackage{amsthm}
\usepackage{amsmath}
\usepackage{amssymb}
\usepackage{enumitem}
\usepackage{comment}
\usepackage[normalem]{ulem}
\usepackage[all]{xy}
\usepackage{enumitem}
\usepackage{soul}

\newcommand{\R}{\mathbb{R}}
\newcommand{\C}{\mathbb{C}}
\newcommand{\Z}{\mathbb{Z}}
\newcommand{\Q}{\mathbb{Q}}

\newcommand{\M}{\mathfrak{M}} 

\newcommand{\Jplus}{\mathbb{J}_+}
\newcommand{\Jzero}{\mathbb{J}_0}
\newcommand{\lambdapert}{\lambda_{\varepsilon,L}}
\newcommand{\Jpert}{J_{\varepsilon}}
\newcommand{\upert}{u_{\varepsilon}}

\newcommand{\op}{\operatorname}
\newcommand{\at}{\mathfrak{at}}

\newcommand{\simp}{\operatorname{simp}}

\newcommand{\ind}{\operatorname{ind}}
\newcommand{\wind}{\operatorname{wind}}
\newcommand{\image}{\operatorname{image}}

\newtheorem{theorem}{Theorem}[section]
\newtheorem{lemma}[theorem]{Lemma}

\newtheorem{cor}[theorem]{Corollary}
\newtheorem{proposition}[theorem]{Proposition}

\newtheorem*{T:ot}{Theorem \ref{proposition ot}}
\newtheorem*{T:concavetoric}{Theorem \ref{theorem general}}
\newtheorem*{T:concaveOT}{Theorem \ref{theorem ot}}
\newtheorem*{T:concaveTight}{Theorem \ref{theorem tight}}
\newtheorem*{T:at}{Theorem \ref{theorem at}}
\newtheorem*{T:construction}{Theorem \ref{theorem:construction}}
\newtheorem*{T:asymptotics}{Theorem \ref{theorem:asymptotics}}
\newtheorem*{T:noright}{Theorem \ref{thm:no curve from right}}

\theoremstyle{definition}
\newtheorem{definition}[theorem]{Definition}
\newtheorem{setup}[theorem]{Setup}
\newtheorem{convention}[theorem]{Convention}
\newtheorem{notation}[theorem]{Notation}

\theoremstyle{remark}
\newtheorem{example}[theorem]{Example}

\newtheorem{remark}[theorem]{Remark}

\DeclarePairedDelimiter{\ceil}{\lceil}{\rceil}
\DeclarePairedDelimiter{\floor}{\lfloor}{\rfloor}

\title[Contact toric structures and concave boundaries of linear plumbings]{Properties of contact toric structures and concave boundaries of linear plumbings}

\author[A. Marinkovi\'c]{Aleksandra Marinkovi\'c}
\address{A. Marinkovi\'c, University of Belgrade, Faculty of Mathematics, Belgrade, Serbia}
\email{aleksandra.marinkovic@matf.bg.ac.rs}

\author[J. Nelson]{Jo Nelson}
\address{J. Nelson,
Rice University, Houston, Texas}
\email{jo.nelson@rice.edu}

\author[A. Rechtman]{Ana Rechtman}
\address{A. Rechtman, Institut Fourier,
Universit\'e Grenoble Alpes,
100 rue des math\'ematiques,
38610, Gi\`eres, France\\
Institut Universitaire de France (IUF)}
\email{ana.rechtman@univ-grenoble-alpes.fr}

\author[L. Starkston]{\\Laura Starkston}
\address{L. Starkston,
UC Davis, California}
\email{lstarkston@ucdavis.edu}

\author[S. Tanny]{Shira Tanny}
\address{S. Tanny, Weizmann Institute of Science, Rehovot, Israel}
\email{tanny.shira@gmail.com}

\author[L. Wang]{Luya Wang}
\address{L. Wang, Institute for Advanced Study, Princeton, New Jersey}
\email{luyawang@ias.edu}

\date{}

\begin{document}

\begin{abstract}
We consider plumbings of symplectic disk bundles over spheres admitting concave contact boundary, with the goal of understanding the geometric properties of the boundary contact structure in terms of the data of the plumbing. We focus on the linear plumbing case in this article. We study the properties of the contact structure using two different sets of tools. First, we prove that all such contact manifolds have a global contact toric structure, and use tools from toric geometry to identify when the contact structure is tight versus overtwisted. Second, we study algebraic torsion measurements from embedded contact homology (ECH) for these concavely induced contact manifolds, which has largely been unexplored. We develop a toolkit establishing existence and constraints of pseudoholomorphic curves adapted to the Morse-Bott Reeb dynamics of these plumbing examples, to provide the ECH algebraic torsion and contact invariant calculations for the concave boundaries of linear plumbings.

\end{abstract}
\maketitle

\tableofcontents

\section{Introduction}

This article makes connections between contact toric geometry, plumbing concave caps and algebraic torsion measurements in embedded contact homology (ECH). The motivation stems from problems on the realization of singular complex algebraic plane curves, as well as questions about symplectic fillability of a large natural class of contact manifolds and detection of non-fillability and tightness through Floer theories.

Our main results include the following. Beginning in the setting of toric geometry, we explain a complete criterion to detect overtwistedness through an explicit overtwisted disk in contact toric 3-manifolds. We develop a construction of a global symplectic toric structure on any minimal concave linear plumbing inducing a contact toric structure on its boundary. We use the term concave in this article to mean that there is an inward pointing Liouville vector field along the boundary. Through these two results, we give a characterization of which concave linear plumbings have tight contact boundaries and which have overtwisted contact boundaries in terms of the normal Euler numbers appearing in the plumbing. 

Next we turn to the Floer theory side, and our theory of choice is embedded contact homology (ECH). Our goal is to calculate a measurement called ``the algebraic torsion" in ECH \cite[Appendix {by Hutchings}]{at} for contact structures arising as concave boundaries of plumbings. The algebraic torsion can be thought of as a more subtle measurement than the contact invariant and can detect non-fillability. We calculate it {as well as the ECH contact invariant} in the case of contact toric $3$-manifolds by developing a toolkit for analyzing the {Morse-Bott} Reeb dynamics and pseudoholomorphic curves. This toolkit includes constructions of pseudoholomorphic curves, obstructions coming from positivity of intersection and action arguments, and constraints coming from the asymptotic expansion of the asymptotic ends of the curve and the interaction between the toric structure with the stable and unstable manifolds of hyperbolic orbits of the Reeb flow.

In future work, we plan to apply this toolkit to more complicated examples of plumbings which are non-linear, in order to detect tightness and non-fillability in these contact manifolds. A particularly interesting subset of contact manifolds obtained as the concave boundary of a plumbing, are those arising as the boundary of a neighborhood of a singular complex plane curve, where the plumbing is obtained by considering the normal crossing resolution. If the singular curve can be realized in the complex plane, the corresponding contact structure is symplectically (Stein) fillable. However, there are many types of singular configurations which cannot be realized, and many of these have non-fillable contact structures (see~\cite{GollaStarkston}). Our motivating question is, can we use the contact structure to detect non-realizability? If the contact structure were overtwisted it certainly would not be fillable and thus the singular curve would not be realizable. More subtly, the contact structure could be tight, but non-fillability could in principle be detected through algebraic torsion, still obstructing realizability of the singular curve.

From the point of view of ECH, there are very few examples of computations in the literature with interesting topology. Most existing examples study different contact forms on $S^3$. Other examples in the literature include the three-torus \cite{HutchingsSullivanT3}, prequantization bundles \cite{NelsonWeiler} and connected sums \cite{Luya}. Computations of a transverse knot filtered version on $S^3$ were done in~\cite{Weiler, NelsonWeiler2, NelsonWeiler3}. A preprint of Choi provided a conjectural combinatorial formula for the ECH complex for invariant contact forms on $I\times T^2$
and work in the preprint of Yao~\cite{Yao} developing the Morse-Bott setting under a technical hypothesis makes some progress towards showing this combinatorial formula aligns with the ECH differential.  In particular, additional work is needed to  establish that the combinatorial and ECH differentials agree after collapsing at end points of the interval, when the generators include the elliptic orbits appearing as the images of the two collapsed tori, cf. \cite[\S 5.6]{NelsonWeiler3}.  In our paper, we will not rely on these works, but rather prove our computations directly. 

One goal of this article is to expand the range of computations of ECH quantities to more topologically complex $3$-manifolds. We look closely at lens spaces (including $S^1\times S^2$), but our tools are designed to be useful in our future work studying these invariants on other graph manifolds (boundaries of plumbings). There are many examples of concave boundaries of plumbings which have been shown to be non-fillable through fairly subtle ad hoc methods distinct from those we use here~\cite{GollaStarkston}. Thus we expect that these examples can exhibit a range of different finite non-zero values of algebraic torsion, corresponding to different levels of non-fillability that are still tight. This would enable the observed phenomena in natural examples to be understood in terms of direct ECH algebraic torsion measurements. Then applying our computational methods to examples where fillability is unknown could in turn give new results about fillability, tightness, and singular curve realizability.

Now we will explain our motivation and results in more detail.

\subsection{Concave plumbings}

A contact manifold can arise naturally as the boundary of a symplectic manifold if there is a transverse Liouville vector field near the boundary. When the Liouville vector field points outward from the boundary, we say the boundary is convex, while when it points inward, we say the boundary is concave. Contact manifolds arising as the convex boundary of a symplectic manifold are (strongly) symplectically fillable and are known to be tight~\cite{gromov,eliashbergfilling}. Thus, there are contact manifolds which cannot arise as a convex boundary of any symplectic manifold. While overtwisted structures were the first such example, there are tight examples as well starting with~\cite{EtnyreHondaTightNonfillable}. It is generally an interesting problem to identify contact manifolds which fail to be fillable in this more subtle (tight) way. On the other hand, every contact manifold arises as the \emph{concave} boundary of a symplectic manifold~\cite{etnyrecaps,eliashbergcaps,conwayetnyrecaps}. Thus, if we want to understand properties of a concave contact boundary, we will need specific information about the symplectic topology of the concave cap, not just its existence.

A particularly nice class of symplectic $4$-manifolds with boundary are those built from plumbing together disk bundles over surfaces. Given any collection of transversely intersecting symplectic surfaces in a symplectic $4$-manifold, its tubular neighborhood is a symplectic plumbing. As mentioned above, a normal crossing resolution of a singular complex curve is a significant example. Thus, it is relatively easy to find and recognize plumbings inside of a closed symplectic $4$-manifold.

It is often convenient to encode a plumbing with a decorated graph as follows. The vertices correspond to the disk bundles over surfaces. Each vertex is decorated by a pair of integers $(g_i,s_i)$ giving the genus of the surface and the Euler number of the disk bundle. There is an edge between two vertices if the corresponding disk bundles are plumbed together (i.e. if the corresponding surfaces intersect). To specify a symplectic structure, we must also keep track of the symplectic area of each surface, which can be recorded by upgrading the decoration to $(g_i,s_i,a_i)$.

Some work was done to explore when symplectic plumbings admit a convex or concave Liouville structure. Gay and Stipsicz~\cite{gs} proved that negative definite plumbings admit a convex filling structure. Li and Mak~\cite{lm} used the same method to show that many plumbings which are not negative definite admit a concave cap structure.

Our aim in this program is to find ways of characterizing key geometric properties of contact manifolds arising as the concave boundary of a plumbing. We are particularly interested in characterizing when the contact manifold is overtwisted versus tight, and when it is symplectically fillable or non-fillable.

In this article, we focus on an initial restricted class of plumbings: linear plumbings of disk bundles over spheres. In this case, the contact 3-manifolds appearing on the boundary are lens spaces (including $S^3$ and $S^1\times S^2$). Since there is a classification of tight contact structures on lens spaces~\cite{Honda} which implies that every tight contact structure is fillable, this class will not be a source of examples of tight but non-fillable contact manifolds. However, more general concave plumbings are likely to admit such examples, and our work in the linear case is an integral step in ascertaining such phenomena. 

In particular, there are plumbings coming from normal crossing resolutions of non-realizable singular complex curves, whose concave contact boundaries are known to be non-fillable, but through quite subtle arguments~\cite{GollaStarkston}. It is not currently known whether these contact boundaries are tight or overtwisted. There are no obvious overtwisted disks, and the subtlety of the obstructions to fillability suggest that these contact manifolds may lie in this interesting range of tight but non-fillable manifolds. Analyzing these contact structures is the goal of the next stage of this program which we will address in future work. This paper develops the machinery to study concave boundaries of linear plumbings and provides a toolkit that we will use to study more complicated plumbings.

\subsection{Contact toric 3-manifolds}

While there is an abundance of articles and results about toric structures on symplectic manifolds, the study of their contact cousins has been more limited. Important foundational work was done by Banyaga and Molino (\cite{BM93},\cite{BM96}), Boyer and Galicki (\cite{BG00}) and Lerman (\cite{Lerman_contact_toric}) to characterize contact toric manifolds and their moment map images.  In this article, we focus on $3$-dimensional contact toric manifolds. By work of Lerman~\cite{Lerman_contact_toric}, the underlying 3-manifolds must be $T^3$, or a lens space $L(p,q)$ (including $S^3=L(1,0)$ and $S^1\times S^2=L(0,1)$). The $T^3$ case is distinguished by the property that the toric action is free. The contact manifold in the non-free case is characterized by two real numbers, representing the starting and ending angles of the moment image curve which we denote by $t_1$ and $t_2$. We demonstrate the existence of an overtwisted disk in the contact manifold in the case when the moment image curve traverses an angle strictly greater than $\pi$, and see that when the angle is less than or equal to $\pi$ the contact structure is tight. See also \cite[Proposition 9.11]{sym42}.

\begin{T:ot}
	Let $(Y^3,\xi)$ be a compact connected contact manifold with a non-free toric action obtained from Lerman's construction from real numbers $t_1,t_2$. The contact structure $\xi$ is overtwisted if and only if  $t_2-t_1> \pi$.
\end{T:ot}

In concurrent work~\cite{MarinkovicStarkston}, the first and fourth authors classify all contact toric $3$-manifolds up to contactomorphism, thus characterizing which contact structures support a toric action.

Next, we turn our attention to $4$-dimensional symplectic plumbings of disk bundles over surfaces admitting a concave Liouville structure along their boundary. Our goal is to understand the contact boundary of such a concave plumbing, in terms of the data of the decorated plumbing graph. We focus on linear plumbings of disk bundles over $2$-spheres. Namely, the graph is connected and linear ({there are two vertices of valency $1$ and the rest have valency $2$}), and $g_i=0$ for all $i$. Thus, the data of the plumbing can be encoded by the ordered list of normal Euler numbers following the linear ordering of the graph. We show, with a non-negativity condition which is necessary for concavity, that such a plumbing admits a symplectic toric structure with concave contact toric boundary.

\begin{T:concavetoric} 
	A linear plumbing $(s_1,\ldots,s_n)$ where each surface is a 2-sphere and $s_i \geq0$ for at least one index $i=1, \ldots, n$ admits a structure of a symplectic toric 4-manifold with a concave contact toric boundary.
\end{T:concavetoric} 

Figure~\ref{fig:3figs} represents the moment map images of three such plumbings, where the green curves correspond to the boundaries. The radial rays emanating from the origin  point transversally into the polytopes. Thus, the corresponding vector field is Liouville, pointing inward the manifolds. 
In particular, this proves that such plumbings admit a concave Liouville structure. More general conditions for the existence of a concave Liouville structure are provided in~\cite{lm}, but our proof through toric geometry is independent of their results. A converse to this theorem is given in~\cite{MarinkovicStarkston}--that every contact toric manifold arises as the concave boundary of a linear plumbing. The main power of Theorem~\ref{theorem general}, is that it allows us to study the contact manifold induced on the concave boundary of the plumbing using the contact toric structure and Theorem~\ref{proposition ot}. By analyzing our construction of the toric structure on the plumbing, we are able to determine conditions on the normal Euler numbers ensuring that the angle traversed by the moment map image is strictly greater than $\pi$, and thus the contact structure is overtwisted by Theorem~\ref{proposition ot}.

\begin{T:concaveOT} The contact structure on the  boundary of a linear plumbing $(s_1,\ldots,s_n)$, 
	is overtwisted if for some index $i\in\{1,\ldots,n\}$ where   $s_i\geq0$,  one of the following cases occurs:
\begin{itemize}
    \item[(a)] either $s_is_{i+1}\geq 2$ or  $s_is_{i+1}\geq1$ and $n>2$;
    \item[(b)] there exists  $j\in\{1, \ldots, n\}$ such that $|i-j|>1$, and for which either $s_j\geq1$ or $s_j\geq0$ and $n>3$;
  \item[(c)] $s_i=0$ and either $s_{i-1}+s_{i+1}\geq1$ or $s_{i-1}+s_{i+1}\geq0$ and $n>3$.     
\end{itemize}       
\end{T:concaveOT}

Similarly, we provide conditions on the normal Euler numbers such that the contact structure on the concave boundary of the corresponding plumbing is tight.

\begin{T:concaveTight} 
 The contact structure on the boundary of a linear plumbing $(s_1,\ldots,s_n)$,  where   $s_i\geq0$  for at least one $i\in\{1,\ldots,n\},$ is tight if one of the following cases occurs:   
 \begin{itemize}
     \item[(a)] $s_j\leq -2$, for all $j\neq i,$
     
 \item[(b)]$s_i=0$,  $s_{i-1}+s_{i+1} \leq-2$ and
$s_j \leq-2$, for all $j\neq i,i+1.$        \end{itemize}
\end{T:concaveTight}

   \begin{figure}\centering
\includegraphics[width=17cm]{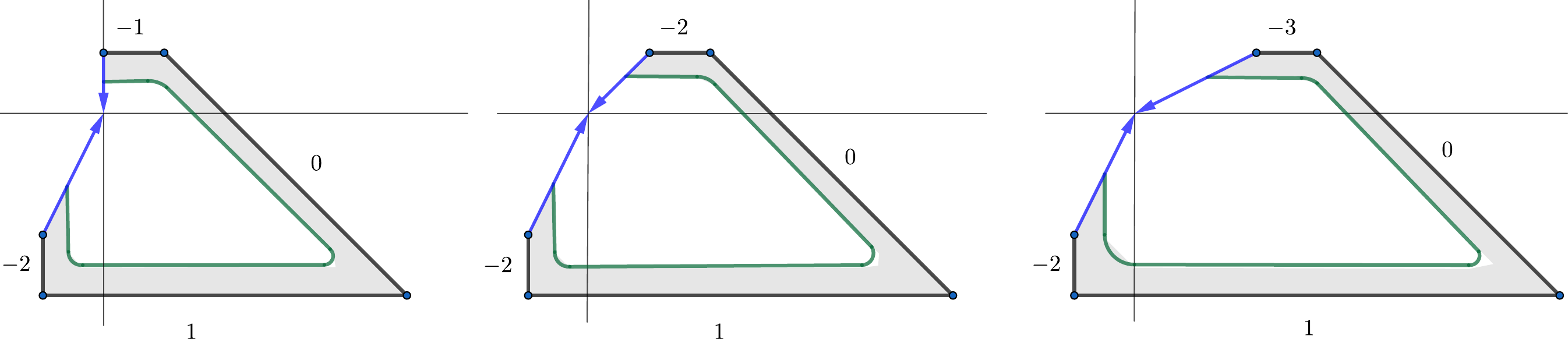}
\caption{The moment map images of the plumbings $(-2,1,0,-1)$ (left), $(-2,1,0,-2)$ (middle) and $(-2,1,0,-3)$ (right).
}\label{fig:3figs}
   \end{figure}

Note that these two theorems are very close to complementary, but there are some borderline cases of concave plumbings that are not covered by either of these two theorems. This is not an indication of a failure of the techniques, but rather the complexity of the statement. There is a fractal-like boundary (an infinitely repeating pattern) between the tight and overtwisted cases, of which condition (c) of Theorem~\ref{theorem ot} and condition (b) of Theorem~\ref{theorem tight} are the first level. We discuss other cases in Remark~\ref{remark tight or ot}.
{In Figure \ref{fig:3figs} we demonstrate this subtlety by presenting the moment map images of  three plumbings.
The contact structure on the boundary of  $(-2,1,0,-1)$ is overtwisted by Theorem~\ref{theorem ot} (c) 
and that of $(-2,1,0,-3)$ is tight by Theorem~\ref{theorem tight} (b). The plumbing $(-2,1,0,-2)$ is not covered by these theorems. However, the conclusion can be derived directly from Theorem~\ref{proposition ot}, as the blue rays of the moment cone span an angle greater than $\pi$ only in the left image. Therefore, the contact structure on the boundary of $(-2,1,0,-2)$ is tight 
(see also Example~\ref{ex:4vertices}).}

Note that most concave boundaries of non-linear plumbings will not admit a global contact toric structure. This is because the underlying 3-manifold is usually neither a lens space nor $T^3$. However, non-linear plumbings will still have many co-dimension zero submanifolds with boundary which \emph{do} admit a contact toric structure. This is why our development of pseudoholomorphic tools in the linear case will prove useful in the non-linear case, even though the global classification results about contact toric manifolds will not apply.

\subsection{Algebraic torsion}

In order to detect contact manifolds which are non-fillable and tight, we need measurements which obstruct overtwistedness and fillability. For this, we turn to algebraic torsion and contact invariants from Floer theory. There are different flavors of algebraic torsion, starting with the version defined by Latschev and Wendl using symplectic field theory~\cite{at}. In the appendix to that article, Hutchings defined a variant using embedded contact homology. Variants in Heegaard Floer homology were explored by Kutluhan, Mati\'c, Van Horn-Morris, and Wand in~\cite{KMVHMW1,KMVHMW2}. Roughly speaking, the algebraic torsion vanishes on overtwisted contact manifolds, and takes the value $\infty$ on fillable contact manifolds, so finite non-zero values of algebraic torsion are meant to measure the interesting case of tight and non-fillable contact manifolds. Because algebraic torsion is defined through Floer theories, it is generally quite involved to compute it in examples.

In this article, we focus on the version of algebraic torsion from embedded contact homology (ECH). Our long term goal is to be able to compute the ECH algebraic torsion for the contact boundary of any concave plumbing. We hope to detect tightness and interesting levels of non-fillability in many of these examples, particularly those motivated by singular complex curves.  Since the Floer theoretic pieces of the algebraic torsion computation are so intricate and as a result, left largely unexplored, we start by building a toolkit which applies to the concave plumbing setting. 

As a first application of this toolkit, we show that we can explicitly calculate algebraic torsion and vanishing/non-vanishing of the ECH contact invariant for concave boundaries of linear plumbings, which we know by Theorem~\ref{theorem general} are contact toric $3$-manifolds with non-free action. Below, $c(\cdot)$ denotes the ECH contact invariant and $\at,\at_{\simp}$ are two flavors of the ECH algebraic torsion. For precise definitions see \S~\ref{subsec:ECH_contact_inv} and \ref{subsec:AT_in_ECH}.

\begin{T:at} 
 Let $(Y^3,\xi)$ be a compact connected contact toric manifold characterized by two real numbers $t_1,t_2$, which define the corresponding moment cone as in Remark~\ref{rem:cone}.  
 \begin{enumerate}[label=(\alph*)]
  \item   If $t_2-t_1< \pi$ then
$c(\xi) \neq 0$ and $\at_{\simp}(Y, \lambda, J)= \infty$;
  
     \item if  $t_2-t_1> \pi$ then $c(\xi)=0$ and $\at(Y, \lambda,J)=0$;   

  \item if $t_2-t_1= \pi$ then
 $\at_{\simp}(Y, \lambda, J)>0$,
 \end{enumerate} 
    for all ECH data $(\lambda, J).$ 
\end{T:at}

The results are what we expect, based on properties of algebraic torsion and Theorem~\ref{proposition ot} (together with results of Honda~\cite{Honda} implying that all tight contact structures on lens spaces are fillable). The value of this theorem is not the result, but rather the methods used to prove it. To do this, we calculate information about the closed Reeb orbits and pseudoholomorphic curves which contribute towards the calculations of algebraic torsion.

\begin{remark}
We note that there are no explicit computations in the literature of ECH contact invariants or algebraic torsion measurements for lens spaces with contact structures. 
Prior computations of ECH have been mostly focused on different toric contact forms on $S^3$, with limited other examples, as described below.  Non-toric computations of ECH have been carried out for prequantization bundles of negative Euler number in \cite{NelsonWeiler, chenECH} and Seifert $T(p,q)$ fibrations of $S^3$ in \cite{NelsonWeiler2, NelsonWeiler3}.

Toric combinatorial ECH chain complexes, which are distinct from our methods, have been described in other settings, including some tight lens spaces $L(n,1)$, cf. \cite{intoconcave, ramosferreira, trejos2025echcapacitiesconcavesingular, ferreiraRamosVicente}.  These draw inspiration from the one first described for $T^3$ in \cite{HutchingsSullivanT3}, and utilize the Morse-Bott means to compute it \cite{yao2022cascades, Yao} as well as an algorithm previously given to compute toric combinatorial chain complexes for toric manifolds \cite{Choi}. 

However, the combinatorial toric differential described by other authors has only partially been proven to equal the ECH differential.  In particular, the combinatorial ``rounding corners" differential does agree with the ECH differential on  $T^2\times[0,1]$.  When collapsing circles in $T^2\times\{0\}$ and $T^2\times\{1\}$ via the contact cut procedure described in \cite{Lerman_contact_cut}, there is no proof as of yet in the literature that the combinatorial rounding corners differential agrees with the ECH differential when the Reeb currents  include the elliptic orbits, which arise as the images of the two collapsed tori. (For example, if the manifold is $S^3$, then when it is realized by way of being embedded in $\C^2$, these orbits are the Hopf link consisting of the intersections of $S^3$ with the complex axes.) 
\end{remark}

We now review the general strategy and some of the intermediate results in proving Theorem~\ref{theorem at}. First, we consider the Reeb dynamics associated to a contact form which is invariant under the toric action. For this contact form, closed Reeb orbits come in Morse-Bott families which foliate torus fibers. Then, we determine closed Reeb orbits of a perturbed contact form whose short periodic orbits are nondegenerate. After the perturbation, closed Reeb orbits under a specified action bound appear in pairs in torus fibers, where one orbit is elliptic and the other is hyperbolic. See \textsection\ref{subsec:nondeg_contact_form}.

The ECH chain complex is generated by collections of embedded Reeb orbits with multiplicities, and the differential counts  pseudoholomorphic curves of ECH and Fredholm index one in the symplectization that are positively and negatively asymptotic to closed Reeb orbits. The contact invariant is the homology class represented by the empty collection of Reeb orbits, $c(\xi):=[\emptyset]$. The order of algebraic torsion roughly measures a certain topological index of pseudoholomorphic curves which contribute towards ``killing'' the {empty set} $\emptyset$ (see \textsection\ref{s:ECHbackground}). To compute the algebraic torsion, we are particularly interested in pseudoholomorphic curves with no negative ends, because these curves contribute to differentials with the $\emptyset$ in their image.

Our first result about pseudoholomorphic curves which contributes towards our calculation of the algebraic torsion for contact toric $3$-manifolds is a construction. In what follows, $\lambda$ denotes a toric contact form satisfying certain assumptions stated in Setup~\ref{setup:construction}.

\begin{T:construction} 
    Let $(N,\lambda)$ be a contact toric manifold  with exactly one singular fiber and boundary a torus, satisfying the assumptions of Setup~\ref{setup:construction}. {There exists  a $\lambda$-compatible almost complex structure $J_0$ on $\R\times N$, such that for each Reeb orbit $\gamma_c$ in {the} {torus fiber} $T_{x_0}$ there is {an embedded} $J_0$-holomorphic plane $u_0:\C\to \R\times N$ which is positively asymptotic  to $\gamma_c$}. Additionally, there exist for any $L>1$:
    \begin{itemize}
        \item an $L$-nondegenerate perturbation $\lambdapert$ of $\lambda$,
        \item a closed positive hyperbolic Reeb orbit $\gamma_h$ of $\lambdapert$ in $T_{x_0}$,
        \item and a regular $\Jpert$-holomorphic plane $\upert$ positively asymptotic to $\gamma_h$ with $\ind(\upert)=I(\upert) = 1$ and $\Jplus(\upert) = 0$ for any $\lambdapert$-compatible almost complex structure $\Jpert$ on $\R\times N$ sufficiently close to $J_0$.
    \end{itemize}
\end{T:construction}

Here $\ind$ stands for the Fredholm index, $I$ is the ECH index and $\Jplus$ is a topological index defined in \S\ref{s:ECHbackground}. Note that $N$ is a solid torus, and the key property of the toric contact form $\lambda$ from Setup~\ref{setup:construction}  is that there exists a torus fiber $T_{x_0}$ in $N$ with Reeb slope equal to the meridian of the solid torus (i.e. the Reeb slope bounds a disk in $N$).

Theorem~\ref{theorem:construction} tells us that the empty set appears in $\partial \gamma_h$. If we can show that nothing else appears in $\partial \gamma_h$, we will be able to calculate algebraic torsion. We do this in the case that $t_2-t_1>\pi$ through two different types of constraints on pseudoholomorphic curves.

The first obstruction uses a positivity of intersection argument to confine all $\Jpert$-holomorphic curves whose positive ends are exactly $\gamma_h$ from reaching the second singular fiber, and prevents them from having any negative asymptotic ends. This forces them to lie in the same relative homology class as our constructed $\Jpert$-holomorphic plane obtained in Theorem~\ref{theorem:construction}. Additionally, we use an action argument to ensure that the positive asymptotic end of such a $\Jpert$-holomorphic curve does not lie ``to the right'' of the torus fiber $T_{x_0}$.

\begin{T:noright}
Let $\lambda$ be a toric contact form satisfying Setup~\ref{setup:morethanpi}, $\lambdapert$ its $L$-nondegenerate perturbation satisfying Setup~\ref{setup:deltaperturbation}, $\Jpert$  any $\lambdapert$-compatible almost complex structure, and $u$ a connected $\Jpert$-holomorphic curve, with exactly one positive asymptotic end which is at the closed positive hyperbolic Reeb orbit $\gamma_h$ in the torus fiber $T_{x_0}$. 
Then $u$ has no negative ends, $[u]=E\in H_2(Y,\gamma_h)$, $u$ has Fredholm index $1$, and $u$ does not approach $\gamma_h$ from the right.
\end{T:noright}

The notions of \emph{from the left} and \emph{from the right} mean that the asymptotic end lies in one component of $\R\times (Y\setminus T_{x_0})$. Note that Setup~\ref{setup:morethanpi} adds a key condition on the invariant contact form which utilizes the property that $t_2-t_1>\pi$ (and cannot be satisfied when $t_2-t_1\leq \pi$). The relative homology class $E$ is the unique relative homology class supported in the left half of $Y$ (away from the second singular fiber) which has boundary on $\gamma_h$ with multiplicity $1$.

To complete the count of $\Jpert$-holomorphic curves contributing to $\partial \gamma_h$, we analyze the positive asymptotic end.

\begin{T:asymptotics}
    For sufficiently small positive $\varepsilon$, let $\lambdapert$ be the perturbed contact form constructed in \S\ref{ss:perturb} and let $\Jpert$  be a compatible almost complex structure on $\R\times Y$. Let $\gamma_h$ be a closed positive hyperbolic Reeb orbit lying in the perturbation of a positive Morse-Bott torus at a center of the perturbation, and $\tau_0$ the trivialization from (\ref{eqn:tau0}).  
    
    Then any Fredholm index one $\Jpert$-holomorphic curve $u$ with a positive asymptotic end at $\gamma_h$ approaches from the left or from the right. 
    
    Moreover, if two Fredholm index one, $\Jpert$-holomorphic curves have $Q_{\tau_0}(u_1,u_2)=0$, no negative ends, the positive end is exactly $\gamma_h$, and $u_1$ and $u_2$ approach $\gamma_h$ from the same side then $u_1=u_2$ up to $\R$-translation.
\end{T:asymptotics}

Theorem~\ref{thm:no curve from right} ruled out curves approaching from the right, so we complete the argument by deducing that the curve $\upert$ from Theorem~\ref{theorem:construction} is the unique curve from the left. We are able to compute that $Q_{\tau_0}(u',\upert)=0$ for any $\Jpert$-holomorphic curve $u'$ positively asymptotic to $\gamma_h$ because we know the relative homology class of $u'$ based on Theorem~\ref{thm:no curve from right}. Thus the final part of Theorem~\ref{theorem:asymptotics} completes the count of curves and allows us to deduce Theorem~\ref{theorem at}\ref{case more pi}.

The argument for part~\ref{case equal pi} of Theorem~\ref{theorem at} uses similar tools, while part~\ref{case less pi} is a straightforward argument which does not require analysis of pseudoholomorphic curves (see \textsection\ref{section computing}).

\subsection{Organization of the paper}

Section~\ref{s:preliminaries} provides a review of elementary background on plumbings and symplectic toric geometry. Section~\ref{sec:contact_toric_manifolds} gives the fundamentals and background of contact toric geometry and proves Theorem~\ref{proposition ot}. Section~\ref{section concave toric} shows how to construct a global concave toric structure on linear plumbings with at least one non-negative Euler number (Theorem~\ref{theorem general}). Section~\ref{section tight-ot} determines which plumbings concavely induce tight or overtwisted contact boundary (Theorems~\ref{theorem ot} and~\ref{theorem tight}). Then we move to the ECH side of the article. Section~\ref{s:ECHbackground} gives an overview of the definition of ECH and algebraic torsion, as well as some fundamental properties. Section~\ref{section computing} gives the statement outline of the proof of the computations of algebraic torsion for contact toric $3$-manifolds. Sections~\ref{section plane exists},~\ref{section constraints}, and~\ref{sec:asymptotic}  prove the technical results that go into this proof outline.

\subsection*{Acknowledgements}

This project was initiated through the Women in Geometry conference in 2023 at the Banff International Research Station. The project was further supported by the Summer Collaborators program at the Institute for Advanced Study. Our work would not have been possible without these two opportunities to collaborate in person. We thank Michael Hutchings for helpful conversations and correspondence and thank the referee for their careful reading of our paper and their helpful suggestions.

A.M. is partially supported by the Ministry of Education, Science and Technological Development, Republic of Serbia, through the projects 451-03-66/2024-03/200104 and 451-03-136/2025-03/200104.

J.N. is partially supported by National Science Foundation grants DMS-2104411 and CAREER DMS-2142694. 

L.S. was supported by NSF CAREER grant DMS-2042345, and a Sloan Fellowship FG-2021-16254.

 S.T. was partially supported by a grant from the Institute for Advanced Study School of Mathematics,  the Schmidt Futures program,  a research grant from the Center for New Scientists at the Weizmann Institute of Science and Alon fellowship.

L.W. acknowledges support from NSF Grant DMS-2303437 and IAS Giorgio and Elena Petronio Fellow II Fund.

\section{Preliminaries on plumbings and symplectic toric geometry} \label{s:preliminaries}
In this section, we briefly review some topological and symplectic toric constructions, setting up our notation along the way.  

\subsection{Plumbings} We start by reviewing the plumbing construction and some previously understood interactions with symplectic structures.

A plumbing graph is a finite connected graph whose vertices, numbered from 1 to $n$, are assigned two numbers $(g_i,s_i)$ and have valency $d_i$, and whose edges are assigned a sign.  A plumbing graph defines a closed 4-manifold that is obtained in the following way: 
\begin{itemize}
    \item Each vertex corresponds to a disc bundle over a surface of genus $g_i$ with self-intersection number $s_i$.
\end{itemize}
In our case, the base surface is always a sphere so $g_i=0$ for every $i$ and thus vertices are assigned only the self-intersection number $s_i$. Two disc bundles are plumbed together if the corresponding vertices are joined by an edge. 

To plumb two disc bundles together, first choose a small disc in each base $D_1$ and $D_2$. The union of fibers over each $D_i$ is bundle isomorphic to $D_i\times D^2$. To plumb the disc bundles, we glue $D_1\times D^2$ to $D_2\times D^2$ by a map which preserves the product structure but switches the two factors. This gluing must be either orientation preserving or reversing, depending on a sign labeling the edge in the plumbing graph. When the graph is a tree, the result is independent of this choice. In this article we will not specify signs on edges, and will always assume the positive (orientation preserving) case.

  Associated to a (general) plumbing graph $\Gamma$ with $n$ vertices there is an $n\times n$ intersection matrix defined by $Q_\Gamma=[a_{ij}]$ such that
  \begin{itemize}
      \item   $a_{ii}=s_i$ for every $1\leq i\leq n$;
  \item $a_{ij}=0$ if there is no edge between the vertices $i$ and $j$ and $i\neq j$;
  \item  $a_{ij}=\pm 1$ if there is an edge between the vertices $i$ and $j$, $ i \neq j$, with an orientation $\pm$ assigned. 
\end{itemize}

A plumbing graph is linear if it is connected, it has exactly two vertices of valency 1 and all remaining vertices of valency 2. The two vertices of valency 1, are called end-vertices.
 Such a graph is of the form:
  \begin{center} \includegraphics[scale=.7]{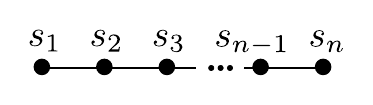} \end{center} 
Note, we have ommitted the label $g_i$ since we have assumed $g_i=0$, and we choose to order the indices of the vertices such that consecutive indices correspond to adjacent vertices in the graph.

For a linear plumbing graph, we have
\begin{equation*}
    Q_\Gamma =\begin{pmatrix} s_1 & 1& 0 & 0 & \cdots &0\\
    1 & s_2& 1 & 0 &\cdots& 0\\ 
    0 & 1 & s_3 &1 & \cdots & 0\\
    && \ddots&&&\\
    0 & 0& \cdots &0 & 1 & s_n
    \end{pmatrix}.
\end{equation*}

There are natural symplectic structures on general plumbings with positive edges, obtained by choosing a symplectic structure on each disk bundle such that the zero section is a symplectic surface, and the disk bundle is its standard small symplectic neighborhood. Then assuming the symplectic areas of $D_1$ and $D_2$ are the same as the symplectic area of the $D^2$ fibers, the plumbing gluing is a symplectomorphism. To a linear plumbing graph with $n$ vertices, we can associate the vector of the symplectic areas of each zero section $a=(a_1,\dots, a_n)\in \R_{>0}^n$.

As plumbings have a non-empty boundary it is natural to explore if the boundary admits a contact type structure. We recall the definition.

\begin{definition}
Let $(W,\omega)$ be a compact symplectic 4-manifold with boundary $Y=\partial W$, then:
\begin{enumerate}
    \item \label{i:cttype} $(W, \omega)$ has contact type boundary if there is a primitive $\tilde{\lambda}$ for $\omega$ ($d\tilde{\lambda}=\omega$) defined in a neighborhood of $\partial W$ in $W$, such that the restriction, $\lambda$, of $\tilde{\lambda}$ to $\partial W$ is a contact form.
    \item The contact boundary $(Y,\lambda)$ is \emph{convex} if the orientation induced on $Y$ by the non-vanishing form $\lambda\wedge d\lambda$ coincides with the orientation of $Y$ as boundary of $(W,\omega)$. Otherwise, $(Y,\lambda)$ is a \emph{concave} contact boundary.
\end{enumerate}
\end{definition}

Condition (\ref{i:cttype}) is equivalent to the existence of the Liouville vector field $V$ defined near the boundary  that is transversal to the boundary. When $V$ points out of (respectively, into) the boundary, the contact structure is convex
(respectively, concave).

\begin{remark}\label{rem:GS}
    
According to Gay-Stipsicz \cite{gs}, if  $Q_{\Gamma}$ is negative-definite, then a symplectic structure on the corresponding plumbing admits a Liouville vector field pointing transversaly out of the boundary, making the boundary of convex contact type. Using the same technique, Li-Mak 
\cite{lm} show
 a symplectic plumbing admits a concave Liouville structure if it satisfies the ``negative GS''  criterion: 
 \begin{itemize}
     \item  There exists $z\in \R_{<0}^n$ such that $-Q_{\Gamma} z = a$, where
$a=(a_1,\dots, a_n)\in \R_{>0}^n$ are the symplectic areas of the base spheres.
 \end{itemize}
 \end{remark}

\subsection{Symplectic  toric manifolds}\label{subsec:toric_intro}
We now give a short survey on the properties of symplectic toric manifolds that will be significantly used in this article. For more details we refer to Cannas da Silva's book \cite{ana_cannas}.

A symplectic toric manifold is a symplectic manifold $(W^{2n},\omega)$  equipped with an effective Hamiltonian action of the torus $T^n=\mathbb R^n/\mathbb{Z}^n$.
An associated  moment map  $\mu=(\mu_1,\ldots,\mu_n): W\to\mathbb R^{n}$ is defined, up to translation,
by 
$$\iota_{X_k}\omega=-d\mu_k, \qquad k=1,\ldots,n,$$
where the vector fields $X_k,$ $k=1,\ldots,n,$ are the generators of the action.
Beside translations, $ \pm SL(n,\mathbb{Z})$ transformations of the moment map image $\mu(W)$ do not change the symplectic toric structure. Indeed, any $G\in  \pm SL(n,\mathbb{Z})$ transformation of $\mu(W)\subset \mathbb{R}^n$ is induced by a reparametrization 
of the torus action.

\vskip1mm
   \begin{figure}\centering
\includegraphics[width=11cm]{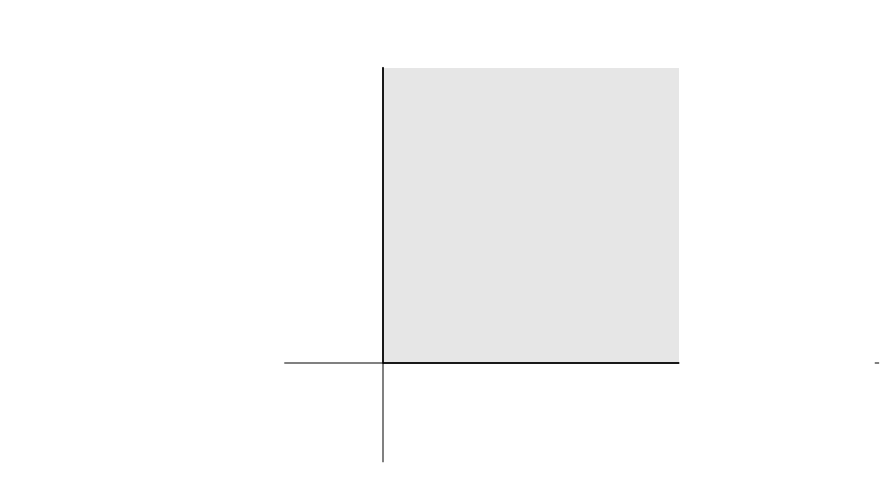}
\caption{The standard moment cone as the moment map image of $(\mathbb{C}^2,\omega_{st})$.
}\label{standard moment cone}
   \end{figure}

An equivariant version of the Darboux theorem states that a small neighborhood of a fixed point in a symplectic toric manifold $W$ is modelled by a toric complex space $(\mathbb{C}^n,\omega_{st}=\frac{i}{2}\sum_{k=1}^ndz_k\wedge d\bar{z}_k)$, with the toric action generated by the Hamiltonian vector fields 
$X_k=2\pi i(z_k\frac{\partial}{\partial z_k}-\bar{z}_k\frac{\partial}{\partial \bar{z}_k})$, for $k=1,\ldots,n$. The corresponding moment map is given by
\begin{equation}\label{standard H}
\mu(z_1,\ldots, z_n)=\pi(|z_1|^2,\ldots,|z_n|^2)+c
\end{equation}
and the moment map image is the standard cone (see Figure~\ref{standard moment cone}). Therefore, a moment map image of  a small neighborhood of a vertex can be mapped by some $SL(n,\mathbb{Z})$ transformation to the standard cone. 
 In particular,  fixed points in a symplectic toric manifold   correspond to vertices of a moment map image.

According to Atiyah–Guillemin–Sternberg convexity theorem
(\cite{Atiyah, GuiSte}), a moment map image of a closed symplectic toric manifold is a convex polytope and a pre-image of any point in the moment polytope is connected, i.e. one orbit. Moreover, a moment polytope is a Delzant polytope 
(\cite{Delzant}), that is, simple (there are $n$ edges emanating from any vertex), rational (inward normal vectors of the facets of the moment polytope belong to $\mathbb Z^n$) and smooth (inward normal  vectors of the facets meeting in a vertex form a $\mathbb Z^n$-basis).

In general, any $d$-dimensional orbit is an isotropic torus $T^d$ and it corresponds to a point in a $d$-dimensional face of a moment map image. If $v_1,\ldots, v_{n-d}\in\mathbb Z^n$ are linearly independent primitive inward normal vectors to that face, then
the corresponding $d$-dimensional orbit is obtained  by collapsing $T^n$ along the circles of slopes $v_j\in H_1(T^n,\mathbb Z),$ $j=1,\ldots, n-d.$
The set of full dimensional (Lagrangian) orbits  is an  open dense subset of $W$ and it is equivariantly symplectomorphic to   $(T^n\times U, \sum_{k=1}^ndp_k\wedge dq_k)$ with  a moment map $\mu(q_1,\ldots,q_n,p_1,\ldots,p_n)=(p_1,\ldots,,p_n)$, where $U\subset\mathbb R^n$ is an open set diffeomorphic to an open disc. That is, 
Hamiltonian functions make a completely integrable system with global action-angle coordinates $(q_1,\ldots,q_n,p_1,\ldots,p_n)\in T^n\times U.$

\subsubsection{\bf Symplectic toric 4-manifolds}\label{section: symplectic 4}
We now focus on some properties of symplectic toric 4-manifolds and their moment map images. 

A finite edge in the moment map image corresponds to an invariant symplectic sphere that is a collection of isotropic circle orbits and two fixed points of the toric action. If $w \in  \mathbb Z^2$
denotes an inward normal vector to that edge, then these circle orbits are obtained by collapsing the tori orbits along the circle of the slope 
$w \in H_1(T^2, \mathbb Z^2).$ 
Further, a symplectic area and a self-intersection number of such a sphere  can be read off from the moment map image in the following way (for details, see \cite[\S 7.2,7.3]{sym42}):
\begin{itemize}
    \item The symplectic area of the symplectic sphere corresponding to the edge with end points $V_1$ and $V_2$ is equal to the affine length of the interval $(V_1,V_2)$, namely it is equal to 
    $$
    |a|  \in \mathbb R_{>0},\quad\text{ where  }\quad V_2-V_1=a\cdot e,\quad\text{ for some primitive vector }e \in  \mathbb Z^2.$$ 
    (see Figure~\ref{self-int}, where $e=(u_2+v_2,-(u_1+v_1))$). 
\item The self-intersection number of this symplectic sphere is equal to the determinant of the inward normal vectors of the edges that are adjacent to the given edge, taken in the clock-wise direction.
\end{itemize}
In particular, $SL(2,  \mathbb Z)$ transformations of a moment map image preserve symplectic areas and self-intersection numbers.

\begin{figure}
\centering
\includegraphics[width=6cm]{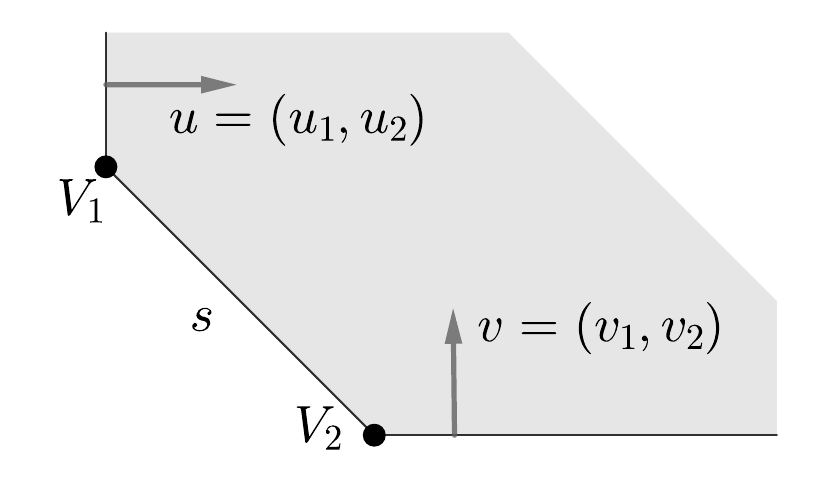}
\caption{A moment map image of a neighborhood of a symplectic sphere with self-intersection number $s=v_1u_2-v_2u_1$.}
\label{self-int}
\end{figure}

A moment map image of any closed symplectic toric 4-manifold is a 2-dimensional convex polytope where  primitive inward normal vectors of any two edges meeting in a vertex form a $ \mathbb Z^2$ basis (Delzant condition).  When the polytopes have exactly 4 edges they are called Hirzebruch trapezoids (see Figure~\ref{trapezoid}), as they are precisely moment map images of Hirzebruch surfaces, $\mathbb{CP}^2$ blown up once and $S^2 \times S^2$.
In \textsection\ref{section tight-ot}, we make use of Hirzebruch trapezoids in order to construct certain linear plumbings with a tight concave boundary.

\begin{figure}
\centering
\includegraphics[width=7cm]{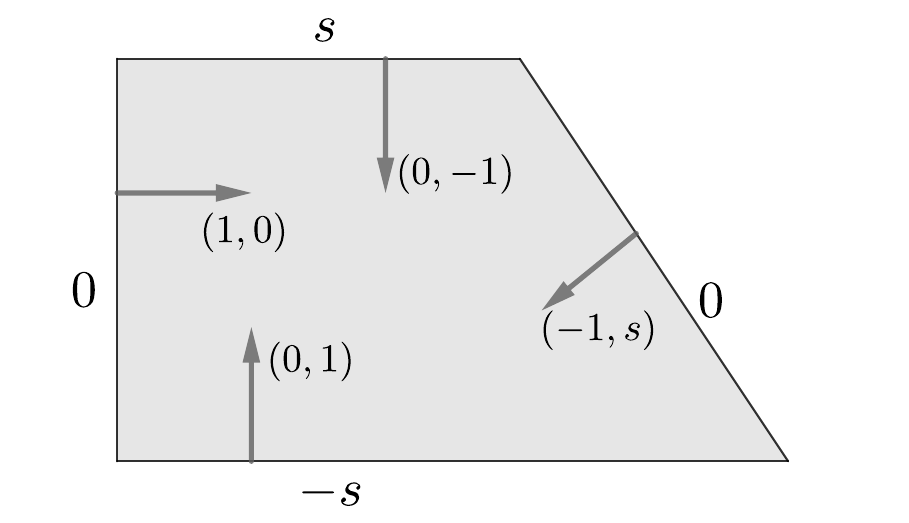}
\caption{A Hirzebruch trapezoid}
\label{trapezoid}
\end{figure}

A symplectic blow up at a fixed point of a symplectic toric 4-manifold $(W,  \omega)$ is a new symplectic toric 4-manifold. The corresponding moment map image is obtained by chopping off the corner of a moment map image of $(W, \omega)$ as shown in Figure~\ref{blow up}. The new edge obtained in this way is the moment map image  of the exceptional divisor of a blow up and has  self-intersection number $-1$. Self-intersection numbers of the spheres corresponding to the  edges that form the corner  lower by 1, after blow up.

    \begin{figure}
\centering
\includegraphics[width=15cm]{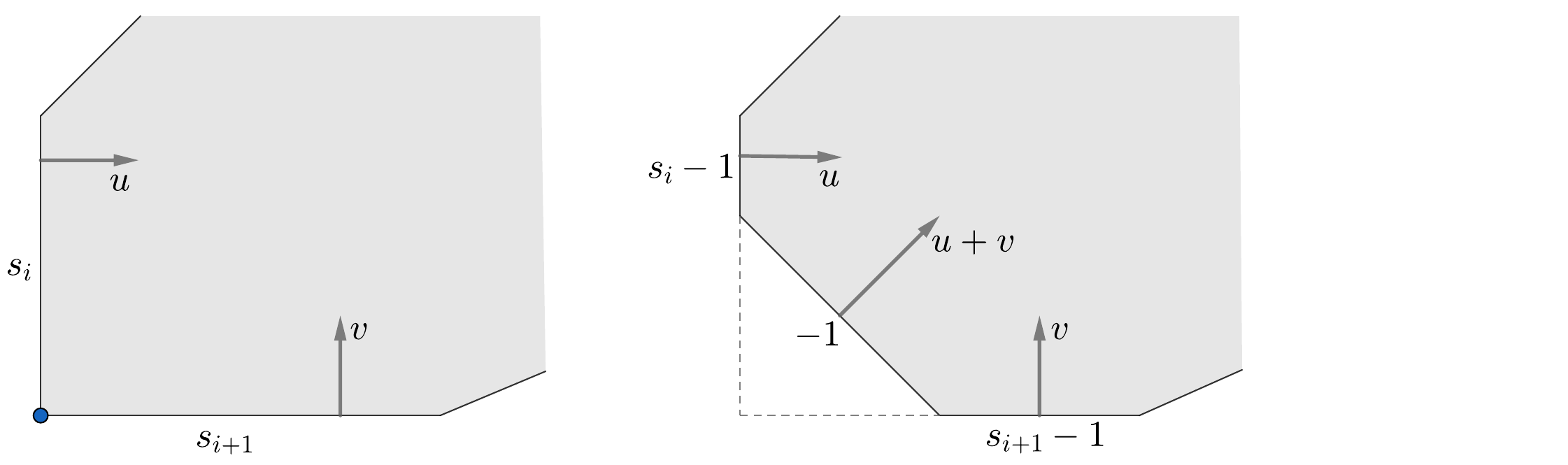}
\caption{A moment map image before and after a blow up of the fixed point, with the changes in  the self-intersection numbers. }
\label{blow up}
\end{figure}

\section{Contact toric manifolds}\label{sec:contact_toric_manifolds}
In this section we focus on contact manifolds equipped with a toric action. To any such manifold one can associate a moment cone that can be very useful for
 reading topological properties of the manifold as well as properties of the contact structure.
 In Theorem~\ref{proposition ot}, we will see that it is sufficient to know an angle related to the 2-dimensional moment  cone to determine whether the contact structure on a 3-manifold is overtwisted or tight. This criterion is then used in \S\ref{section tight-ot} to explore when the concave boundary of the linear plumbing is overtwisted or tight. Moreover, the contact toric structure will be especially helpful to understand the Reeb dynamics and used significantly 
 in our construction of a pseudoholomorphic plane in \S\ref{section plane exists}.  These are then utilized to compute the contact invariant and algebraic torsion in ECH in  \S\ref{section computing}.

 For more details associated to contact toric manifolds, we refer the reader to \cite{Lerman_contact_toric}.  We start from a symplectic toric manifold
$(W^{2n},\omega)$ that has a non-empty connected boundary $Y$. 
Let $V$ be a Liouville vector field defined near  $Y$  and  transverse to the boundary. The vector field $V$ induces a contact structure $\xi=\ker \iota_V\omega$ on  $Y$. As already mentioned,
if $V$ points out of (respectively, into) $W$ then $(Y,\xi)$ is called a convex (respectively, concave) contact boundary.

\begin{lemma}\label{lemma toric contact} 
A Hamiltonian action of a torus $T^n$ on a symplectic manifold $(W^{2n},\omega)$ with a convex (resp., concave) contact boundary $(Y,\xi)$ preserves the contact structure $\xi$.

\end{lemma}
\begin{proof}
A toric  action on $W$ is an effective action by diffeomorphisms, thus, it induces a well defined effective $T^n$ action by diffeomorphisms on the boundary $Y$. 
Let $V$ be a Liouville vector field defined near $Y.$ Then, the vector field defined by  the average
$$V_{\text{inv}}=\int_{\theta\in T^n}(\theta^*V)d\theta,$$
is a Liouville vector field invariant under the toric action on $W$. Therefore, the induced contact form
$\lambda_{\text{inv}}=\iota_{V_{\text{inv}}}\omega$ on the boundary $Y$ is also invariant under the toric action.
\end{proof}

A  contact manifold $(Y^{2n-1},\xi)$  equipped with an effective  $T^n$ action that preserves the contact structure is called a \emph{contact toric manifold}. 
Due to Lemma~\ref{lemma toric contact}, a convex or a concave contact boundary of a symplectic toric manifold is a contact toric manifold.

\begin{example}

\begin{itemize}
\item[(a)]  Consider the case $n=2$.  A toric domain $X_{\Omega}$ in $\mathbb C^2$ is defined as a preimage under the moment map $\mu$ given by (\ref{standard H}), for $c=0$, of a domain
$\Omega\subset(\mathbb R_{\geq0})^2$. Therefore, a toric domain $X_{\Omega}$ inherits a symplectic structure and a toric action from $(\mathbb C^2, \omega_{st})$. If the radial vector field 
$V=\frac{1}{2}\sum_{k=1}^2(z_k\frac{\partial}{\partial z_k}+\bar{z}_k\frac{\partial}{\partial \bar{z}_k})$  is  transverse to the boundary $\partial X_{\Omega}$, then $X_{\Omega}$ is a symplectic toric manifold with a convex contact boundary.

\item[(b)] Linear plumbings $(s_1, \ldots, s_n)$ where $s_i\geq0,$ for some $i\in\{1, \ldots, n\}$, admit a structure of symplectic toric manifolds with a concave contact  boundary, as established in \textsection\ref{section concave toric}.
\end{itemize}
\end{example}

If  $(Y^{2n-1}, \xi)$ is a contact boundary of a symplectic toric manifold $(W^{2n},\omega)$  and if $\lambda=\iota_V \omega$ is an invariant contact form, then, according to Cartan's formula, it holds
$$0=\mathcal{L}_{X_k}\lambda=d(\lambda(X_k))+\iota_{X_k}d\lambda,$$
where $X_k,$ $k=1,\ldots, n,$ are the generators of the toric action.
Since $d\lambda=\omega$, the function $\lambda(X_k)$ defined on  $Y$ is  a restriction of the unique moment map function defined on $W$.
Therefore, to every contact toric manifold $(Y^{2n-1}, \xi)$ with an invariant contact form $\lambda$ we associate  a contact moment map  $\mu_{\lambda}=(\mu_1,\ldots,\mu_n): Y\to\mathbb R^{n}$  uniquely
defined by 
$$\mu_k=\lambda(X_k) \qquad \mbox{for} \qquad k=1,\ldots,n.$$
In particular, along the set of full dimensional orbits  where $W$ admits global action-angle coordinates, one can choose $\lambda=\sum_{k=1}^np_kdq_k.$ 

\vskip2mm

The natural lift of the toric action on $(Y^{2n-1}, \xi=\ker\lambda)$  to the symplectization $( \R \times Y, d(e^r\lambda))$ is also a toric action that makes the symplectization  a symplectic toric manifold with a moment map  $\mu_{ \xi}=e^r\mu_{ \lambda}.$ 
Note that the origin is never in the moment map image $\mu_{ \xi}(\R \times Y)$ by \cite[Lemma 2.12]{Lerman_contact_toric}. 
The \emph{moment cone} of a contact toric manifold is defined as the moment map image $\mu_{ \xi}(\R \times Y)$ together with the origin. 

For example,  
the moment cone  of the standard contact sphere $(S^{2n-1}, \ker(\lambda_{st}=\frac{i}{4}\sum_{k=1}^n(z_kd\bar{z}_k-\bar z_kdz_k))$ with respect to 
the Hamiltonian vector fields 
$X_k=2\pi i(z_k\frac{\partial}{\partial z_k}-\bar{z}_k\frac{\partial}{\partial \bar{z}_k})$, for $k=1,\ldots,n$,
is the standard cone  (Figure~\ref{standard moment cone}). 

Note that every full dimensional orbit in a contact toric manifold is a pre-Lagrangian submanifold, that is, it is a diffeomorphic image of a Lagrangian submanifold in the symplectisation.

\begin{remark}

While a contact moment map $\mu_{\lambda}$ depends on the choice of an invariant contact form $ \lambda$, the moment cone depends only on the contact structure. Moreover, translating from the symplectic toric case, any $SL(n,\mathbb Z)$ transformation of the moment cone does not change the corresponding contact toric structure.
Note also that any $d$-dimensional orbit in $Y$ lifts to a line of $d$-dimensional orbits in $ \R \times Y.$ Because fixed points are zero-dimensional orbits and are isolated in a symplectic toric manifold, there are no fixed points in a contact toric manifold. 
    
\end{remark}

\begin{remark}\label{remark curve}
The moment cone of a contact toric manifold $(Y, \xi=\ker\lambda)$ can equivalently be defined as a cone over a hypersurface $\mu_{\lambda}(Y)$, for any invariant contact form $\lambda.$ 
Moreover, any embedded hypersurface  in the moment cone, disjoint from the origin,  that intersects  the facets of the cone along the boundary and that is transversal to the radial vector field is a moment map image $\mu_{\lambda}(Y),$ for a particular choice of an invariant contact form $\lambda.$ 
In particular, in dimension 3, any embedded curve in the moment cone, disjoint from the origin, that intersects the rays of the cone in end points and that is transversal to the radial vector field is a moment map image.
Note that the contact structure does not depend on this choice of the hypersurface defining the cone, but the corresponding invariant contact form and associated Reeb dynamics do (see  \S\ref{subsec contact toric}).

\end{remark}

\subsection{Classification of contact toric 3-manifolds with singular orbits}\label{section:lens spaces}

 In this section we restate and explain the classification of compact connected contact toric 3-manifolds with a non-free toric action, done by Lerman in \cite{Lerman_contact_toric}, that will help us to understand the contact boundary of a linear plumbing.

We start from one example that, according to the classification Theorem~\ref{theorem Lerman class} presented below, covers all possible cases with convex moment cone.

 \begin{example}\label{example lens tight} 
 The lens space $L(k,l)$, $k\geq0$, where $k$ and $l$ are relatively prime integers, is obtained by gluing two solid tori along their boundaries by a diffeomorphism $S^1\times \partial D^2 \to S^1\times \partial D^2$ that sends the circle represented by $(0,1) \in H_1(T^2,\mathbb{Z})$ (i.e. meridian) to the  circle represented by $(k,l) \in H_1(T^2,\mathbb{Z}).$

\begin{itemize}
\item[(a)] If $k\neq0$ the lens space $L(k,l)$   can equivalently be defined as a quotient space of the unit sphere $S^3$ under the free $\mathbb Z_k$ action generated by the rotation 
$$e^{i 2\pi /k}\ast(z_1,z_2)\to (e^{i 2\pi /k}z_1, e^{i 2\pi l/k}z_2).$$
The contact form $\alpha_{st}$ on $S^3$ is invariant under this action, thus, it induces a contact form $\alpha_{kl}$ on $L(k,l).$ We refer to $\ker\alpha_{kl}$ as a standard tight contact structure on $L(k,l).$

The following reparametrization of the standard toric action  on $(S^3,\ker\lambda_{st})$
$$(e^{i2\pi t_1},e^{i2\pi t_2})*(z_1,z_2)\mapsto(e^{i2\pi t_1}z_1,e^{i2\pi  l t_1}e^{i2\pi  t_2}z_2)$$
commutes with the given $\mathbb Z_k$ action on $S^3$ and, therefore, induces a $T^2$ action on $L(k,l).$ In order to have an effective action  we divide the first acting circle by $k.$ Then, we obtain a toric action on $(L(k,l), \ker\alpha_{kl})$ whose moment map is given by
$$\mu([z_1,z_2])=\pi\left(\frac{1}{k}(|z_1|^2+l |z_2|^2), |z_2|^2\right).$$
 The corresponding moment map image and the moment cone are shown in Figure~\ref{convex lens} on the left.
  
 \item[(b)] If $(k,l)=(0,1)$ the corresponding lens space is $S^1_{\theta}\times S^2_{h,z}$,  where $S^2$ is thought at the unit sphere in $\mathbb{R}^3\simeq \mathbb{R}\times \mathbb{C}$. For the standard contact form  $hd\theta+\frac{i}{4}(zd\bar{z}-\bar zdz)$ and the toric action given by
 $$(e^{i2\pi t_1},e^{i2\pi t_2})*(e^{i2\pi \theta},h,z)\mapsto(e^{i2\pi (t_1+\theta)},h,e^{i2\pi   t_2}z)$$
   the moment map is
 $$\mu(\theta, h,z)=(h,\pi|z|^2).$$
 The corresponding moment map image and the moment cone are shown in Figure~\ref{convex lens} on the right.
 \end{itemize}

    \begin{figure}
\centering
\includegraphics[width=14cm]{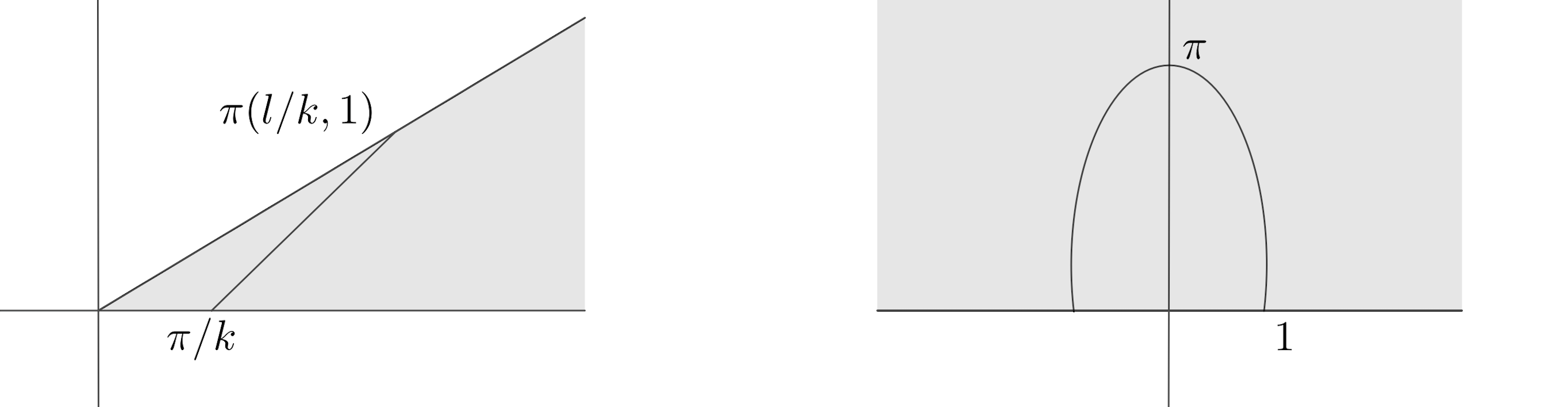}
\caption{A moment cone of the lens space  $L(k,l)$ (left) and of $S^1\times S^2$ (right) with the standard tight contact structure.}
\label{convex lens}
\end{figure}      
\end{example}

\begin{remark}\label{remark Lens}
From the moment cone of a lens space (Figure~\ref{convex lens}) we conclude that a lens space $L(k,l)$ can equivalently be defined as a quotient of $T^2\times[0,1]$ where $T^2\times \{0\}$  is collapsed along the circle of slope $(0,1)$ and $T^2 \times \{1\}$ is collapsed along the circle of slope $(k,-l)$. Namely, $(0,1)$ and $(k,-l)$ are primitive inward normal vectors to the rays that span the moment cone and they define collapsed orbits (see \textsection\ref{subsec:toric_intro}).

\end{remark}
Next we state Lerman's classification result, and  afterwards elucidate the moment cones associated to the lens spaces.

\begin{theorem}{\em \cite[Theorem 2.18.(2)]{Lerman_contact_toric}}\label{theorem Lerman class}
Any  compact connected contact toric manifold $(Y^3,\xi)$ with a non-free toric action can be obtained from the below construction, which is determined by two real numbers $t_1,t_2$ with $0\leq t_1<2\pi,t_1<t_2$ such that 
$\tan t_1$ and $\tan t_2$, when defined, are rational numbers.
\end{theorem}

We briefly describe the compact connected contact toric manifold $(Y^3, \xi)$ determined by the numbers $t_1$ and $t_2$. 
Start from $T^2\times [0,1]$ with coordinates $(q_1, q_2,x)$,  the contact structure 
\[
\xi_{t_1,t_2}:=\ker(\cos (t_1(1-x)+t_2x)dq_1+\sin (t_1(1-x)+t_2x) dq_2)
\]
and the toric action given by the standard rotation of $T^2$ coordinates.
According to Lerman, we  recover the topology of a manifold $Y$ from $T^2\times [0,1]$ by collapsing the  tori $T^2\times \{0\}$ and $T^2\times \{1\}$ along particular circles. As explained in \textsection\ref{subsec:toric_intro}, if $v_0  \in \mathbb{Z}^2$ and $v_1 \in \mathbb{Z}^2$ are primitive inward normal vectors to the given rays of the moment cone, then the collapsing circles are represented by the slopes $v_0 \in H_1(T^2,\mathbb{Z})$ and $v_1 \in H_1(T^2,\mathbb{Z})$  respectively.  After performing a suitable $SL(2,\mathbb{Z})$ transformation of  the given rays (and the moment cone), we may assume  that the first ray is given by the direction $(1,0)$ and the second ray is given by the direction $(l,k).$ 
Then, the  tori $T^2\times \{0\}$ and $T^2\times \{1\}$ are collapsed along the circles of slopes $(0,1)$ and $(k,-l)$, respectively.  Therefore, $Y$ is diffeomorphic to the lens space $L(k,l)$ (see Remark~\ref{remark Lens}). Further, 
 by the method of contact cut, Lerman in \cite{Lerman_contact_cut} shows that $Y$ inherits a contact structure from $T^2\times [0,1]$. Finally, the toric action on $T^2\times [0,1]$ induces a toric action on $Y$ whose corresponding moment cone  is defined precisely by the given rays.
 
 In particular, every real number $x\in(t_1,t_2)$ corresponds to a unique regular $T^2$ orbit in $Y$, while the numbers $t_1$ and $t_2$ correspond to singular (circle) orbits in $Y$.

As a result, we have the following means of determining when the moment cone uniquely determines the contact toric 3-manifold.
\begin{remark}\label{rem:cone}
The numbers $t_1$ and $t_2$ are  the angles (in radians) between the positive part of the $x$-axis and the  rays $R_1$ and $R_2$, respectively (see Figure~\ref{fig: cone}). As 
$\tan t_1$ and $\tan t_2$, when defined, are rational numbers, the slopes of the rays $R_1$ and $R_2$ are always elements of $\mathbb Z^2.$
If $t_2-t_1<2\pi$,  the  moment cone of $(Y,\xi)$ is  spanned by the  rays $R_1$ and $R_2$. If $t_2-t_1\geq 2\pi$,  the moment cone  is the whole space $\mathbb R^2.$ In particular, due to Theorem~\ref{theorem Lerman class} the moment cone uniquely determines the contact toric 3-manifold if and only if $t_2-t_1< 2\pi.  $ 

    \begin{figure}
\centering
\includegraphics[width=16cm]{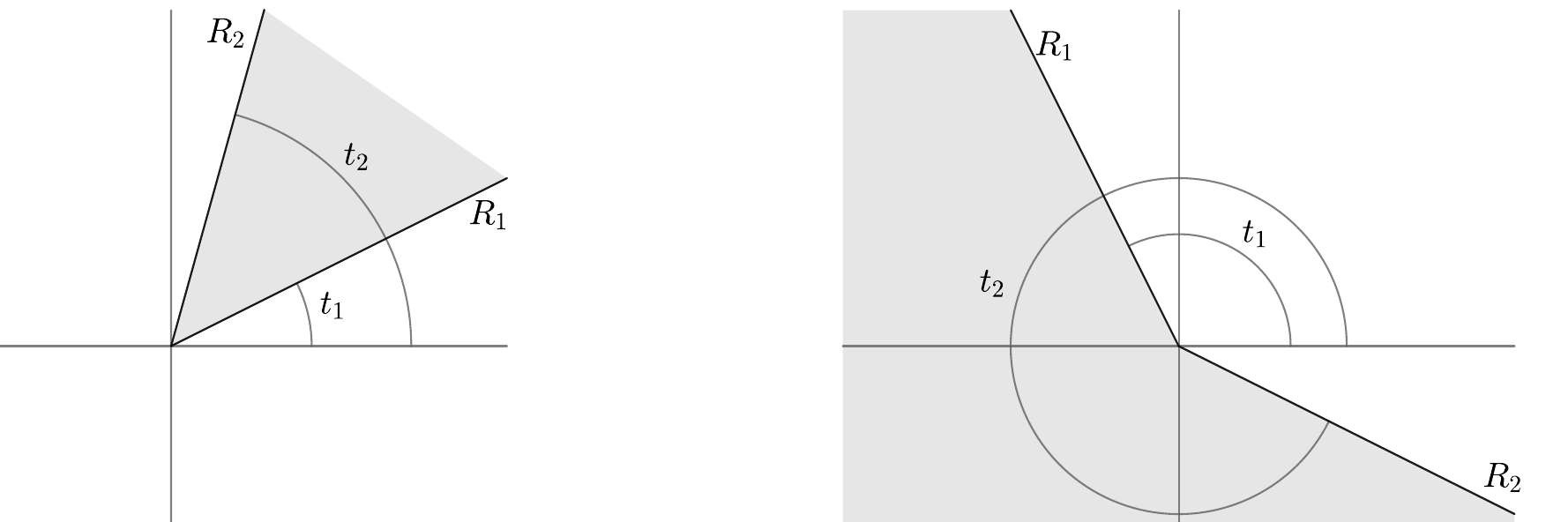}
\caption{Moment cones defined by the angles $t_1$ and $t_2$.}
\label{fig: cone}
\end{figure}

\end{remark}

 \begin{remark}  \label{remark: all cones for L(k,l)}
  From a classical theorem by Reidemeister (\cite{Reid}) it follows that
$$L(k,l)\cong L(k',l')\qquad \textrm{if and only if} \qquad k=k', l'\equiv \pm l^{\pm1} \ (\textrm{mod})\  k .$$
That is, $L(k,l)$ is homeomorphic to $L(k,-l),$ $L(k, l+nk),$ $L(k,-l+nk),$ $L(k,r+nk)$ and $L(k,-r+nk),$ for any $n\in\mathbb Z$ and $r,s\in \mathbb Z$, where $lr-ks=1.$ 

Therefore, according to Remark~\ref{remark Lens} all the  convex and concave cones whose first ray is the positive part of the $x$-axis and the second ray is given by the directions $ (\pm l+nk,  \pm k)$ or $(\pm r+nk,\pm k)$ correspond to $L(k,l).$ 

 \end{remark}

 \begin{remark} To complete the classification of contact toric 3-manifolds we remark that contact toric 3-manifolds with all regular orbits are also classified by Lerman in  \cite{Lerman_contact_toric}. These are  $$(T^3, 
\ker(\cos (2n \pi t)dq_1+\sin (2n \pi t) dq_2), \mbox{ for all }n \in\mathbb N,$$ where the toric action is given by the rotation of $( q_1, q_2)$ coordinates.
Therefore, the corresponding moment cone is always the whole plane $ \mathbb R^2$ and the pre-image of a point in the moment cone is a disjoint union of $n$ regular torus orbits.
 
  \end{remark} 

 \subsection{Tight and overtwisted contact toric 3-manifolds with singular orbits}
 For a contact toric 3-manifold with singular orbits, we now explain how one can read from the 2-dimensional moment cone  if the corresponding contact structure  is tight or overtwisted. 
 We remark that no such phenomena exist in higher dimensions, as all contact toric manifolds in higher dimensions  are at least weakly fillable by Marinkovi\'c \cite{am}. 
 
 \begin{theorem}\label{proposition ot}
 Let $(Y^3,\xi)$ be a compact connected contact manifold with a non-free toric action obtained from Lerman's construction from real numbers $t_1,t_2$. The contact structure $\xi$ is overtwisted if and only if  $t_2-t_1> \pi$.
\end{theorem}

 \begin{proof} 
   If $t_2-t_1\leq \pi$ then according to Theorem~\ref{theorem Lerman class} and Example~\ref{example lens tight}, the 3-manifold $(Y, \xi)$ is contactomorphic to a lens space with the standard tight contact structure. 

 Assume $t_2-t_1> \pi$. After suitable $SL(2,\mathbb{Z})$ transformation of the moment cone, we  assume $t_1=0.$ Consider the annulus $A_c\subset T^2\times [0,1]$ 
 defined by
 $$A_c=\{(c,\theta, x)\in  T^2\times [0,1] \hskip2mm |\hskip2mm \theta\in S^1,  x\in[0,{\pi}/{t_2}]  \},$$
 for any $c\in S^1.$
 
 To obtain $Y$, the boundary torus $T^2\times \{0\}$ is collapsed along the circle of slope $(0,1)$ and hence the circle $\{c\}\times S^1\times \{0\}$ quotients to a point  in $Y$. Therefore $A_c$ quotients to a disc $D_c$ in $Y$. The  boundary of $D_c$ corresponds to 
 $x=\frac{\pi}{t_2}$ 
and the contact structure on $Y$ near the boundary of $D_c$ is
 defined by 
 $\ker(\cos (t_2x)d\theta_1+\sin (t_2x) d\theta_2)$. Since $\theta_1=c$ along $D_c$ it follows that the boundary of $D_c$ 
  is tangent to the contact structure. Therefore, this is an overtwisted disc.

 Alternatively, an overtwisted disc can be constructed in the following way.  Denote by 
$\mu=(\mu_1,\mu_2)$  the moment map  for the toric action on $Y$. Note that $\mu_2$ vanishes along the circle orbit $o_1$  that corresponds to $t_1=0$ and, since $t_2> \pi$, 
$\mu_2$ vanishes along the torus orbit $o_2$  that corresponds to $x=\pi$. 
 Consider a circle subaction of the given toric action on $Y$ defined by the inclusion $x \rightarrow (0,x).$    This is a Hamiltonian circle action with the induced Hamiltonian function  $\mu_2.$
 With respect to this circle subaction the circle $o_1$ is a collection of fixed points,  as it is obtained by collapsing $T^2$  along the circle of slope   $(0,1)$,  
 and the torus $o_2$ is a collection of Legendrian orbits. 
 We now connect one fixed point $F$ in  $o_1$ orbit with any  Legendrian orbit $l$ in $o_2$, through the family of non-Legendrian circle orbits in the following way. Let $\pi: Y \to Y/S^1$ denote a natural projection to the orbit space $Y/S^1$. Take any smooth curve $\gamma:[0,1] \to Y/S^1$ such that $\gamma(0)=\pi(F)$, $\gamma(1)=\pi(l)$ and $\gamma$ does not intersect any other point that is a projection of a Legendrian orbit in $Y$.  Then, the lift of $\gamma$   is an overtwisted disc in $Y$. For more details we refer to \cite[Lemma IV. 19.]{klaus_thesis}.  
 \end{proof}

We conclude our discussions by observing the following classification of contact toric 3-manifolds, including a uniqueness result for overtwisted contact structures.
   
   \begin{remark} \label{remark adding 2pi} Consider contact toric 3-manifolds $(Y, \xi)$ and   $(Y', \xi')$ that correspond to the numbers $t_1,t_2$ and $t_1'=t_1,t_2'=t_2+2n\pi$, for some $n\in\mathbb N,$ respectively. According to Lerman,  $Y$ and $Y'$ are diffeomorphic manifolds and the corresponding contact structures are homotopic as 2-plane fields by \cite[Theorem 3.2]{Lerman_contact_cut}. 
   
   In particular, if $t_2-t_1\leq \pi$   the contact structure $\xi'$ is the unique  overtwisted   contact structure in the same homotopy class as the  tight contact structure  $\xi$.
   If  $t_2-t_1>\pi$ then  contact structures $\xi$ and $\xi'$ are homotopic overtwisted contact structures.
   
      \end{remark}

\section{Constructions of  linear plumbings with concave boundary using toric models} \label{section concave toric}

In this section we give our construction of the symplectic toric structure  with a concave contact toric boundary, on certain linear plumbings over spheres.  In  \S\ref{section tight-ot} we will use this method to realize the toric moment cones in terms of the self-intersection numbers associated to the plumbings,  allowing us to establish our classification results for when the  concave contact 3-manifold associated to the boundary of a linear plumbing is either tight, and hence symplectically fillable, or overtwisted.  We will also make use of the induced global contact toric structure in our computation of algebraic torsion in  \S\ref{section computing}.

We prove the following theorem.
      
  \begin{theorem} \label{theorem general}    
A linear plumbing $(s_1,\ldots,s_n)$ where each surface is a 2-sphere and $s_i \geq0$ for at least one index $i\in\{1, \ldots, n\}$, admits a structure of a symplectic toric 4-manifold with a concave contact toric boundary.
   \end{theorem}   

Throughout this section $i$  denotes a specific index such that $s_i\geq 0$.
The proof is derived by constructing a  global moment map image in such a way that the chain of $n$ adjacent edges corresponds to the chain of $n$ spheres in the base of the plumbing with the given self-intersection numbers.
The case $n=2$ is proved in \S\ref{L-shape} and the general case is proved in \S\ref{gluing of L-shapes}.

The idea of imposing a global symplectic toric structure on the plumbing is already seen in  \cite{sym42}, where Symington decomposes the plumbing $(s_1,\ldots,s_n)$ into the sequence of $n$ disc bundles over 2-spheres, attaches  a moment map image to each of them and glues them all together. Our contribution is that with our method we are able to recognise the concavity of the boundary, as well as criteria for tightness and overtwistedness of the contact structure, later in \S\ref{section tight-ot}.

Note that, according to Gay-Stipsitz \cite[Theorem 1.2]{gs}, a linear plumbing $(s_1,\ldots,s_n)$ where each surface is a 2-sphere and  $s_j \leq-2$, for all 
$j=1, \ldots, n$, admits a structure of a symplectic manifold with a convex contact type boundary, as the corresponding self-intersection form is negative-definite. Therefore, after performing blow down to any sphere with self-intersection number $-1$,  we are in one of two possible cases. Either $s_j \leq-2$, for all 
$j=1, \ldots, n$ and, then, the boundary is of convex contact type, or, there exists at least one index $i\in\{1, \ldots, n\}$  for which $s_i \geq0$ and then the boundary is of concave contact type, as we establish in Theorem~\ref{theorem general}.

\subsection{Plumbing $(s_1,s_2)$,  with $s_1\geq0 $ or $s_{2}\geq0 $} \label{L-shape}

Consider the plumbing of two disc bundles over 2-spheres with self-intersection numbers $s_1$ and $s_{2}$.
The  self-intersection form of this plumbing is 
$Q_\Gamma=\begin{bmatrix}
s_1 & 1 \\
1 & s_{2} 
\end{bmatrix}.$
Without loss of generality, we assume that the plumbed spheres are not exceptional divisors of a blow up, i.e. $s_1,s_{2}\neq-1$. Otherwise, we could perform a blow down, by replacing such a sphere with a Darboux neighborhood of a point. In this way we change the second homology of the 4-manifold, but not the boundary of the plumbing. 

Suppose $s_1\geq0$. Observe first that  there exist $(z_1,z_{2})\in\mathbb R^2_{<0}$ such that 
  $$-Q_\Gamma(z_1,z_{2})=(a_1,a_{2})$$
  for some $(a_1,a_{2})\in \mathbb R^2_{>0}.$ Namely, if $s_{2}\geq0$ we take $z_1=z_{2}=-1$ and if $s_{2}<0$ we take $z_{2}=-1, z_1=s_{2}-1.$ The proof is analogous if $s_2\geq0$ and $s_1<0.$ Therefore, we construct a moment map image as shown in Figure~\ref{fig:L-shape} and we refer to it as an L-shape. 
  
   Let $W$ denote the  symplectic toric manifold whose moment map image is the constructed L-shape. In this case where we plumb exactly two disk bundles, we perform symplectic reduction along the portions which project to the orange and blue edges on the boundary of the region according to the inward normals of those edges, but we do not perform symplectic reduction along the portion which projects to the green dotted portion of the boundary of the moment map image. The boundary of $W$ projects to the green dotted curve in the boundary of the moment map image.

\begin{figure}
\centering
\includegraphics[scale=.4]{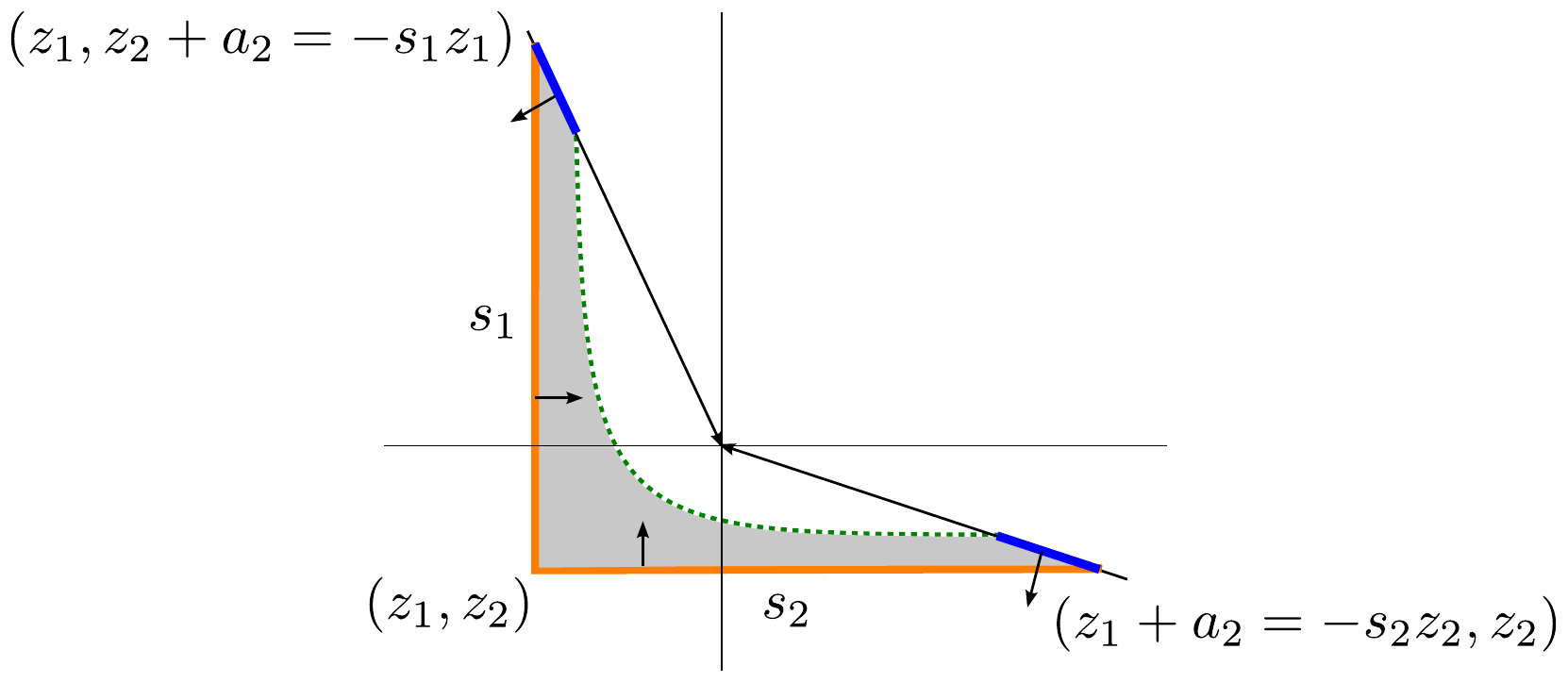}
\caption{L-shape moment map image corresponding to the pair $(s_1,s_{2})$. In this example, $s_1,s_{2}>0$.}
\label{fig:L-shape}
\end{figure}

  The orange edges of the L-shape correspond to symplectic spheres with self-intersection numbers $s_1$ and $s_2$ and symplectic areas  $a_1$ and $a_2,$ as can be computed from the moment map image (see \textsection\ref{section: symplectic 4}). The three vertices of the L-shape  correspond to  fixed points of the action. The points in the blue lines, except the vertices, correspond to circles in $W$ and, therefore, the blue lines correspond to 2-dimensional discs. The interior points of the green curve correspond to tori while the end points correspond to circle orbits obtained by collapsing the tori along particular slopes.  In particular $\partial W$ is diffeomorphic to a lens space. 

   \begin{lemma} \label{lemma global vector field}    
A symplectic toric manifold $W$ corresponding to a linear plumbing $(s_1,s_{2})$, where $s_1\geq0 $,  admits a Liouville vector field defined near the boundary $\partial W$ that is transversal to the boundary and points inward along the boundary.
    \end{lemma}     
  
      \begin{proof}  
As explained in \textsection\ref{subsec:toric_intro}, on the set of regular torus orbits in $W$ we have action-angle coordinates $(p_1,q_1,p_2,q_2)$ and  the symplectic form is given by  $\omega=dp_1\wedge dq_1+dp_2\wedge dq_2$. Therefore, the vector field $V=p_1\partial_{p_1}+p_2\partial_{p_2}$ is a Liouville vector field on this set of regular torus orbits, which is the part of the boundary of $W$ that projects to the interior of the green curve. This vector field $V$ is invariant under the toric action so we can depict it in the moment map image as radial arrows emanating from the origin.  Since $z_1,z_{2}<0$, $V$ points inward to the L-shape along the green curve.

Let us explain how $V$ can be smoothly extended to a neighborhood of the two circle orbits in $\partial W$, which are the fibers over the endpoints of the green dotted curve. As global action--angle coordinates are not defined on the singular orbits we will use polar coordinates $(\rho_1,\theta_1,\rho_2,\theta_2)$ in the following way. Consider the circle orbit that projects to a green point on the ray given by the direction $(1,-s_1)$.  
We  suppose that the green curve representing the boundary of $W$ is a straight line in a small neighborhood of this point. Take a small neighborhood of the chosen circle orbit that projects to the shaded parallelogram in Figure~\ref{fig:change}
on the left.

    \begin{figure}
\centering
\includegraphics[scale=.4]{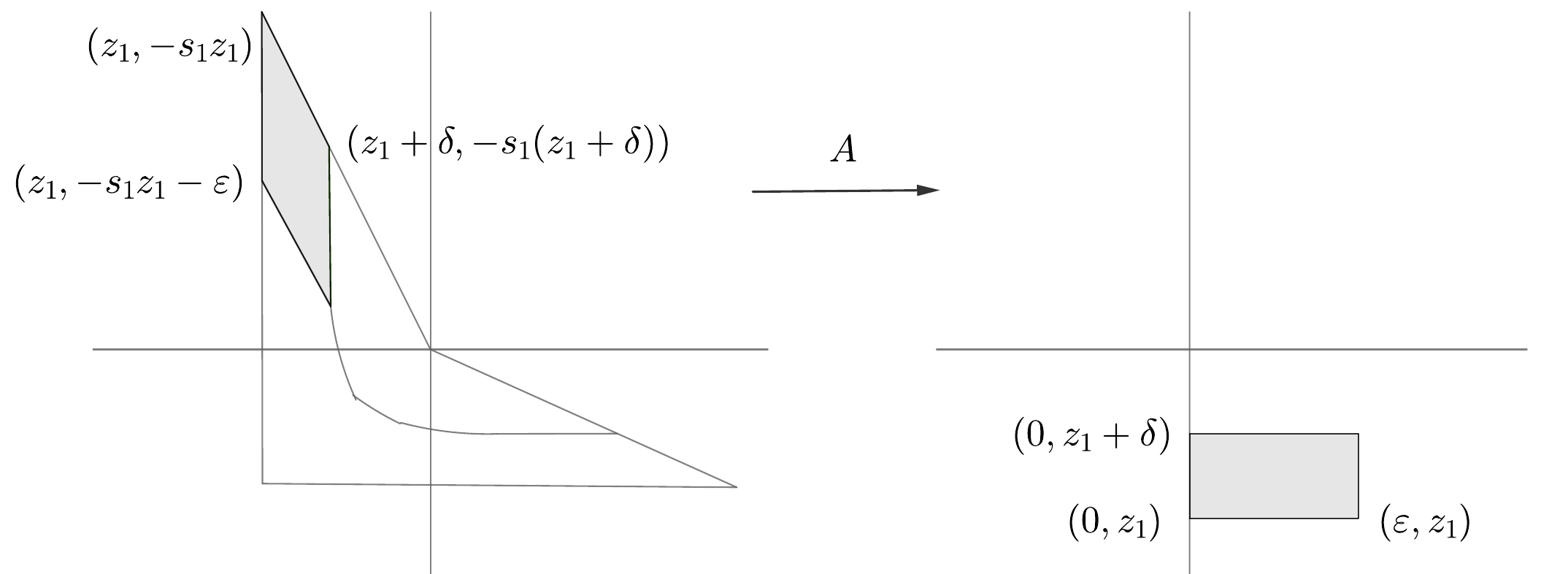}
\caption{$SL(2, \mathbb Z)$-transformation of the action-angle  coordinates}
\label{fig:change}
\end{figure}   

We first perform the transformation
  $A=\left[\begin{array}{cc} -s_1 & -1\\ 1 & 0 \end{array} \right]  $ of the moment map image that maps the shaded parallelogram on the left, to the rectangle on the right part of Figure~\ref{fig:change}. In this way we obtain new $A(p_1,p_2)$ coordinates in the moment map image and $A^{-T}(q_1,q_2)$ coordinates on each torus orbit. Then, we use the following transformation
  \begin{equation}
    \begin{split}\nonumber
     A(p_1,p_2) & \mapsto\left(  \frac{\rho_1^2}{2},  \frac{\rho_2^2}{2}+z_1\right),\\
     A^{-T}(q_1,q_2)  & \mapsto \left( \theta_1,  \theta_2 \right)
    \end{split}
\end{equation}
in order to work with polar coordinates since it is well understood how polar coordinates degenerate when collapsing the angular coordinate $\theta$ at radius $\rho=0$, which corresponds to the circle orbits.

Putting these transformations together we obtain

\begin{equation}
    \begin{split}
       p_1&= \frac{\rho_2^2}{2}+z_1,\hskip7mm p_2= -\frac{\rho_1^2}{2}- \frac{s_1\rho_2^2}{2}-s_1z_1,\\       q_1 &=-s_1 \theta_1+\theta_2,\hskip2mm q_2= -\theta_1.
       \end{split}
\end{equation}
Thus,
\begin{equation}
    \begin{split}
        dp_1&= \rho_2d\rho_2,\hskip13mm dp_2= -\rho_1 d\rho_1-s_1\rho_2 d\rho_2,\\       dq_1 &=-s_1 d\theta_1+d\theta_2,\hskip2mm dq_2= -d\theta_1.    \end{split}
\end{equation}
and
\begin{equation}
    \begin{split}
   \partial_{p_1}&= \frac{1}{\rho_2}\partial_{\rho_2}-  \frac{s_1}{\rho_1}\partial_{\rho_1},\hskip2mm\partial_{p_2}=-\frac{1}{\rho_1}\partial_{\rho_1}\\
     \partial_{q_1}&= \partial_{\theta_2}, \hskip20mm
      \partial_{q_2}= -\partial_{\theta_1}-s_1 \partial_{\theta_2}.     \end{split}
\end{equation}

The symplectic form $
\omega =dp_1\wedge dq_1+dp_2\wedge dq_2$ transforms to $\rho_1d\rho_1\wedge d\theta_1+\rho_2d\rho_2\wedge d\theta_2$ and
 the Liouville vector field $V=p_1\partial_{p_1}+p_2\partial_{p_2}$ transforms to
$$V=\frac{\rho_1}{2}\partial_{\rho_1} + \left(\frac{\rho_2}{2}+ \frac{z_1}{\rho_2}\right)\partial_{\rho_2}.$$
Along the line of singular orbits we have $\rho_1=0$ and thus $V$ extends to this set as well. The proof is analogous in the neighborhood of the other circle orbit. Therefore, $V$ is globaly defined in the neighborhood of the boundary of $W.$

Finally,  the term $\frac{\rho_2}{2}+ \frac{z_1}{\rho_2}$ is negative for a sufficiently small $\rho_2>0,$ since $z_1<0.$
Equivalently, in the moment map image  the radial rays emanating from the origin points into the interior of the L-shape and are transversal to the green curve.
Therefore, the Liouville vector field $V$  is transversal to the boundary of $W$ and points inward the boundary.
      \end{proof}

 As the Liouville vector field points transversely inward  along the boundary of $W_i$, it gives the boundary of $W_i$ a concave contact structure. According to Lemma~\ref{lemma toric contact}, this contact structure admits a toric action as well. This completes the proof of Theorem~\ref{theorem general} in the case $n=2.$

 \vskip2mm
 
\begin{remark}\label{rem:rays of L-shape}    
 We point out that, due to Theorem~\ref{theorem Lerman class}, the contact structure on the boundary of the plumbing $(s_1,s_2)$ is determined by the rays 
 \begin{equation}\label{eq:rays of one L-shape}
 R_1=(1,-s_1), \ \ \  R_2=(-s_{2},1),
\end{equation}
 pointing toward the origin,
 since $z_{2}+a_{1}=-s_1z_1$ and $z_1+a_{2}=-s_{2}z_{2}$.
Note that the contact structure on the boundary does not depend on the choice of $(z_1,z_{2})$ and $(a_1,a_{2}).$ The corresponding moment cone is shown in Figure~\ref{fig: cone} on the right.
\end{remark}

\subsection{Linear plumbing $(s_1,\ldots,s_n)$, with at least one $s_i\geq0 $} \label{gluing of L-shapes}

Consider the linear plumbing of $n$ disc bundles over 2-spheres with self-intersection numbers $s_1, \ldots, s_n\in\mathbb Z,$ where   $s_i\geq0 $ for at least one index $i \in  \{1, \ldots, n\}$. 
As in the case of a plumbing of two disc bundles over 2-spheres, we assume that no sphere is an exceptional divisors,  that is $s_j\neq -1$ for every $j=1, \ldots, n$. 
The corresponding  self-intersection form is 
$$Q_\Gamma=\left(\begin{array}{ccccccc} 
s_1 & 1 & 0 & 0 & \cdots & 0 & 0\\
1 & s_2 & 1 & 0 & \cdots & 0 & 0\\
0 & 1 & s_3 & 1 & \cdots & 0 & 0\\
0 & 0 & 1 & \ddots &        &  & \\
  &  &  &  & \ddots & 1 & \\
  &  &   &    &        1  & s_{n-1} & 1\\
  &  &   &    &          & 1 & s_n
\end{array}  \right) .$$

 Now we will prove Theorem~\ref{theorem general}. The rough idea is to build the toric structure on the linear plumbing by gluing together $n-1$ toric pieces, whose moment map images are the same as the L-shapes for the 2-vertex plumbings of \S\ref{L-shape}. Next, we use $SL(2,\mathbb{Z})$ transformations of the plane to paste all the L-shapes together to obtain a global moment map image. We then show that the resulting global symplectic toric manifold $W$ is diffeomorphic to the given linear plumbing.

\begin{proof}[Proof of Theorem~\ref{theorem general}]
 
 \textbf{Step 1.} \emph{Toric decomposition.}\\
Suppose $i$ is an index such that $s_i\geq 0$. We first decompose the sequence $(s_1,\dots, s_n)$ into $n-1$ pairs 
\begin{equation}\label{sequence of L-shapes}
\mathcal{P}_1:=(s_1,0),\ldots, \mathcal{P}_{i-1}:=(s_{i-1},0),\mathcal{P}_i:=(s_i,s_{i+1}),\mathcal{P}_{i+1}:=(0,s_{i+2}),\ldots, \mathcal{P}_{n-1}:=(0,s_n).
\end{equation}
Observe that in each pair $\mathcal{P}_j$, at least one of the two integers is non-negative.

Next, we associate a toric piece to each $\mathcal{P}_j$.  The corresponding moment map image will be $L_j$, the L-shape of \S\ref{L-shape} determined by the pair $\mathcal{P}_j$, with vertex at coordinates $(z_j,z_{j+1})$ where we will specify the values $z_j$ below. For $L_2,\dots, L_{n-2}$, we keep the toric action free over the points in the blue edges instead of performing symplectic reduction (except at the vertex where the blue and orange edges meet, where the toric fiber will now be a circle). Topologically, this toric piece is just a $4$-ball, but its boundary is decomposed into three pieces: a solid torus above each of the two blue edges, and one $T^2\times I$ above the interior of the green edge. The solid torus preimage of the upper left blue edge of $L_j$ will be glued via an $SL(2,\Z)$ transformation to the solid torus preimage of the lower right blue edge of $L_{j-1}$. For the $L_1$ we do symplectic reduction along the top blue edge, but not the bottom right blue edge, and for $L_{n-1}$ we do symplectic reduction only on the bottom right blue edge. Thus at the end, when all the L-shapes will be glued together via the $SL(2,\Z)$ transformations, the only remaining part of the boundary of the symplectic toric $4$-manifold will be the preimage of the union of the green dotted curves as in Figures~\ref{fig:L-shape0}, \ref{fig:L-shape1} and \ref{fig:L-shape2}. \\

\begin{figure}
\centering
\includegraphics[scale=.5]{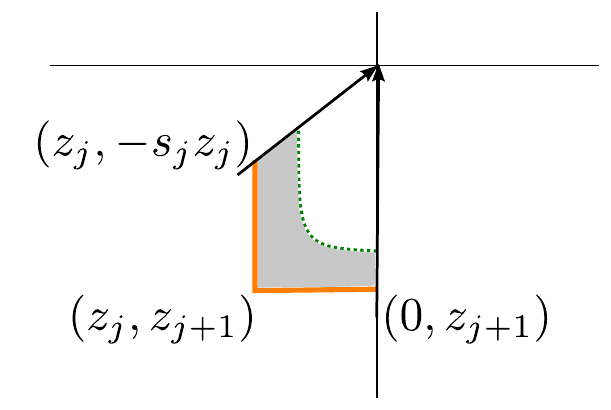}
\caption{L-shape moment map image corresponding to the pair $\mathcal{P}_j=(s_j,0)$ for $j<i$. In this example, $s_j<0$.}
\label{fig:L-shape0}
\end{figure}

\begin{figure}
\centering
\includegraphics[scale=.4]{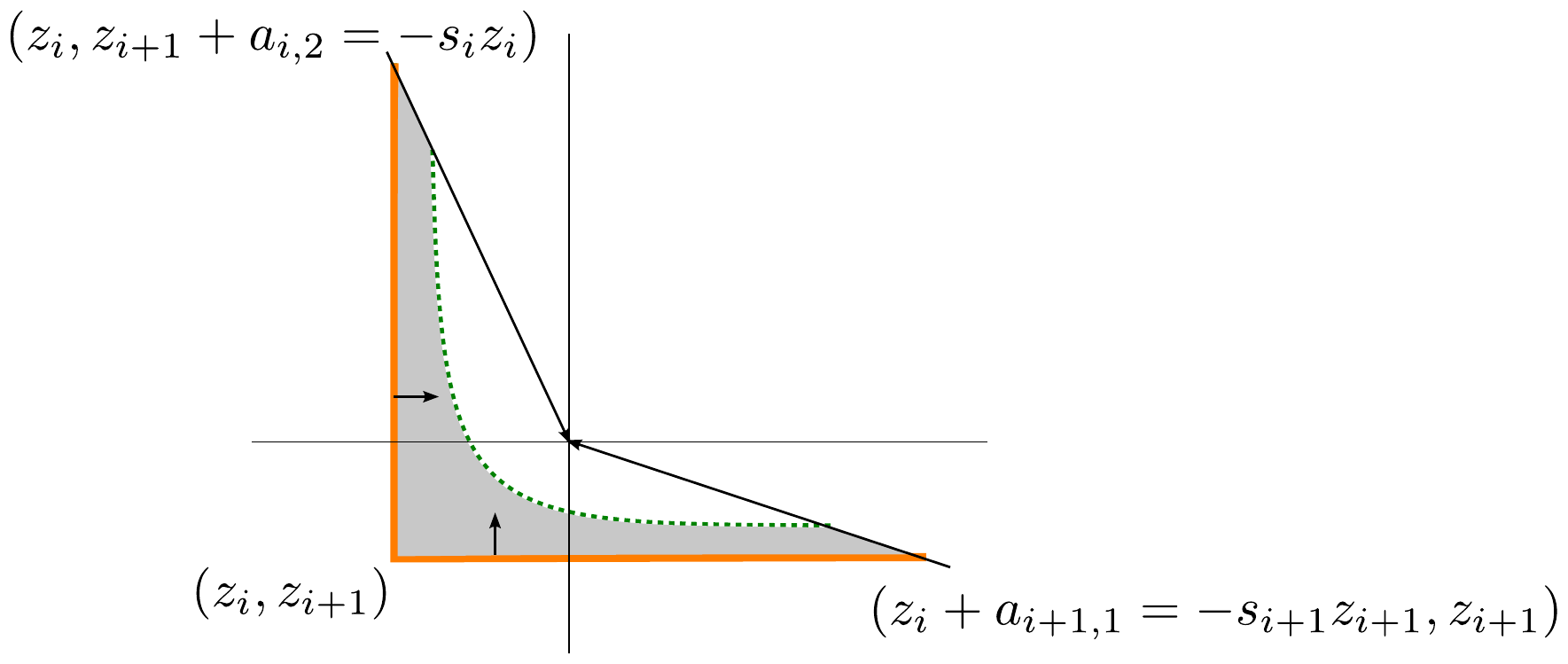}
\caption{L-shape moment map image corresponding to the pair $\mathcal{P}_i=(s_i,s_{i+1})$. In this example, $s_i,s_{i+1}>0$.}
\label{fig:L-shape1}
\end{figure}

\begin{figure}
\centering
\includegraphics[scale=.5]{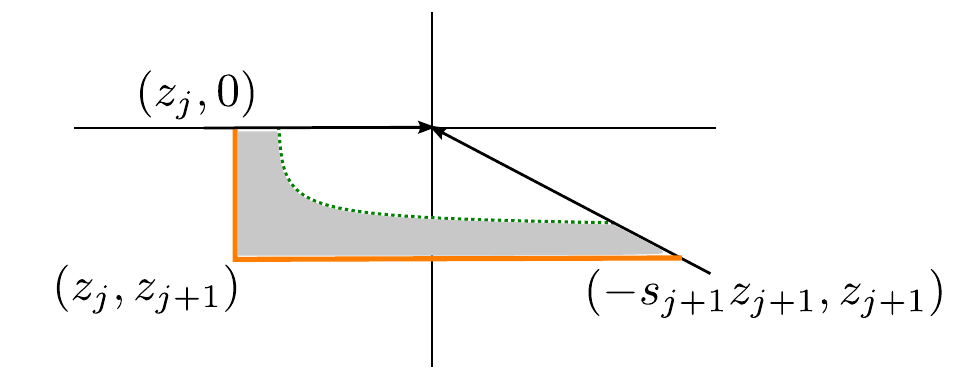}
\caption{L-shape moment map image corresponding to the pair $\mathcal{P}_j=(0,s_{j+1})$ for $j>i$. In this example, $s_j>0$.}
\label{fig:L-shape2}
\end{figure}

\noindent \textbf{Step 2.} \emph{Pinning down coordinates of the L-shapes.}\\ To fully specify the L-shape associated to each pair $\mathcal{P}_j$, we must also specify the coordinates $(z_j,z_{j+1})$ of its corner vertex and the associated symplectic areas $(a_j,a_{j+1})$, which determine the remaining two vertices. For a concave structure, we need $z_j<0$. We now specify how to choose the values $z_1,\dots, z_n \in \R_{<0}$.

Denote by 
 $Q_j$, $j=1, \ldots, n-1$, the corresponding self-intersection forms of the 2-vertices plumbings $\mathcal{P}_j$. 
First, let us  find  $(z_1, \ldots,z_{n})\in\mathbb R^n_{<0}$ such that
\begin{equation}
 \begin{split}\nonumber     -Q_1(z_1,z_{2}) &=(a_{1},a_{2,1}),\\
     -Q_j(z_j,z_{j+1}) &=(a_{j,2},a_{j+1,1}),\hskip2mm j=2,\dots, n-2,\\
     -Q_{n-1}(z_{n-1},z_{n}) &=(a_{n-1,2},a_{n}), 
     \end{split}
\end{equation}
for some $a_1,a_{j,1},a_{j,2}, a_n \in \mathbb R_{>0}$ as follows: 
\begin{itemize}
    \item If $1<i$,  take $z_1=-1$ and   choose $z_j<0$, for all $1<j<i,$ by  using the inequality $z_j<-s_{j-1}z_{j-1}$.

\item Analogously, if $i<n,$ take $z_n=-1$ and choose $z_j<0$ for all $n>j>i$ while satisfying by using the inequality  $z_{j}<-s_{j+1}z_{j+1}.$ 

\item Finally, choose $z_i<0$ so that $z_i<-s_{i-1}z_{i-1}$ and
    $z_i<-s_{i+1}z_{i+1}$.  
\end{itemize}

Set $a_j=a_{j,1}+a_{j,2}$ for all $j=2,\dots, n-1$.  Let us briefly explain why 
\begin{equation}\label{eqn:area}
 \begin{split}
     a_1 & =-s_1z_1-z_2,\\
     a_j  & = -s_jz_j-z_{j+1}-z_{j-1},\hskip2mm j=2,\dots, n-1, \\
     a_n & = -z_{n-1}-s_nz_n. 
     \end{split}
\end{equation}
The equalities for $a_1$ and $a_n$ follow from the respective equations 
$$-Q_{1}(z_{1},z_{2})=(a_{1},a_{2,1}) \qquad  
\mbox{and} \qquad  -Q_{n-1}(z_{n-1},z_{n})=(a_{n-1,2},a_{n}).$$ To show the equality for $a_j$, for any $j=2,\dots, n-1,$ 
consider two adjacent pairs 
$$\mathcal{P}_{j-1}=(s_{j-1},s_j^{-})  \qquad \mbox{and} \qquad \mathcal{P}_j=(s_j^{+},s_{j+1})$$
where $s_j^{-}+s_j^{+}=s_j$. Apply the equations
$$-Q_{j-1}(z_{j-1},z_{j})=(a_{j-1,2},a_{j,1}) \qquad\mbox{and} \qquad -Q_j(z_j,z_{j+1})=(a_{j,2},a_{j+1,1}),$$
together with the definition $a_j=a_{j,1}+a_{j,2}$. Thus \eqref{eqn:area} holds. \\

\begin{remark}  \label{rem: split1}    Note that the choice of $s_j^-$ and $s_j^+$ in the equation 
$s_j^{-}+s_j^{+}=s_j$ does not affect the value of $a_j$,  since the computation gives 
$a_j = -(s_j^-+s_j^+)z_j-z_{j+1}-z_{j-1}$. See 
 also Remark \ref{rem:split s_j} where it is shown that this choice does not affect global self-intersection numbers, as well.
 Note that, in the decomposition (\ref{sequence of L-shapes}) we have chosen 
$s_j^-=0, s_j^+=s_j,$ if $j\leq i,$ 
and $s_j^-=s_j,s_j^+=0$, if $j>i.$ 
Any choice works as long as in each $\mathcal{P}_j$ at least one coordinate is non-negative.
\end{remark}
  
\noindent  \textbf{Step 3.} \emph{Gluing together the L-shapes.}\\
Now we show that these $n-1$ toric pieces can be glued along the solid tori in the preimage of the blue edges via $SL(2,\Z)$ transformations to obtain a global toric structure. We subsequently verify that this global symplectic toric manifold is symplectomorphic to the linear plumbing $(s_1,\dots, s_n)$ where the core spheres have symplectic areas $(a_1,\dots, a_n)$ and that  the boundary admits a concave structure.

We claim that for $j=2,\dots, n-1$, the $SL(2,\Z)$ transformation 
$$A_j = \left[\begin{array}{cc} -s_{j} & -1\\ 1 & 0 \end{array} \right]$$ 
takes the blue interval in the top right of the L-shape associated to $\mathcal{P}_j$ to the blue interval in the bottom right of the L-shape associated to $\mathcal{P}_{j-1}$. For example, if $\mathcal{P}_j=(0,s_{j+1})$ and $\mathcal{P}_{j-1}=(0,s_{j})$, the upper right blue interval of the $j^{th}$ L-shape is the straight line segment from $(z_{j},0)$ to $(z_{j}+\delta,0)$ and the lower left blue interval of $j-1^{st}$ L-shape is the straight line segment from $(-s_jz_j,z_j)$ to $(-s_j(z_j+\delta),z_j+\delta)$, and $A_j$ is the correct transformation relating these. The reader can verify computationally the other cases.

Our global toric manifold will then have moment map image
$$L_1\cup A_2(L_2 \cup A_3(L_3 \cup \cdots A_{n-1}(L_{n-1}) )).  $$

In this global picture, the union of the orange edges become $n$ edges, whose toric preimages are $n$ symplectic spheres. We claim that their self-intersection numbers are exactly $(s_1,\dots, s_n)$, and their symplectic areas are $(a_1,\dots, a_n)$, which suffices to verify that this global toric manifold is symplectomorphic to the specified linear plumbing, according to \cite[Proposition 3.5]{sym_rational_blow}, see also \cite{rae}.

\begin{remark}\label{rem:split s_j}    Note that the same matrix $A_j$ can be applied in order to glue
      $ \mathcal{P}_j=(s_j^{+},s_{j+1}) $  to    $\mathcal{P}_{j-1}=(s_{j-1},s_j^{-})$, for any choice
     of $s_j^-$ and $s_j^+$, where
$s_j^{-}+s_j^{+}=s_j,$ since $A_j$ maps the point $(-1,s_j^+) $ of the L-shape associated to 
$ \mathcal{P}_j$ to the point $(s_j^-,-1)$ of $ \mathcal{P}_{j-1}$. See also Remark
 \ref{rem: split1}.

\end{remark}

\noindent \textbf{Step 4.} \emph{Verification of the plumbing data via the moment map image.}\\
As explained in 
\textsection\ref{section: symplectic 4},  both self-intersection and symplectic area of symplectic spheres corresponding to edges can be read from the moment map image.
 The first sphere is the preimage of the vertical edge of $L_1$ with end points $(z_1,-s_1z_1)$ and $(z_1,z_2)$. 
 The symplectic area is equal to the affine length of this edge $-s_1z_1-z_2=a_1$ by equation~\ref{eqn:area}. The self-intersection number is equal to $\det((0,1),(-s_1,-1))=s_{1},$ where $(-s_{1},-1)$ is the inward normal vector to the ray given by the direction $(1,-s_1)$.
For $j=2,\dots, n-1$, the $j^{th}$ sphere is the preimage of the edge contained in 
$$A_2\dots A_{j-1}(L_{j-1}\cup A_jL_j).$$ 
Instead of computing self-intersection and symplectic area in the moment image $A_2\dots A_{j-1}(L_{j-1}\cup A_jL_j)$ we  compute it in $L_{j-1}\cup A_jL_j$ (Figure~\ref{gluing}),  as the $SL(2,\Z)$ transformation $A_2\dots A_{j-1}$ preserves the computation (see \textsection\ref{section: symplectic 4}). 

In $L_{j-1}\cup A_jL_j$, the sphere is the preimage of the edge with end points $(z_{j-1},z_j)$ and $(-s_jz_j-z_{j+1},z_j)$. The symplectic area is simply the affine length of this interval $$-s_jz_j-z_{j+1}-z_{j-1}=a_j$$ by \eqref{eqn:area}. The self-intersection number is equal to $$\det((-1,-s_{j}),(1,0))=s_{j},$$ where $(-1,-s_{j})$ is the inward normal vector to the edge given by the direction $A_{j}(1,0)$.
 The $n^{th}$ sphere in the chain  is the preimage of the edge contained in 
$A_2\dots A_{n-1}(L_n)$. Equivalently, it is the pre-image of the edge in $L_n$ with end points $(z_{n-1},z_n)$ and $(-s_nz_n,z_n).$ The symplectic area is  $$-s_nz_n-z_{n-1}=a_n$$ by \eqref{eqn:area}. The self-intersection number is  $$\det((-1,-s_{n}),(1,0))=s_{n},$$ where $(-1,-s_{n})$ is the inward normal vector to the ray given by the direction $(-s_n,1)$. \\

\noindent \textbf{Step 5.} \emph{Checking the concave contact boundary condition.}\\
 Finally, the radial vector field 
$p_1\partial_{p_1}+p_2\partial_{p_2}$
 is invariant  under the gluing procedure, since its moment map image is invariant under  linear transformations.
 Thus, it is Liouville and transversal to the boundary along the set of all regular torus orbits in $\partial W$.
The proof of Lemma~\ref{lemma global vector field}   generalises to this case, so  the radial vector field 
   extends to a neighborhood of two singular orbits in $\partial W$. It is a Liouville vector field with respect to this symplectic structure and it points transversaly into $W$ along $\partial W$. Therefore, the boundary $\partial W$ admits a concave contact  structure that is, according to Lemma~\ref{lemma toric contact} also a toric structure. 
\end{proof}

\vskip2mm
 
   We now describe an example of the construction with $i=2$ and $n=4$.
      
   \begin{figure}
\centering
\includegraphics[width=13cm]{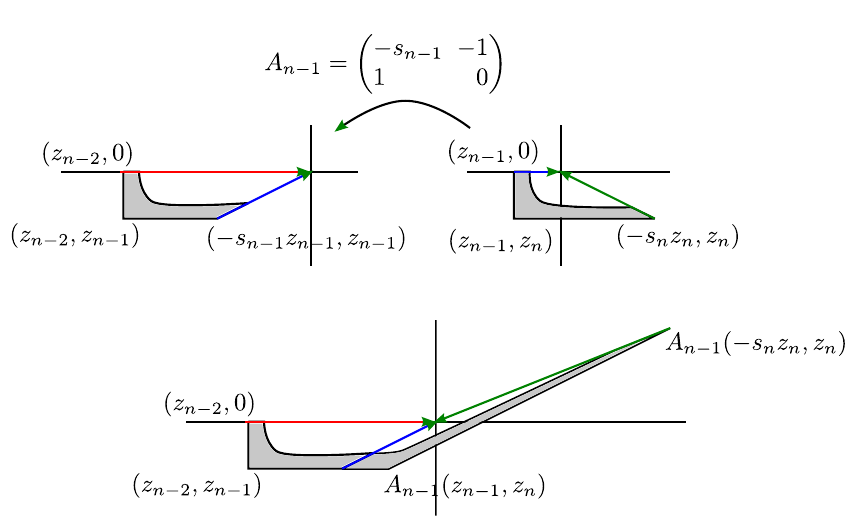}
\caption{Gluing the L-shape $(0,s_n)$ to the L-shape $(0,s_{n-1})$}
\label{gluing}
\end{figure}        

    \begin{example} \label{ex:4vertices}
        Consider the linear plumbing of disk bundles over spheres with self-intersection numbers $(-2,1,0,-2)$. Following the procedure above, we can set $z_1=-1$, $z_2=-3$, $z_3=-3$, $z_4=-1$. We build three L-shapes corresponding to the pairs $(-2,0)$, $(1,0)$, $(0,-2)$. We glue the $(0,-2)$ L-shape to the $(1,0)$ L-shape via the $SL(2,\Z)$ transformation
        $$A_3 = \left[ \begin{array}{cc} 0&-1\\1&0 \end{array} \right] $$
        and glue the union of the $(1,0)$ L-shape with $A_3$ of the $(0,-2)$ L-shape to the $(-2,0)$ L-shape via the $SL(2,\Z)$ transformation
        $$A_2 = \left[ \begin{array}{cc} -1&-1\\1&0 \end{array} \right].$$ The result is shown in Figure~\ref{fig:ex4vertices}.
    \end{example}

    \begin{figure}
        \centering
        \includegraphics[scale=1.5]{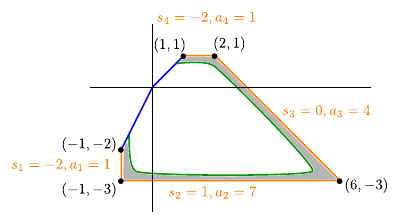}
        \caption{The global toric moment map image of the plumbing $(-2,1,0,-2)$ with symplectic areas $(1,7,4,1)$.}
        \label{fig:ex4vertices}
    \end{figure}

 \begin{remark}

A construction of a symplectic form and a Liouville vector field on a plumbing, in the case when the corresponding self-intersection form is negative-definite, was given by Gay and Stipsicz~\cite{gs}. In that case, the contact structure on the boundary is of convex type. Their construction is generalized to the concave case by  Li and Mak \cite{lm} and uses a combination of toric models that correspond to edges of the plumbing and product models that correspond to the vertices of the plumbing. These are glued together according to the graph of the plumbing.

In the case of a  linear plumbing of disk bundles over spheres, the product models are $D^2\times D^2$ and $A\times D^2$ portions, where $A$ denotes the annulus, which also admit toric coordinates. Therefore, our method for the linear case simplifies their method. However, our method cannot be extended to a non-linear case because a product model 
$\Sigma\times D^2$, where $\Sigma$ is a higher genus $g$ surface, that corresponds to a vertex with valences $g+1$, does not admit toric structure.

Note also that by our construction, we establish that
    $$-Q_\Gamma(z_1,  \ldots, z_n)=(a_1,  \ldots, a_n),$$
    where $(z_1,  \ldots, z_n) \in \mathbb R^n_{<0}$, $(a_1,  \ldots, a_n) \in \mathbb R^n_{>0}$   and this is the necessary conditions for the Gay-Stipsicz and Li-Mak constructions to yield a concave symplectic structure on a manifold $W$  that is diffeomorphic to the linear plumbing $(s_1,\ldots, s_n)$.   
    \end{remark}

  \section{Tight and overtwisted contact structures on the concave boundary of a linear plumbing}   \label{section tight-ot}
  
   Let  $(s_1,\ldots,s_n)$ be any linear plumbing over spheres  where   $s_i\geq0$  for  at least one index $i\in\{1,\ldots,n\}$.  As explained in the previous section, the total space  can be equipped with a symplectic toric structure with a concave contact toric boundary.  In this section we make use of the moment cones associated to the contact toric structures and we show how the intersection numbers  $s_1,\ldots,s_n$ determine the topology of the boundary  and, in almost all cases,  whether this contact structure is tight or overtwisted.

 The first part of the following proposition is well known (\cite[Theorem 6.1.]{Neumann}) in the case when $s_j\leq-2$,  for all $j\in\{1,\ldots,n\}.$ 
To set the stage, we provide the proof in the case when $s_i\geq0$  for at least one $i\in\{1,\ldots,n\}.$   Although the statement is purely topological, the proof is derived using toric tools.
   
  \begin{proposition}\label{prop lens} 
  Let $(s_1,\ldots, s_n)$ be a linear plumbing, where   $s_i\geq0$  for at least one $i\in\{1,\ldots,n\}.$    
  \begin{itemize}
      \item[(a)] Assume that the continued fraction   $$[s_1,\ldots, s_n]=s_1-\frac{1}{s_2-\frac{1}{\cdots-\frac{1}{s_n}}}$$
  is well defined. Then the boundary of the given linear plumbing 
   is diffeomorphic to the lens space $L(k,l)$, where $k/l=[s_1, \ldots, s_n].$

 \item[(b)] If the determinant of the  corresponding intersection form $Q$ is equal to $k\in\mathbb Z$ then the boundary of the  linear plumbing  is diffeomorphic to $L(k,l)$, for some $l\in\mathbb Z$. In particular, if $ \det Q$ is equal to $0,1$ or $2,$ then the boundary is diffeomorphic to $S^1 \times S^2,$ $S^3$ or $ \mathbb{RP}^3$, respectively.
 \end{itemize} 
   
   \end{proposition}

     \begin{proof}
   (a) 
   According to Theorem~\ref{theorem Lerman class} and the explanation below, the topology of the boundary is specified by  the two rays bounding the corresponding moment cone. In particular, when these rays are given by the directions $R_1=(-1,0)$ and $R_2=(-l,-k)$, pointing toward the origin, then the boundary is diffeomorphic to a lens space $L(k,l)$, by Remark~\ref{remark: all cones for L(k,l)}. 
   
   Next we explain why $k/l=[s_1, \ldots, s_n].$ First, we observe that   by equation 
   \eqref{eq:rays of one L-shape}
   and the construction developed in 
\textsection\ref{gluing of L-shapes}  it follows that the rays bounding
the moment cone of  the boundary of 
$(s_1,\ldots, s_n)$ are
 $$R_1=(1,-s_1)\hskip2mm\textrm{and}
 \hskip2mm R_2=\begin{cases}
        (-s_2,1),&\mbox{ if $n=2$,  }\\  
  A_2  \cdots A_{n-1}(-s_n,1),&\mbox{ if $n\geq3$,}
 \end{cases}
$$ 
where $(-s_n,1)$ is the second ray defining the last L-shape and 
$A_j = \left[\begin{array}{cc} -s_{j} & -1\\ 1 & 0 \end{array} \right].$
 Note that both rays point toward the origin.

Second, we perform an $SL(2, \mathbb Z)$ transformation
      $B= \left[\begin{array}{cc} -1 & 0
\\ -s_1 & -1 \end{array} \right]$      
    that maps the ray $R_1=(1,-s_1)$ to the ray $(-1,0)$. 
 Then, the ray $R_2$ will be mapped to the ray $(-l,-k),$ for some integer values of $k$ and $l$. Therefore, if $n=2$, 
    $$(-l,-k)=B(-s_2,1)=(s_2,s_1s_2-1)$$
    and, thus, $k/l=s_1-1/s_2.$ If $n\geq3$, then
    $$(-l,-k)=BA_2  \cdots A_{n-1}(-s_n,1) $$
   and the result can be shown  in the following way.  Set $(-l_{n-1},-k_{n-1})=A_{n-1}(-s_n,1)$. Then it holds  
   $$k_{n-1}/l_{n-1}=-\frac{1}{s_{n-1}-1/s_n}.$$
   And, more generally, set $(-l_j,-k_j)=A_{j}(l_{j+1},k_{j+1})$, for all $1<j<n$. Then, it holds
   $$k_j/l_j=-\frac{1}{s_{j}+k_{j+1}/l_{j+1}}.$$ 
   In particular,  $(-l_{2},-k_{2})=A_2 \cdots A_{n-1}(-s_n,1)$ and therefore 
   $$k_{2}/l_{2}=-\frac{1}{s_2-\frac{1}{\ldots-1/s_n}}.$$
   Finally, $(-l,-k)=B(l_2,k_2)$ and the result follows. 
  
(b) Let us first show that 
$$\det Q_\Gamma=(-1)^{n-1} \det(R_1,R_2),$$
where  $R_1$ and $R_2$ are the rays defining the moment cone, as in part (a).

The equality is obviously correct if $n=2.$ Therefore, assume $n\geq3.$
Denote by $Q_{j, \ldots,n}$ the intersection form of the linear plumbing $(s_j,  \ldots,s_n),$ for any $1<j \leq n-1,$ and $Q_n=s_n.$ Then
\begin{equation}
    \begin{split}\nonumber
       R_2&=A_2  \cdots A_{n-1}(-Q_n\ ,\ 1)\\
              &=A_2  \cdots A_{n-2}(Q_{n-1,n}\ ,\ -Q_n)\\
              &=A_2  \cdots A_{n-3}(-Q_{n-2,n-1,n}\ ,\ Q_{n-1,n})\\
              &\  \vdots         \\
              &=((-1)^{n-1}Q_{2, \ldots, n}\ ,\  (-1)^n Q_{3, \ldots, n}).
               \end{split}
\end{equation}
On the other hand
$$\det Q_\Gamma=s_1 \det Q_{2,\ldots, n}-\det Q_{3, \ldots, n}$$
and the equality follows.

From the classification results in \S\ref{section:lens spaces} it is clear that the determinant of the rays that span the moment cone of $L(k,l)$ is precisely equal to $k.$ Therefore, we derive the conclusion.

\end{proof}

We now focus on criteria for overtwistedness and tightness of the contact structure.

  \begin{theorem} \label{theorem ot} The contact structure on the  boundary of a linear plumbing $(s_1,\ldots,s_n)$, 
   is overtwisted if for some index $i\in\{1,\ldots,n\}$ where   $s_i\geq0$,  one of the following cases occurs:
\begin{itemize}
    \item[(a)] either $s_is_{i+1}\geq 2$ or  $s_is_{i+1}\geq1$ and $n>2$;
    \item[(b)] there exists  $j\in\{1, \ldots, n\}$ such that $|i-j|>1$, and for which either $s_j\geq1$ or $s_j\geq0$ and $n>3$;
  \item[(c)] $s_i=0$ and either $s_{i-1}+s_{i+1}\geq1$ or $s_{i-1}+s_{i+1}\geq0$ and $n>3$.     
\end{itemize}    
   \end{theorem}      
   
       \begin{proof}  The key idea is to  find a concave piece in the moment map image      such that the bounding rays form an angle strictly bigger than $\pi$ and then apply Theorem~\ref{proposition ot}.  We decompose our linear plumbing as in the proof of Theorem~\ref{theorem general}, cf. (\ref{sequence of L-shapes}), and observe certain L-shapes.

      \begin{itemize}
          \item[(a)]
According to Remark 
\ref{rem:rays of L-shape}, the moment cone corresponding to the contact toric boundary of the plumbing $(s_i,s_{i+1})$ is defined by the rays $R^i_1=(1,-s_i)$ and $R^i_{2}=(-s_{i+1},1).$ 
        These rays span a concave cone if and only if $\det(R^i_1,R^{i}_2)<0,$
        that is $s_is_{i+1}>1.$  Performing suitable $SL(2,\mathbb Z)$ transformations, in order to glue the whole space, the angle between the rays may not be preserved. However,  the determinant is preserved and, therefore, the cone corresponding to L-shape $(s_i,s_{i+1})$ remains concave  after any $SL(2,\mathbb Z)$ transformation. 
       
       If  $s_i=s_{i+1}=1$ then $R^i_1$ and $R^i_{2}$    span the angle $\pi$. In this case any additional  L-shape will contribute to obtain the cone whose rays span an angle greater than $\pi$.
       \item[]
      \item[(b)] Assume $i+1<j$. In the decomposition 
       $$\ldots(s_i,s_{i+1}),(0,s_{i+2}),\ldots, (0,s_j), \ldots, (0,s_n)$$  
       we consider the 
       L-shape corresponding to $(0,s_j)$ and glue it via $A_{j-1}$ to the L-shape corresponding to $(0,s_{j-1})$, if $i+1<j-1$, or to $(s_i,s_{i+1})$, if $i+1=j-1.$  We obtain a piece of the contact boundary whose cone is spanned by the rays
       $R_{j-1}=(1,0)$, if $i+1<j-1,$ or $R_{j-1}=(1,s_i)$, if $i+1=j-1,$ and $R_j=A_{j-1}(-s_j,1)=(s_{j-1}s_j-1,-s_j).$ In both cases the rays $R_{j-1}$ and $R_j$ span an angle greater than $\pi$ and, thus,  the corresponding contact structure is overtwisted. 
       
       If $s_j=0$ then above two rays span the angle $\pi$ and in this case an additional L-shape is needed to obtain the concave piece.
       \item[]
\item[(c)] We glue the L-shape corresponding to  $(0,s_{i+1})$ to the L-shape corresponding to $(s_{i-1},0)$ via $A_{i}.$ We obtain a piece of the contact boundary whose cone is spanned by the rays $R_{i-1}=(1,-s_{i-1})$ and $R_i=A_i(-s_{i+1},1)=(-1,-s_{i+1}).$ These rays span a concave cone if $s_{i-1}+s_{i+1}>0$. 
If $s_{i-1}+s_{i+1}=0$, then these rays span the angle $\pi$ and, as above, any additional L-shape will contribute to obtain greater angle.
\end{itemize}
       
        \end{proof}       

    \begin{figure}
        \centering
        \includegraphics[scale=1]{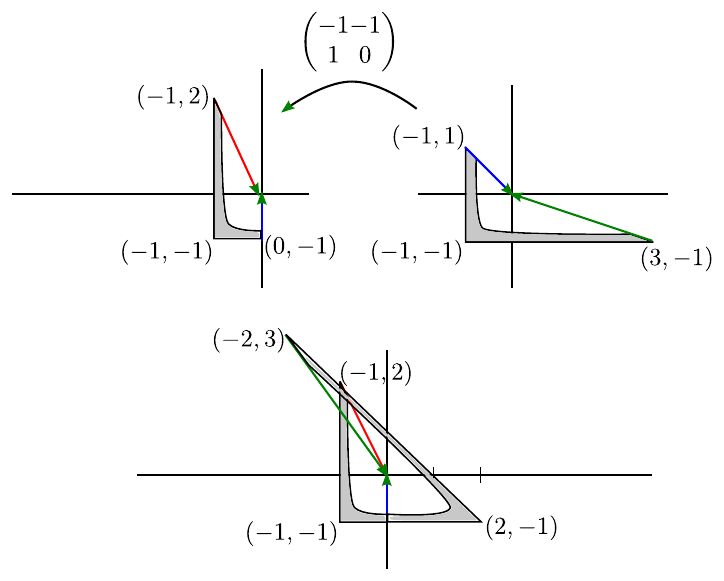}
        \caption{Example of $(2,1,3)$ plumbing.}
        \label{fig:ex213}
    \end{figure}
        
    \begin{example}\label{example ot but convex rays} The boundary of the linear plumbing $(2,1,3)$ is diffeomorphic to a lens space $L(1,2)   \cong S^3$ and the contact structure is overtwisted. Note that the rays that define the contact boundary are $R_1=(1,-2)$ and $R_2=\left[\begin{array}{cc} -1 & -1\\ 1 & 0 \end{array} \right]\left[\begin{array}{c} -3 \\ 1  \end{array} \right]=(2,-3)$, see Figure~\ref{fig:ex213}. These rays span a convex cone, however, the moment cone  over  the boundary of $(2,1,3)$ is the whole space $\mathbb R^2$ because the second ray $R_2$ corresponds to the angle greater than $2\pi.$ See Remark~\ref{rem:cone} for further details.
      \end{example}       
      
      We further show some criteria to obtain a tight concave contact structure on the  boundary of the linear plumbing. The proof is slightly more complicated than in the overtwisted case, as it is not enough to show that the rays that span the cone make a convex span. Namely, as in Example~\ref{example ot but convex rays}, it can happen that the rays of the moment cone span an angle less than $\pi$, but the moment cone is actually the whole space $\mathbb R^2$.  Therefore, we need to compare all the rays that span all L-shapes.

        \begin{theorem} \label{theorem tight} 
         The contact structure on the boundary of a linear plumbing $(s_1,\ldots,s_n)$,  where   $s_i\geq0$  for at least one $i\in\{1,\ldots,n\},$ is tight if one of the following cases occurs:   
         \begin{itemize}
             \item[(a)] $s_j\leq -2$, for all $j\neq i,$

 \item[(b)]$s_i=0$,  $s_{i-1}+s_{i+1} \leq-2$ and
$s_j \leq-2$, for all $j\neq i,i+1.$        \end{itemize}
        \end{theorem}

    \begin{proof}
    \begin{itemize}
        \item[(a)] If $n=2$ then the result follows from the fact that the rays $R^1_1=(1,-s_1)$ and $R^1_2=(-s_2,1)$ satisfy the condition $\det(R^1_1,R^1_2)>0$ and therefore span a convex cone.
    
    Let $n\geq3$.      
   In order to show that the contact structure is tight, it is enough to show that the rays  $R^1_1=(1,-s_1)$ and $R^n_2=A_2\cdots A_{n-1}(-s_n,1)$, span a convex cone for any $n\geq3,$ keeping track that they do not span an angle greater than $2\pi$.

Let $i$ be an index such that  $s_i\geq0$. 
 Suppose first $i<n.$
 We obtain the global toric structure by decomposing into L-shapes with pairs
           $$(s_1,0),\ldots, (s_{i-1},0),(s_i,s_{i+1}),(0,s_{i+2}), \ldots, (0,s_n).$$
        
    Observe that if $s\leq -2$, and $(x,y)$ is such that $x>y>0$, and
      $A = \left[\begin{array}{cc} -s&-1\\1&0 \end{array} \right]$
    then $$(x',y'):=A(x,y)\mbox{ satisfies }x'>y'>0.$$
  Denoting the second ray of the L-shape $L_{n-1}$ by $(x_{n-1},y_{n-1})=(-s_n,1)$ and denoting its transformation by $$(x_j,y_j):=A_{j+1}\cdots A_{n-1}(x_{n-1},y_{n-1}) \qquad \mbox{for} \qquad j=i,\dots, n-2,$$ we see inductively from this observation that $(x_j,y_j)$ satisfies $x_j>y_j>0$ for all $j\geq i$.

 If $i=1$ then the gluing is completed and the obtained cone is spanned by the rays $R^1_1=(1,-s_1)$ and $R^n_2=(x_1,y_1)$. Since $$\det((1,-s_1),(x_1,y_1))=y_1+s_1x_1>0$$ the cone is convex.

Suppose $i>1.$ Denote
 $$(x_{i-1},y_{i-1}) : =A_i(x_i,y_i) = (-s_ix_i-y_i,x_i).$$ Since $s_{i-1}\leq -2$, $s_i\geq 0$, and $x_i>y_i>0$ we have 
 $$\det((1,-s_{i-1}),(-s_ix_i-y_i,x_i))=
 x_i(1-s_{i-1}s_i)-s_{i-1}y_i>0.$$
    This tells us the cone is convex when we glue together $L_{i-1},\dots, L_{n-1}$.
 
 If $i=2$ we are done and if not we proceed  in the following way. Denote $$(x_k,y_k) : =A_{k+1}(x_{k+1},y_{k+1}),\mbox{ for all }1\leq k<i-1.$$ Using that $x_{i-1}<0$ and $y_{i-1}>0$, inductively, one can show that $x_k<y_k<0,$ for all $1\leq k<i-1$.  Note that $R^n_2=(x_1,y_1)$. Then $$\det((1,-s_1),(x_1,y_1))=y_1+s_1x_1>0$$ and the corresponding cone is convex. The case $i=n$ can be treated similarly.

\bigskip

Alternatively, the result can be proved using the blow ups of Hirzebruch surfaces, which were defined in \S\ref{section: symplectic 4}. We will also use this method of proof for part (b).
    
If $s_i\geq0$ and $s_j\leq -2$, for all $j\neq i,$ in the linear plumbing $(s_1,\ldots,s_n)$, then  we start from any Hirzebruch trapezoid given by inward normal vectors  $$(1,-s_i), \ (0,1), \ (-1,0), \ (0,-1).$$ Self-intersection numbers of the symplectic spheres corresponding to the edges defined by these inward normal vectors are $$0, \ s_i, \ 0, \ -s_i, \mbox{ respectively.}$$  In order to obtain 
 two edges with self-intersection numbers $s_{i-1}$ and $s_{i+1}$ that are adjacent to the edge with self-intersection number $s_i$
we  perform a suitable sequence of $s_{i-1}$ and $s_{i+1}$ small blow ups  in the following way:

We blow up the fixed point that corresponds to the vertex that is the intersection of edges defined by inward normal vectors  $(1,-s_i)$ and $(0,-1).$ Then, we blow up the fixed point that corresponds to the vertex that is the intersection of edges defined by inward normal vectors  $(1,-s_i)$ and $(1,-1-s_i)$.  We proceed in the same manner and  perform $s_{i-1}$ blow ups to a vertex of the edge defined by $(1,-s_i)$ and we obtain self-intersection number $s_{i-1}$ of this edge. 

Similarly, we perform a sequence of $s_{i+1}$ small blow ups to a vertex of the edge defined by inward normal vector $(-1,0)$. We proceed analogously  and obtain a chain of spheres with self-intersection numbers $s_1,\ldots, s_{i-1}$ on the left side of the edge $s_i$ and a chain of spheres with self-intersection numbers  $s_{i+1},\ldots, s_{n}$ on the right side of the edge $s_i$, so that the edges with self-intersection numbers     $s_1,\ldots,s_n$  are adjacent.  

We place this moment map image in $\mathbb{R}^2$ in such a way that the origin is the intersection point of two lines that contain two edges adjacent to the chain of edges with self-intersection numbers $s_1,\ldots,s_n$  (see Figure~\ref{(-2,1,-3)}). 
Recall that the arbitrary translation of the moment polytope does not change the  symplectic toric structure because a translation corresponds to adding a constant in $\mathbb R^2$ to the moment map  $(\mu_1,\mu_2)$ and, overall, Hamiltonian functions $\mu_1$ and $\mu_2$ are defined up to  a constant. 

Finally, take any embedded curve $\gamma$ in the moment map image that starts and ends at the edges that are adjacent to the chain of edges with self-intersection numbers $s_1,\ldots,s_n$ 
 and that is transversal to the radial vector field emanating from the origin. Then, $\gamma$ divides the moment polytope in two parts in such a way that the chain of edges with self-intersection numbers $s_1,\ldots,s_n$    is in the one part  and no other edge is in the same part.  Moreover, $\gamma$ corresponds to a concave contact toric boundary of the  symplectic toric manifold $(s_1,\ldots, s_n)$ whose moment map image is precisely the chosen part. 
        
       \begin{figure}
\centering
\includegraphics[width=9cm]{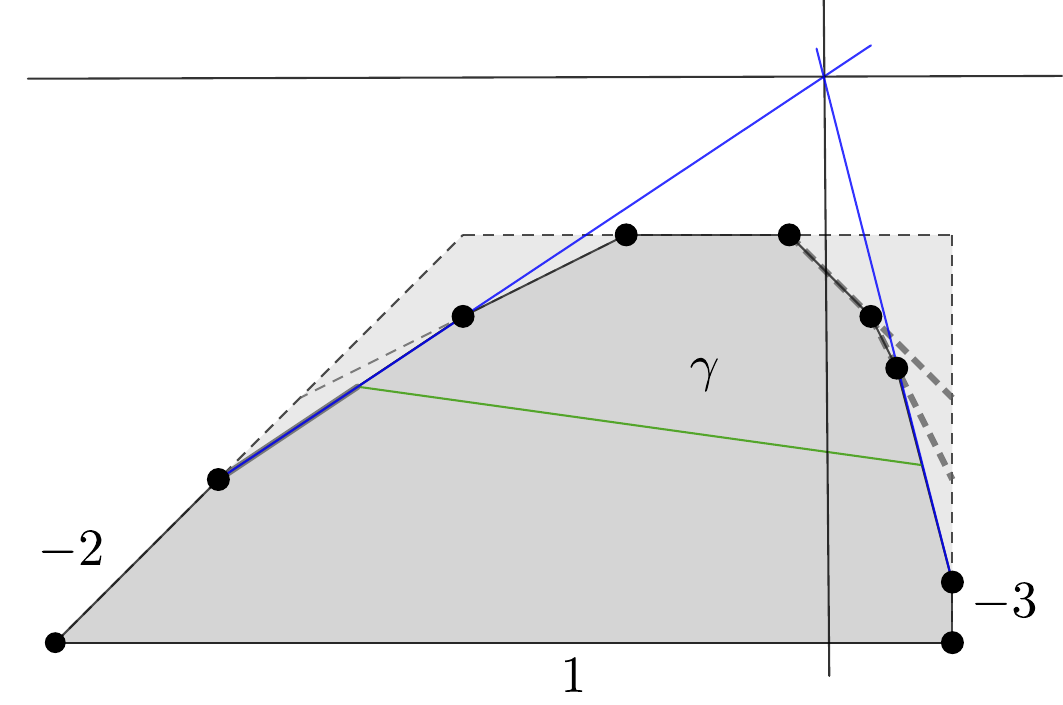}
\caption{A concave contact boundary of the plumbing $(-2,1,-3)$}
\label{(-2,1,-3)}
\end{figure}        

\item[]
   
\item[(b)]
Start from the Hirzebruch trapezoid with inward normal vectors $(1,-s_{i-1})$, $(0,1),$ $(-1,0),$ $(0,-1)$. 
     Self-intersection numbers of the symplectic spheres corresponding to the edges with these inward normal vectors  are $0$, $s_{i-1}$, $0$,  $-s_{i-1},$ respectively.     
      If $s_{i-1}\leq-2$ we are in the case a), therefore, assume $s_{i-1}\geq0$. We first perform a sequence of $|s_{i-1}+s_{i+1}|$ blow ups on the upper right corner to obtain a sphere with self-intersection number $s_{i+1}$.   We further  continue as above.
      \end{itemize}
    \end{proof}
    
    \begin{example} Consider the linear plumbing $(m,\underset{n}{\underbrace{-2,\ldots, -2}}),$ for any $m\geq0.$ The rays that define the moment cone of the contact structure on the boundary are $R_1=(1,-m)$ and
    $R_2=\left[\begin{array}{cc} 2& -1\\ 1 & 0 \end{array} \right]^n\left[\begin{array}{c} 2\\ 1 \end{array} \right]=(n+1,n).$ Therefore, the boundary is diffeomorphic to a lens space $L(mn+m+n,n+1)$ and the corresponding contact structure is tight, since the moment cone is strictly convex, for any $n\in \mathbb N.$
    
    \end{example}

   \begin{remark}\label{remark tight or ot} Theorem~\ref{theorem ot} and Theorem~\ref{theorem tight} cover  most of the cases of  linear plumbings $(s_1,\ldots,s_n)$,  where   $s_i\geq0$,  for at least one 
    $i\in\{1,\ldots,n\}.$ However, there is still an infinite list of  examples where the characterization is subtle due to the computational complexity of the associated linear algebra.  These are precisely the cases when  $$s_i=0 \mbox{ and } \ s_{i-1}+s_{i+1}=-1.$$  We can however work out the details in certain cases, which we explain below.  Our method can additionally be used when the  values of $s_{i-1}$ and $s_{i+1}$ are fixed.)
   
  \begin{enumerate}
      \item The contact boundary of 
 the linear plumbing $(-k,0,k-1)$ is diffeomorphic to $S^3$ with the  tight contact structure, for all $k\geq2$, as can be seen directly by gluing two  L-shapes $(-k,0)$ and $(0,k-1)$. The corresponding cone is spanned by the rays $R_1=(1,k)$ and $R_2=(-1,1-k)$ and, therefore, convex. 
  
\item The contact  boundary of a linear plumbing
 $$(-n,-k,0,k-1,-m_1, \ldots, -m_l), \hskip2mm n,\,k,\, m_i\geq 2,\, l\leq n-1$$
is tight, as the corresponding cone is convex. In particular, the contact boundary of
  
 $$(-n,-k,0,k-1,\underset{n-1}{\underbrace{-2, \ldots, -2}}), \hskip2mm n,k\geq 2,$$
  is diffeomorphic to $S^1\times S^2$ with the tight contact structure, as can be seen by computing the rays that span the corresponding cone. 
The first ray  is $R_1=(1,n)$ and the second ray is 
           $$R_2=\left[\begin{array}{cc} k& -1\\ 1 & 0 \end{array} \right]\left[\begin{array}{cc} 0 & -1\\ 1 & 0 \end{array} \right]\left[\begin{array}{cc} 1-k & -1\\ 1 & 0 \end{array} \right]
           \left[\begin{array}{cc} 2 & -1\\ 1 & 0 \end{array} \right]^{n-2}     \left[\begin{array}{c} 2\\ 1  \end{array} \right] =    \left[\begin{array}{c} -1\\ -n  \end{array} \right].            $$
           Thus, the cone is the half-plane. 
           
\item           
           By adding any L-shape to the construction in the preceding item (2) we obtain a concave cone and therefore the contact boundary of the linear plumbing 
       $$(\cdots, -n,-k,0,k-1,\underset{n-1}{\underbrace{-2, \ldots, -2}}, \cdots), \hskip2mm n,k\geq 2$$
       is overtwisted.
       
 \item Further, the contact boundary of the following linear plumbings

   $$(-n,-k,0,k-1,-m,-t),\hskip2mm n,k,t\geq2, m\geq3, $$
and
 $$(-n,-k,0,k-1,-m,\underset{t}{\underbrace{-2, \ldots, -2}}), \hskip2mm n,k,t\geq 2, m\geq3$$
   are also tight.
   
   \item However, the contact boundary of
    $$(-2,-n,-k,0,k-1,-3,\underset{t}{\underbrace{-2, \ldots, -2}}), \hskip2mm k,t\geq 2,$$
    is overtwisted for $n=2$, and tight for $n\geq3$.

  \end{enumerate}

         \end{remark}  
\begin{example}[Neumann calculus on plumbing graphs]
    Neumann describes a series of operations on plumbing graphs that allow to understand when the boundary of two plumbings is the same topological manifold. Among the operations, two apply to linear plumbing graphs, the ones named R1 and R3 in \cite{Neumann}. The first one, R1, allows to get rid of a vertex of the graph with self-intersection $\pm 1$. The $-1$ case does not change the contact structure on the boundary, as it corresponds to a symplectic blow down  of a $-1$ sphere. This is a reverse procedure of a symplectic blow up, described in \S\ref{section: symplectic 4}. From the toric perspective, there is no change of the inward normal vectors of the edges that are adjacent to the edge that corresponds to a $-1$ sphere, and, therefore, the contact boundary remains unchanged.    
    On the other hand, the $+1$ case changes the contact structure on the boundary as we describe below. The operation R3 allows to get rid of a self-intersection zero vertex, but also changes the contact structure on the boundary manifold. 

    In this example we compute how the contact structure changes assuming that the plumbing corresponds to a manifold with contact convex boundary. We do this by computing the bounding rays before and after each of the Neumann operations. These two operations  change the second bounding ray by multiplying it by $-1$ and reducing the angle between the two rays by $\pi$. In order to have, before and after, a contact concave manifold, the contact structure of the starting plumbing has to be overtwisted. 
    
    All linear plumbings we are considering have spheres as base of the disc  bundles. The fact that the boundaries of the linear plumbings below define the same 3-manifold, that is the same lens space, can be checked by computing the continued fraction in Proposition~\ref{prop lens}.
    \begin{enumerate}
        \item \emph{The operation R1 where the self-intersection 1 vertex is an endpoint.} This operation allows to change the linear plumbing  $(s_1,s_2,1)$ into the linear plumbing $(s_1,s_2-1)$. Following the computations above the bounding rays for the first case are $R_1=(1,-s_1)$ and 
        $$R_2=\begin{pmatrix}-s_2 & -1\\1 & 0
            \end{pmatrix}\begin{pmatrix}
                -1\\1
            \end{pmatrix}=\begin{pmatrix}
            s_2-1\\-1 
            \end{pmatrix}.$$ 
            In the second case the first ray $\widetilde{R}_1=R_1$ and $\widetilde{R}_2=(1-s_2, 1)=-R_2$.
        \item \emph{The operation R1 where the self-intersection 1 vertex has valency 2.} This operation changes the linear plumbing $(s_1,1,s_2)$ into $(s_1-1,s_2-1)$. For the first plumbing the bounding rays are $R_1=(1,-s_1)$ and
        $$R_2=\begin{pmatrix} -1 & -1\\1 & 0
        \end{pmatrix}\begin{pmatrix}-s_2\\1
        \end{pmatrix}=\begin{pmatrix} s_2-1\\-s_2
        \end{pmatrix}.$$
        For the second plumbing the rays are $\widetilde{R}_1=(1,1-s_1)$ and $\widetilde{R}_2=(1-s_2,1)$. The matrix
        $$B=\begin{pmatrix}
            1 & 0\\-1& 1
        \end{pmatrix}\in SL(2,\mathbb{Z}),$$
        satisfies that $B\widetilde{R}_1=R_1$ and $B\widetilde{R}_2=-R_2$.
        \item \emph{The R3 operation for a self-intersection zero vertex with valency 2.} This allows to change the linear plumbing $(s_1,s_2, 0, s_3, s_4)$ into the linear plumbing $(s_1, s_2+s_3, s_4)$. The bounding rays for the first plumbing are $R_1=(1,-s_1)$ and
        $$R_2=\begin{pmatrix}
            -s_2 & -1\\ 1 & 0
        \end{pmatrix}\begin{pmatrix}
            0 & -1\\1 & 0
        \end{pmatrix}\begin{pmatrix}
            -s_3 & -1 \\ 1 & 0
        \end{pmatrix}\begin{pmatrix}
            -s_4\\1
        \end{pmatrix}=\begin{pmatrix}
            s_2+s_3 & 1 \\ -1 & 0
        \end{pmatrix}\begin{pmatrix}
            -s_4\\1
        \end{pmatrix}.$$
        For the second plumbing we have that the first ray is $\widetilde{R}_1=R_1$ and 
        $$\widetilde{R}_2=\begin{pmatrix}
            -s_2-s_3 & -1 \\ 1 & 0
        \end{pmatrix}\begin{pmatrix}
            -s_4\\1
        \end{pmatrix}=-R_1.$$
    \end{enumerate}
    The three situations above change the second bounding  ray by a factor of $-1$. Theorem~\ref{proposition ot} implies that these operations can change an overtwisted contact into a tight contact structure.

Recent work~\cite[\S 3]{vanwhy2025nonaffinesteinmanifoldsnormal} studies Neumann's topological plumbing calculus with respect to the contact topology of the boundary.

\end{example}

\begin{remark}
 We remark that  the concave contact toric structures on the boundary of  $(s_1,\ldots,s_n)$ and $(s_n,\ldots,s_1)$ are equivariantly contactomorphic, meaning that there exists a contactomorphism between them that preserves the toric actions. The reason is that the corresponding moment cones differ by an $SL(2, \mathbb Z)$ transformation. Namely, the cones are defined by the rays $R_1=(1,-s_1), R_2=A_2 \cdots A_{n-1}(-s_n,1)$
 and
 $\tilde R_1=(1,-s_n),\tilde R_2= A_{n-1}\cdots A_2(-s_1,1)$. By making use of linear algebra one can find a desired transformation.

\end{remark}

\section{Background on embedded contact homology} \label{s:ECHbackground}
We now give a self-contained review of embedded contact homology (ECH), including the contact invariant and order of algebraic torsion in ECH.  
 \subsection{The ECH complex}
Let $(Y^3,\lambda)$ be a nondegenerate contact $3$-manifold and $\Gamma \in H_1(Y)$. Recall that a Reeb orbit is nondegenerate if $1$ is not an eigenvalue of the linearized Reeb flow. A contact form $\lambda$ is said to be \emph{nondegenerate} if all Reeb orbits are nondegenerate.
\begin{definition}

A \emph{Reeb current} is a finite set (potentially empty) of pairs $\alpha = \{ (\alpha_i, m_i) \}$, such that
\begin{itemize}
\itemsep-.15em
\item the $\alpha_i$ are distinct embedded Reeb closed orbits, 
\item the $m_i$ are positive integers,
\item the total homology class is given by $\sum_i m_i[\alpha_i] = \Gamma$.
\end{itemize}
We collect a few more remarks, terminology, and definitions below:
\begin{itemize}
\item A closed Reeb orbit $\gamma$ is said to be \emph{embedded} when $\gamma: \R/T\Z \to Y $ is injective, where $T$ is the period or action of the orbit, defined by $T=\mathcal{A}(\gamma):=\int_\gamma \lambda$. 
\item Embedded Reeb orbits are called simple orbits in some literature, but we will not do so as it can lead to serious confusion with respect to the definition of simple algebraic torsion in ECH, cf. Definition~\ref{def:simple_at}.
\item A Reeb current $\alpha = \{ (\alpha_i, m_i) \}$ is said to be an \emph{ECH generator} if  whenever $\alpha_i$ is hyperbolic $m_i=1$. 
\end{itemize}

\end{definition}

The \emph{embedded contact homology chain complex} is a $\Z/2\Z$ vector space freely generated by ECH generators. It splits as a direct sum over first homology classes $\Gamma \in H_1(Y)$ (restricting to the generators in that homology class).  Next we define the ECH differential. 

Let $(\R_r \times Y, d(e^r\lambda))$ be the \emph{symplectization} of $(Y, \lambda)$ and denote by $\pi_\R, \pi_Y$ the projections to each factor.

\begin{definition}
    
An almost complex structure $J$ is said to be a \emph{$\lambda$-compatible} almost complex structure when
\begin{itemize}
\item $J$ is $\mathbb{R}$-invariant,
\item $J(\partial_r) = R$ where $R$ is the Reeb vector field associated to $\lambda$, 
\item $J \xi = \xi $, positively rotating $\xi$ with respect to the orientation on $\xi$ given by $d\lambda$. 
\end{itemize}
\end{definition}

We consider $J$-holomorphic maps $u:(\Sigma, j) \to (\R \times Y, J),$ where $(\Sigma, j)$ is a punctured Riemann surface. The curve $u$ has a \emph{positive asymptotic end} at a closed Reeb orbit $\gamma$ of period $T$, if it admits coordinates in a neighborhood of a puncture $(s, t)\in [0,\infty) \times (\R/T\Z)$, where $T$ is the period of $\gamma$,  such that
\[
j(\partial_s) = \partial_t, \qquad \lim_{s\to \infty} \pi_\R(u(s, t)) = \infty,\qquad \mbox{and} \qquad \lim_{s \to \infty} \pi_Y(u(s, \cdot)) = \gamma.
\]
A \emph{negative asymptotic end} is defined analogously, where the neighborhood of a negative puncture is identified with a \emph{negative half-cylinder}, $s\in(-\infty, 0] \times (\R/T\Z)$ and $s \to -\infty$.   A \emph{$J$-holomorphic curve} is the equivalence class induced by biholomorphisms of a $J$-holomorphic map, which we also denote by $u$. Let $\mathcal{M}(\alpha, \beta, J)$ denote the moduli space of pseudoholomorphic curves in $(\R \times Y, d(e^r\lambda))$ with positive asymptotic ends at $\alpha$ and negative asymptotic ends at $\beta$.

The ECH differential is given in terms of the count of {$J$-holomorphic currents} of {ECH index} one as follows.  Let $\alpha$ and $\beta$ be Reeb currents in the class $\Gamma \in H_1(Y)$. A \emph{$J$-holomorphic current} is a finite set of pairs $\mathcal{C} = \{(u_k, d_k)\}$, where $u_k$ are distinct, connected, somewhere injective $J$-holomorphic curves in $(\R \times Y, d(e^r\lambda))$ and $d_k\in \Z^+$. We say that $\mathcal{C}$ is \emph{positively asymptotic to} a Reeb current $\alpha = \{(\alpha_i, m_i)\}$, if $u_k$ has positive ends at covers of $\alpha_i$ with multiplicity $cov(u_k)_i$ and $\sum_k cov(u_k)_i = m_i$. \emph{Negatively asymptotic} to a Reeb current is defined analogously. Let $\M(\alpha,\beta, J)$ denote the set of $J$-holomorphic currents from $\alpha$ to $\beta$.

The ECH index~\cite{index,revisit}, is used to guarantee that the currents counted by the ECH differential for generic $J$ consist of embedded connected $J$-holomorphic curves and unions of {trivial cylinders} $\R \times \gamma$ with multiplicities, \cite[Proposition~3.7]{lecture}.  In contrast to the Fredholm index, the ECH index is a purely topological index and provides an upper bound on the former \cite[Theorem~4.15]{revisit}, see \eqref{indineq}.  

To define the ECH index, we need to specify a relative homology class $Z \in H_2(Y,\alpha, \beta)$ as follows. Let $\alpha = \{(\alpha_i, m_i)\}$ and $\beta = \{(\beta_j, n_j)\}$ be ECH generators
with $[\alpha] = [\beta]\in H_1(Y)$, where $[\alpha] = \sum_i m_i [\alpha_i]$ and $[\beta] = \sum_j n_j [\beta_j]$. Then $H_2(Y, \alpha, \beta)$ denotes the set of relative homology classes of $2$-chains $Z$ in $Y$ such that
    $$\partial Z = \sum_i m_i\alpha_i - \sum_j n_j \beta_j.$$

\begin{definition}\label{def:I}
 The \emph{ECH index} of a pseudoholomorphic curve $u\in \mathcal{M}(\alpha, \beta, J)$ is defined as 
\begin{equation}\label{ECHindex}
I(u) = c_\tau (u) + Q_\tau(u) + CZ_\tau^I(\alpha,\beta),
\end{equation}
where $\tau$ is a trivialization of the contact structure $\ker \lambda$ along $\alpha$ and $\beta$. 
The first term $c_\tau(u)$ is the \emph{relative first Chern number} of the complex line bundle $\xi|_u$ with respect to $\tau$. The second term $Q_\tau(u)$ is the \emph{relative self-intersection number}, which is a signed count of intersections of $\image(u)$ with a transverse push-off where the push-off is trivial with respect to  $\tau$ at the asymptotic boundary.  The third term is a \emph{total Conley-Zehnder index term},
\[
CZ_\tau^I(\alpha,\beta):= \sum_i \sum_{k=1}^{m_i} CZ_\tau(\alpha^{k}_i) - \sum_j \sum_{k=1}^{n_j} CZ_\tau(\beta^{k}_j).
\]
\end{definition}
\begin{remark}
    The ECH index is well-defined and depends only on $\alpha$, $\beta$ and $[u]=:Z \in H_2(Y,\alpha,\beta)$. This is because $c_\tau$ and $Q_\tau$ both only depend on the relative homology class $Z$. See Lemma 2.5 and \S2.2 in \cite{index}. Therefore, in ECH literature, we are justified to write $I(\alpha, \beta, Z)$.
\end{remark}

The relative intersection number $Q_\tau$ can be similarly defined for a pair of relative homology classes. 
To relate the signed count of intersections between a pair of curves, if they do not frame each other according to the trivialization $\tau$, we need to account for additional ``intersections at infinity.'' This is the role of the asymptotic linking number.

\begin{definition}[Asymptotic linking]\label{def:asymp_linking}    
 To set up the definition of asymptotic linking between two pseudoholomorphic curves, we first define the linking number of two disjoint braids $\zeta$ and $\zeta'$ around an embedded Reeb orbit $\gamma$, which we denote by $$\ell_\tau(\zeta,\zeta') \in \mathbb{Z}.$$  Given a trivialization $\tau$ of $\xi$ along the Reeb orbit $\gamma$, the linking number $\ell_\tau(\zeta,\zeta')$ of two disjoint braids around $\gamma$ is defined to be the linking number of the oriented links $\phi_\tau(\zeta)$ and  $\phi_\tau(\zeta')$ in $\R^3$, where $\phi_\tau$ is an embedding of the tubular neighborhood $\mathcal{N}$ of $\gamma$ which has been identified via the trivialization $\tau$ with $S^1\times D^2$  and then identified with a solid torus in $\R^3$,  so that the projections of $\zeta$ and  $\zeta'$  to $S^1$ are submersions,  cf. \cite[\S 2.6]{revisit}.  The linking number is computed as, one half of the signed count of crossings of a strand of $\phi_\tau(\zeta)$ with a strand of $\phi_\tau(\zeta')$ in the projection to $\R^2\times\{0\}$.  The sign convention is that counterclockwise twists count positively, which differs from knot theory literature \cite[\S 3.1]{index}.

We define the \emph{asymptotic linking number} $l_{\tau}(u_1, u_2)$ between two curves $u_1, u_2$ that are asymptotic to $\gamma$ to be the linking number between the braids 
    \[\zeta_1^+ := \image(u_1) \cap \big(\{s\} \times \mathcal{N}(\gamma)\big)\quad \text{ and }\zeta_2^+ := \image(u_2) \cap \big(\{s\} \times \mathcal{N}(\gamma)\big)\quad\text{ for }s\gg0,\]
minus the linking number between 
\[\zeta_1^- := \image(u_1) \cap \big(\{s\} \times \mathcal{N}(\gamma)\big)\quad\text{ and }\quad \zeta_2^- := \image(u_2) \cap \big(\{s\} \times \mathcal{N}(\gamma)\big)\quad\text{ for }s\ll0.\]
The fact that the isotopy classes of such braids are well-defined is a consequence of \cite[Cor.~2.5 and 2.6]{Siefring}. 
\end{definition}

The following lemma is often useful for computing $Q_{\tau}$ and can be understood as an alternative definition of $Q_{\tau}$.
\begin{lemma}\cite[Lemma 8.5]{index}
\label{lem:Q_tau_lk_def}
Let $u_1, u_2$ be pseudoholomorphic curves. Let $[u_i] = Z_i \in H_2(Y; \alpha, \beta)$ for $i=1,2$. Then,
    \begin{equation}
    Q_{\tau} (Z_1, Z_2) + l_{\tau}(u_1, u_2) = \# (u_1 \cap u_2).
    \end{equation} 
\end{lemma}

The ECH index satisfies a number of properties, as established in \cite{index,revisit,lecture}.  We list a few of the important ones used in this paper.
\begin{remark}\label{rem:I}
 {As above, let $\alpha$, $\beta$ be ECH generators and $[u]=:Z \in H_2(Y,\alpha,\beta)$.}  When the Reeb currents $\alpha$ and $\beta$ are fixed, we denote $I(Z):=I(\alpha,\beta,Z)$. Let $\mathcal{M}(\alpha,\beta, J)$ be the set of $J$-holomorphic curves with positive ends at $\alpha$ and negative ends at $\beta$, without explicitly specifying the partition conditions of the asymptotic ends at the Reeb orbits. 
\begin{itemize}
\item As suggested by the notation,  the ECH index does not depend on the choice of the trivialization $\tau$.

\item(Index Inequality)  Suppose that $u \in \mathcal{M}(\alpha,\beta,J)$ is somewhere injective.  Then 
\begin{equation}\label{indineq}
  \mbox{ ind}(u) \leq I(u) - 2 \Delta(u),    
\end{equation}
where $\Delta(u)$ is a count of the singularities of $u$ in $\R \times Y$ with positive integer weights as in \cite[\S7]{mw} and $\mbox{ind}(u)$ is the Fredholm index of $u$.  
Equality holds only if $u$ is embedded and certain partition conditions hold on the ends of $u$, which fully determine the topological asymptotics at the ends of the curve $u$, cf. \cite[Theorem 4.15]{revisit}. 
\item(Additivity) If $\delta$ is another {Reeb current} 
in the homology class $[\alpha]$, and if $W \in H_2(Y,\beta,\delta)$, then $Z+W \in  H_2(Y,\alpha,\delta)$ and 
\[
I(Z+W)=I(Z)+I(W).
\]
\item(Index Parity) If $\alpha$ and $\beta$ are ECH generators, then 
\[(-1)^{I(Z)} = (-1)^{\varepsilon(\alpha)+\varepsilon(\beta)},\] \noindent where $\varepsilon(\alpha)$ is the number of positive asymptotic ends at positive hyperbolic orbits in $\alpha$ and  $\varepsilon(\beta)$ is the number of negative asymptotic ends at positive hyperbolic orbits of $\beta$. In particular, an ECH index one $J$-holomorphic curve, must be asymptotic to an odd number of positive hyperbolic orbits.

\item(Low ECH Index Classification) Let $J$ be generic and $\mathcal{C} \in \mathfrak{M}(\alpha,\beta,J)$, be any $J$-holomorphic current in $\R\times Y$, which is not necessarily assumed to be somewhere injective. Then:
\begin{itemize}
\item[$\star$] $I(\mathcal{C})\geq 0$, with equality if and only if $\mathcal{C}$ is a union of trivial cylinders with multiplicities.
\item[$\star$] If $I(\mathcal{C}) = 1$, then $\mathcal{C} = \mathcal{C}_0 \sqcup u_1$ where $I(\mathcal{C}_0) = 0$ and $u_1$ is an embedded $J$-holomorphic curve with $I(u_1) = \ind(u_1) = 1$.
\item[$\star$] If  $I(\mathcal{C}) = 2$, and if $\alpha$ and $\beta$ are ECH generators, then $\mathcal{C} = \mathcal{C}_0 \sqcup u_2$ where $I(\mathcal{C}_0) = 0$ and $u_2$ is an embedded $J$-holomorphic curve with $I(u_2) = \ind(u_2) = 2$.
\end{itemize}
\end{itemize}\end{remark}

The ECH index gives rise to a relative $\Z/d\Z$ grading of the chain complex and its homology, where $d$ denotes the divisibility of $c_1(\xi) + 2 \op{PD}(\Gamma) \in H^2(Y;\Z)$ mod torsion, {where $\op{PD}$ is a the Poincar\'e dual}. Given two  Reeb currents $\alpha$ and $\beta$, their index difference $I(\alpha,\beta)$ is defined by choosing an arbitrary $Z\in H_2(Y,\alpha,\beta) $ and setting 
\[
I(\alpha,\beta) = [I(\alpha,\beta,Z)] \in \Z/d\Z,
\]
which is well-defined by the index ambiguity formula \cite[\S3.3]{index}.

\begin{remark} 
There are a few more observations relating to the ECH index and gradings.  These are not strictly necessary for the purposes of this paper, but we collect them for clarity:
\begin{itemize}
\item As a result of the ECH index parity property, for every $\Gamma \in H_1(Y,\Z)$, there is a canonical absolute $\Z/2\Z$ grading given by the parity of the number of positive hyperbolic Reeb orbits \cite[\S3.3]{index}.
\item ECH admits an absolute $\Z$-grading if the class $c_1(\xi)+2PD(\Gamma)$ is torsion. 
    \item When $\Gamma=0$, it is common practice to implicitly upgrade the relative grading on  $ECH_*(Y,\xi,0)$ to a canonical absolute grading by mandating that $I(\emptyset)=0$.  
 
\end{itemize}
\end{remark}

{Given a contact structure,} let $\lambda$ be a nondegenerate contact form and let $J$ be a generic $\lambda$-compatible almost complex structure; we subsequently use the terminology \emph{ECH data} $(\lambda,J)$ to refer to such a pair. Let $\M(\alpha,\beta, J)$ be as before.
If $Z\in H_2(Y,\alpha, \beta)$, define
\[
\M^J(\alpha,\beta,Z) = \{ u \in \widetilde{\M}(\alpha,\beta, J) \ | \ [u]=Z\}/\R.
\]
{Every $J$-holomorphic curve $u$ can be assigned a sign $\epsilon(u) \in {\pm 1}$ depending on some orientation choices  explained in \cite[\S9]{obg2}. For the purposes of this paper (and all other known applications of ECH), it suffices to use the mod 2 count of $J$-holomorphic curves for the ECH differential coefficient, making the orientation choice irrelevant.  We therefore take the 
 ECH differential  on the ECH chain complex to be given by}
\begin{equation}\label{eq:d}
    \partial \alpha : = \sum_{\substack{\beta \mbox{\tiny \ ECH generator} \\ [\alpha]=[\beta]  }} \ \ \ \sum_{Z\in H_2(Y, \alpha, \beta)} \ \ \ \sum_{\substack{u\in\M^J(\alpha,\beta,Z) \\ I(u)=1 }} \epsilon(u) \beta \mod 2.
\end{equation}
By the additivity property of the ECH index, the differential decreases the  grading by 1.

That the ECH differential is well-defined and squares to zero was shown in \cite{obg1,obg2}, with the latter requiring  obstruction bundle gluing.  The homology of the ECH chain complex is the embedded contact homology defined with respect to $\Gamma \in H_1(Y,\Z)$,
\[
ECH_*(Y,\lambda,\Gamma,J) := H_*(ECC_*(Y,\lambda,\Gamma,J),\partial).
\]
The empty set of Reeb currents, $\emptyset$, is a generator of ECH and $\partial \emptyset =0$.


Taubes established an isomorphism between ECH and Seiberg-Witten-Floer homology, see Theorem~\ref{thm:taubes} which encapsulates \cite{Taubes_1, Taubes_2, Taubes_3, Taubes_4, Taubes_5}. Taubes' isomorphism combined with the work of Kronheimer and Mrowka \cite{KM_book} means that ECH does not depend on the choice of $J$ or $\lambda$ and is a well-defined ${\Z}[U]$-module, which we denote by $ECH_*(Y,\xi,\Gamma)$.

\subsubsection{\bf Action filtered ECH}\label{ss:action}

It will be useful to consider the action filtered ECH subcomplex, which is defined as follows. We can  associate to a Reeb orbit $\gamma$ its \emph{symplectic action} 
$$\mathcal{A}(\gamma) := \int_\gamma \lambda,$$
and to each ECH generator $\alpha = \{(\alpha_i, m_i)\}$ the action
$$\mathcal{A}(\alpha) := \sum_i m_i \mathcal{A}(\alpha_i).$$ 
By Stokes' theorem and the definition of a $\lambda$-compatible almost complex structure $J$, we know that the ECH differential decreases the symplectic action. Therefore, one may define the \emph{action filtered} embedded contact homology in terms of the subcomplex 
\[
ECH^L(Y,\lambda,\Gamma, J):=H_*(ECC^L(Y,\lambda, \Gamma, J), \partial^L)
\]which is generated by admissible Reeb currents $\alpha$ such that $\mathcal{A}(\alpha)<L$.  Taking $L=\infty$ gives the usual ECH. Action filtered ECH does not depend on the choice of $J$, as established in \cite{HT_Chord_II}, but does depend on the choice of the contact form $\lambda$.  

Additionally, $ECH^L(Y,\lambda, \Gamma, J)$ is well-defined as long as $\lambda$ is \emph{$L$-nondegenerate}, a weaker notion of nondegeneracy, which holds when all Reeb orbits of  $\lambda$ of action less than $L$ are nondegenerate, and there are no ECH generators of $\lambda$ of action exactly $L$.

We summarize a few more useful properties of action filtered ECH.  First, if $c>0$ is a constant then there is a canonical scaling isomorphism
\begin{equation}\label{e:scale}
ECH^L(Y,\lambda) \cong ECH^{cL}(Y,c\lambda).
\end{equation}
This is due to the fact that $\lambda$ and $c \lambda$ admit the same Reeb orbits up to reparametrization.  Moreover, for any generic $\lambda$-compatible almost complex structure $J$ there exists a unique $c\lambda$-compatible almost complex structure $J^c$ which agrees with $J$ on the contact planes.  Thus the bijection on Reeb orbits and the self-diffeomorphism of $\R \times Y$ sending $(s,y) \mapsto (cs,y)$ induces a bijection between $J$-holomorphic curves and $J^c$-holomorphic curves, yielding an isomorphism at the level of chain complexes:
\begin{equation}\label{eq:scale}
ECC^L(Y,\lambda, J) \cong ECC^{cL}(Y,c\lambda, J^c).
\end{equation}

Secondly, if $L\leq L'$ there are also maps induced by the inclusion of chain complexes: 
\begin{equation}\label{eq:incl}
\begin{split}
\iota: ECH^L(Y,\lambda) \longrightarrow & \ ECH(Y,\lambda), \\
\iota^{L,L'}: ECH^L(Y,\lambda) \longrightarrow & \ ECH^{L'}(Y,\lambda).
\end{split}
\end{equation}
None of the maps in \eqref{eq:scale} and \eqref{eq:incl} depend on $J$ as a result of \cite[Theorem~1.3]{HT_Chord_II}.

\subsubsection{\bf The ECH contact invariant}\label{subsec:ECH_contact_inv}
The {ECH contact invariant} is defined to be the homology class of the empty Reeb current and is denoted by $c(\xi):=[\emptyset]$. That $c(\xi)$ is an invariant (meaning that it does not depends on the choice of $\lambda$) follows from the following deep theorem of Taubes, which establishes the topological invariance of embedded contact homology (except that one needs to shift $\Gamma$ if one changes the contact structure), cf. \cite{Taubes_1, Taubes_2, Taubes_3, Taubes_4, Taubes_5}.

\begin{theorem}[Taubes]
\label{thm:taubes}
 If $Y$ is connected, then there is a canonical isomorphism of relatively graded $\Z/2\Z[U]$-modules 
\[
ECH_*(Y,\lambda,\Gamma,J) \simeq \widehat{HM}^{-*}(Y,\mathfrak{s}_\xi + {PD}(\Gamma)),
\]
which sends the ECH contact invariant $c(\xi):=[\emptyset] \in ECH(Y,\xi,0)$ to the contact invariant in Seiberg-Witten Floer (monopole) cohomology.
\end{theorem}

Here $\widehat{HM}^*$ denotes the ``from'' version of Seiberg-Witten Floer (monopole) cohomology, which is fully explained in the book by Kronheimer and Mrowka \cite{KM_book}, and $\mathfrak{s}_\xi$ denotes the canonical spin-c structure determined by the oriented 2-plane field $\xi$.  The contact invariant in Seiberg-Witten Floer (monopole) cohomology was implicitly defined (up to sign) in \cite{Kronheimer_Mrowka_contact_str}. 

\begin{remark}\label{rem:c}
We collect some of the important results relating to the (non)vanishing of the contact invariant:
\begin{enumerate}
\item If $(Y,\xi)$ is overtwisted, then the ECH contact invariant vanishes by the existence of an embedded positive hyperbolic Reeb orbit $\gamma$ which has smallest action amongst all Reeb orbits and a unique Fredholm and ECH index 1 embedded J-holomorphic plane $u$ in $\R \times Y$ with positive end at $\gamma$ by \cite[Appendix by Eliashberg]{Yau_contact_ot}. 
\item If $(Y,\xi)$ is exactly symplectically fillable then the ECH contact invariant is nonvanishing by \cite{echcap, HT_Chord_II}. Additionally, Hutchings' work in progress shows that if $(Y,\xi)$ is strongly symplectically fillable then $c(\xi)\neq 0$ by way of more refined ECH cobordism maps induced by Seiberg-Witten-Floer (monopole) homology \cite{Hutchings_ech_sft}.
\item There are analogous definitions and nonvanishing results of the contact invariant in Heegaard-Floer homology and Seiberg-Witten Floer (monopole) homology; see \cite{Kronheimer_Mrowka_contact_str, Ozsvath_Szabo_contact_str}.
\item Applications of the monopole contact invariant $c(\xi) \in \widehat{HM}^{-*}(Y,\mathfrak{s}_\xi + {PD}(\Gamma))$ necessitate a naturality property for the map induced by a strong symplectic cobordism between two contact manifolds.  Echeverria gave a proof of this naturality property, which for technical reasons was done with mod 2 coefficients \cite{Echeverria}.
\end{enumerate}
\end{remark}

\subsection{Algebraic torsion in ECH}\label{subsec:AT_in_ECH}
It is of interest to find obstructions to the existence of symplectic fillings and exact symplectic cobordisms as well as to determine if a given nonfillable contact structure is tight.  The ECH contact invariant is not ideally suited to these purposes, and given ECH data $(\lambda,J)$, it is helpful to utilize an integer known as the (order of) algebraic torsion, which is derived from a filtration on the topological complexity of the ECH index one curves that have no negative ends \cite[Appendix by Hutchings]{at}.  The definition of ECH algebraic torsion was inspired by the work of Latschev and Wendl in the context of symplectic field theory \cite{at}\footnote{Symplectic field theory is defined in terms of a very different count of $J$-holomorphic curves, and associates a $BV_\infty$ algebra to a contact manifold.}.

The definition of ECH algebraic torsion utilizes a topological index $\Jplus$ similar to the ECH index.   However, $\Jplus$ more directly captures the topological complexity of the curves counted by the ECH differential and gives rise to a relative filtration on ECH.  The index $\Jplus$ is defined in terms of a related index, $\Jzero$; both enjoy many of the same properties of the ECH index and directly give an absolute grading on the homotopy class of oriented 2-plane fields \cite[Proposition~6.5]{revisit}.  They are defined as follows.

\begin{definition}
Let $\alpha$ and $\beta$ be arbitrary Reeb currents such that $[\alpha] =[\beta]$ and $Z\in H_2(Y,\alpha,\beta)$.  Define
\[
\Jzero(\alpha,\beta,Z):=-c_\tau(Z) + Q_\tau(Z)+ CZ_\tau^{\mathbb{J}}(\alpha,\beta),
\] 
where 
\[
CZ_\tau^{\mathbb{J}}(\alpha,\beta) = \sum_i \sum_{k=1}^{m_i-1}CZ_\tau(\alpha_i^k) -\sum_j \sum_{k=1}^{n_j-1}CZ_\tau(\beta_j^k). 
\]
\end{definition}
The differences between the definition of $I$ in Definition~\ref{def:I} and $\Jzero$ are that the sign of the relative Chern class term is switched and that the Conley-Zehnder terms are slightly different, in that we do not sum to the last respective $m_i$ and $n_j$ iterates of $\alpha_i$ and $\beta_j$.

From $\Jzero$ we can define the variant $\Jplus$ by
\[
\Jplus(\alpha,\beta,Z):= \Jzero(\alpha,\beta,Z) + |\alpha| - |\beta|,
\]
where given an arbitrary Reeb current $\gamma = \{ (\gamma_i,m_i)\}$ we define
\[
|\gamma| := \sum_i \left\{ 
\begin{array}{rc} 1, & \gamma_i \mbox{ elliptic,} \\ 
m_i, & \gamma_i \mbox{ positive hyperbolic,} \\ 
\lceil \frac{m_i}{2} \rceil, & \gamma_i \mbox{ negative hyperbolic.} \\ 
\end{array}\right.
\]
Note that in the above, if $\gamma$ is an ECH generator 
then when $\gamma_i$ is hyperbolic we have that $m_i = \lceil \frac{m_i}{2} \rceil =1 $ and {thus} $|\gamma|$ is the multiplicity of the embedded orbits associated with the Reeb current $\gamma$.

The following result is the analogue of the ECH index inequality for $\Jzero$, which incorporates the genus of a $J$-holomorphic curve and provides a bound on the (negative) Euler characteristic. 

\begin{proposition}{\em\cite[Proposition~6.9, Theorem~6.6]{revisit}}\label{prop:J0}
Let $(\lambda,J)$ be ECH data and let $\alpha$ and $\beta$ be arbitrary Reeb currents. 
Suppose $u \in \mathcal{M}(\alpha,\beta,J)$ is a somewhere injective and irreducible $J$-holomorphic curve {of genus $g$}.  Then $\Jplus(u) \geq 0$ and
\begin{equation}\label{eq:J0irr}
\Jplus(u) \geq 2 \left(g-1 + \Delta( u) + \sum_i \left\{ \begin{array}{ll}n_{\alpha_i}  & {\alpha_i}  \mbox{ elliptic,} \\ 1  & {\alpha_i}  \mbox { hyperbolic,}\\ \end{array}\right. - \sum_j \left\{ \begin{array}{ll}n_{\beta_j} -1  & {\beta_j}  \mbox{ elliptic,} \\ 0  & {\beta_j}  \mbox { hyperbolic,} \\ \end{array}\right.\right)
\end{equation}

If $\op{ind}(u) = I(u)$, e.g. $u$ is an irreducible (connected) curve in $\R \times Y$ contributing to the ECH differential, which does not contain trivial cylinders,  then \eqref{eq:J0irr} is sharp with $\Delta(u)=0$. Here the sum is over all embedded Reeb orbits $\alpha_i$ (respectively, $\beta_i)$ in $Y$ at which $u$ has positive (respectively, negative) ends in $Y$; $\Delta(u)$ is a count of the singularities of $u$ in $\R \times Y$ with positive integer
weights; $n_\gamma$ denotes the number of ends of $u$ at covers of $\gamma$.

\end{proposition}

\begin{remark} \label{remark J+}
Using various properties of the ECH index $I$ and $\Jplus$ we can conclude the following: 
 \begin{itemize}
    \item  The parity of 
     \[I(u)-\Jplus(u) = 2c_\tau(u) + \sum_i CZ_\tau(\alpha_i^{m_i}) - \sum_j CZ_\tau(\beta_j^{n_j}) - |\alpha| + |\beta| \] 
     agrees with the parity of the number of positive hyperbolic Reeb orbits, which agrees with the parity of $I(u)$ by Remark~\ref{rem:I} (Index Parity). 
     \item  Thus if a $J$-holomorphic curve $u$ contributes to the ECH differential then $\Jplus(u)$ is even.  
     \item The index $\Jplus$ is additive under gluing \cite[Proposition 6.5(a)]{index}: If $Z\in H_2(Y,\alpha,\beta)$ and  $W\in H_2(Y,\beta,\gamma)$ then \[\Jplus(\alpha,\beta,Z)+\Jplus(\beta,\gamma,W)=\Jplus(\alpha,\gamma,Z+W)\]
 \end{itemize}   

\end{remark}

This allows us to define the following $\Jplus$ spectral sequence, which we use to define the order of algebraic torsion.

\begin{definition}[$\Jplus$ spectral sequence]
As a result of Remark~\ref{remark J+}, we have a splitting of the ECH differential $\partial$ into
\[
\partial = \partial_0 + \partial_1 + \cdots,
\]
where $\partial_\ell$ denotes the contribution from $J$-holomorphic curves with $$\Jplus(u) = 2\ell.$$  Since the ECH differential squares to zero and by the additivity property for $\Jplus$,  we have that $\partial_0^2=0$, $\partial_0\partial_1 + \partial_1\partial_0 =0$, and so on. 

This gives rise to the ``$\Jplus$ spectral sequence" $E^*(Y,\lambda,J)$ where $E^1$ is the homology with respect to $\partial_0$, $E^2$ is the homology of $\partial_1$ acting on $E^1$, and so on. 

One can similarly define a $\Jplus$ spectral sequence $E^{L,*}(Y,\lambda,J)$ for the subcomplex $ECC^L(Y,\lambda,J)$ given an action $L>0$.
\end{definition}
For better or for worse, this spectral sequence is not invariant under deformation of the contact form because we do not have control over the $\Jplus$ index of multiply covered curves in exact symplectic cobordisms and negative $\Jplus$  curves may contribute to the chain map. However, by restricting to the so-called ``simple" subcomplex of ECH, detailed further below, we will be able to obstruct the existence of multiple covered curves in the corresponding exact symplectic cobordisms, which provides a weak naturality property for the map induced by a strong symplectic cobordism. This in turn permits the desired applications: obstructions of symplectic fillings and obstructions of exact symplectic cobordisms.

The $\Jplus$ spectral sequence yields the following definition of the order of algebraic $k$-torsion, associated to ECH data $(\lambda,J)$, as follows, summarizing the highlights of \cite[Appendix by Hutchings]{at}.

\begin{definition}[Order of algebraic torsion]
Given $L\in (0,\infty]$, we define the \emph{order of algebraic torsion} $\at^L(Y,\lambda,J)$ to be the smallest nonnegative integer $k$ such that $\emptyset$ becomes zero in the $E^{k+1}$ page of the spectral sequence $E^{L,*}(Y,\lambda,J)$. Note that $\at^L(Y,\lambda,J)$  can be defined for an $L$-nondegenerate contact form $\lambda$ and a generic  $\lambda$-compatible $J$.   If no such $k$ exists, we define $\at^L(Y,\lambda,J)=\infty$. Define $\at(Y,\lambda,J):=\at^\infty(Y,\lambda,J)$, assuming that $(\lambda, J)$ comprise ECH data. 
\end{definition}

We want to underscore that at present, ECH algebraic torsion should be thought of as a measurement of nonfillability rather than a refinement of the contact invariant. This is because it is not yet clear how, if it is possible, to establish the following naturality property: if there is an exact symplectic cobordism from $(Y_+,\lambda_+)$ to $(Y_-,\lambda_-)$ then
\begin{equation}\label{eq:natural}
    \at(Y_+,\lambda_+) \geq \at(Y_-,\lambda_-).
\end{equation}
Should \eqref{eq:natural} hold, then it would follow that the order of algebraic torsion  depends only on $\xi$ and is monotone with respect to exact symplectic cobordisms.

\begin{remark}
We summarize some observations regarding the order of algebraic torsion.
\begin{itemize}

\item If $\at(Y,\lambda,J)= \infty$ for some ECH data $(\lambda, J)$ then $c(\xi)\neq 0$.
    \item A priori, $c(\xi)\neq 0$ does not necessarily imply $\at=\infty$.  While we can conclude that the `total' ECH differential does not kill the generator $\emptyset$, we do not know how the expression of the total differential appears with respect to the pages of the spectral sequence. For example, $\at=1$ means there is a pseudoholomorphic curve of $\Jplus$ index 1 killing $\emptyset$, so $\partial_1(\alpha) = \emptyset$, where $\alpha$ is an orbit set,  but there could still be a $\Jplus$ curve of higher index $k$, with $\partial_k(\alpha) = [\beta] $ e.g. $\partial(\alpha) = \emptyset + \beta$.  In this setting we could thus have $c(\xi)\neq0$.
       \item Proposition~\ref{thm:hutchingsot} below states that $\at=0$ for overtwisted contact manifolds. The converse is an open conjecture. Even weaker, it is not known that $\at=0$ implies $c(\xi)=0.$
\item One needs additional information about the complex so that it is possible to directly conclude from the chain complex whether or not the contact invariant is nonvanishing.  We are able to deduce this as a result of our analysis of pseudoholomorphic curves in our setting.
\item    In particular, if there is a unique pseudoholomorphic curve which is a plane killing the generator $\emptyset$, as in our case and we know enough about the chain complex, we can establish that $c(\xi)=0.$
\item Relying on the Morse-Bott methods of \cite{HutchingsSullivanT3, yao2022cascades}, it can be shown that Morse-Bott ECH yields $$\at(T^3, \cos (2n \pi t)dq_1+\sin (2n \pi t) dq_2),J) = 1 \mbox{ \ for all }  n \in\mathbb \Z_{\geq2}.$$ These are the $T^3$ admitting Giroux torsion; the associated $J$-holomorphic curve is a cylinder with two positive ends one on `$e$' and the other on `$h$', which is obtained from ``rounding the corner'' on `$e$'.

\end{itemize}

\end{remark}

To establish a weaker naturality property, which we use to obstruct fillability,  we need the following definition of \emph{simple algebraic torsion}, as explained in \cite[Appendix by Hutchings]{at}.

\begin{definition}
\label{def:simple_at}
    An ECH generator $\alpha = \{(\alpha_i, m_i)\}$ is a \emph{simple generator} with respect to $J$ when the following conditions hold:
\begin{itemize}
    \item  $m_i = 1$ for all $i$;
    \item For any $\beta = \{(\beta_j, n_j)\}$ at the negative end of a (possibly broken) $J$-holomorphic curve with positive end at $\alpha$, all $n_j = 1$. 
\end{itemize}    
    Given $L\in (0,\infty]$, define $ECC^L_{\simp}(Y,\lambda,J)$ to be the subcomplex $ECC(Y,\lambda,J)$ generated by simple {ECH generators} 
    of action $<L$.
\end{definition}

It is important to note that this notion of simple does not agree with the simple terminology in symplectic field theory literature. As demonstrated in \cite[Lemma A.14]{at}, the second condition in Definition~\ref{def:simple_at} enables to define a chain map induced by an exact symplectic cobordism from $(Y_+, \lambda_+)$ to $(Y_-, \lambda_-)$,
\[
\Phi: ECC^L_{\simp}(Y_+,\lambda_+,J_+) \to ECC(Y_-,\lambda_-,J_-)
\]
such that 
\begin{enumerate}
    \item $\Phi(\emptyset) = \emptyset$;
    \item there is a decomposition $\Phi = \Phi_0 +  \Phi_1 + ... $ such that
    $\sum_{i+j=k}(\partial_i\Phi_j - \Phi_i \partial_j)=0$  for each $k \in \Z_{\geq0}$.
\end{enumerate}
In particular:
\begin{itemize}
    \item[] If the orbit set $\alpha_+$ is simple, no holomorphic curve in $\mathcal{M}^J(\alpha_+,\alpha_-)$ admits a multiply covered component and no broken holomorphic curve arising as a limit of a sequence of curves in $\mathcal{M}^J(\alpha_+,\alpha_-)$  admits a multiply covered component in the cobordism level.
\end{itemize}
This ensures that
\begin{itemize}
    \item If $\alpha_+$ is simple and if $u\in\mathcal{M}^J(\alpha_+,\alpha_-)$ has $I(u)=0$ then $u$ is regular.
    \item $\Phi$ is well-defined because if  $\alpha_+$ is simple then  there are only finitely many curves with $I(u)=0$ and $\Jplus(u)\geq0$ for all  $u\in\mathcal{M}^J(\alpha_+,\alpha_-)$.
\end{itemize}    

In practice, it can be difficult to ascertain which ECH generators comprise the simple subcomplex of ECH if the differential does not vanish by grading reasons.  In the setting under consideration in this paper, we utilize asymptotic analysis of the expansion of $J$-holomorphic curves to understand the simple subcomplex in \S\ref{sec:asymptotic}.

\begin{definition}[Simple algebraic torsion]
    Given $L\in (0, \infty]$, we define $\at^L_{\simp}$ to be the smallest nonnegative integer $k$ such that $\emptyset$ becomes zero in the $E^{k+1}$ page of the spectral sequence $E^{L,*}(Y,\lambda,J)$ when we restrict to simple orbit sets. Again, if no such $k$ exists, we define $\at^L_{\simp}(Y,\lambda,J)=\infty$. Define $\at_{\simp}(Y,\lambda,J):=\at^{\infty}_{\simp}(Y,\lambda,J)$.
\end{definition}

The following theorem relates $\at_{\simp}$ and $\at$ in an exact cobordism and is needed for the computations in \textsection\ref{section computing}. Recall that $(\lambda, J)$ is said to be ECH data if $\lambda$ is nondegenerate and $J$ is a generic $\lambda$-compatible almost complex structure.

\begin{theorem}[{\cite[Theorem A.9]{at}}]
\label{thm:hutchings_A9}
    Suppose there is an exact symplectic cobordism from $(Y_+, \lambda_+)$ to $(Y_-, \lambda_-)$\footnote{Note that Hutchings uses the opposite convention for the direction of symplectic cobordism even of the main text of \cite{at}. We will use Hutchings' convention in this paper.}. Then for any choices of ECH data $(\lambda_\pm, J_\pm)$,
    $$\at^L_{\simp}(Y_+, \lambda_+, J_+) \geq \at^L(Y_-, \lambda_-, J_-)$$
    for each $L\in (0, \infty]$.
\end{theorem}

That the order of algebraic torsion is zero for any overtwisted nondegenerate contact form $\lambda$ and generic $\lambda$-compatible $J$ follows from the reasoning given in Remark~\ref{rem:c} (1) in conjunction with the weak naturality property above.

\begin{proposition}[{\cite[Corollary A.11]{at}}]
    \label{thm:hutchingsot}
    Suppose $(Y,\xi)$ is overtwisted. Then $\at(Y,\lambda,J)=0$ whenever $(\lambda,J)$ is ECH data for $(Y,\xi)$.
\end{proposition}

We will use the following observation (similar to \cite[Corollary A.10 ]{at}) to obstruct when $(Y, \xi)$ is not exactly fillable in \S\ref{section computing}.  Since, by Proposition~\ref{prop lens} or \cite{Neumann}, the manifolds we consider are all known to be lens spaces, we will be able to deduce that non-fillability implies overtwistedness in the setting at hand.

\begin{lemma} \label{lemma fillable at}
    If $(Y,\xi)$ is exactly fillable, then $\at_{\simp}(Y,\lambda, J) = \infty$ for any ECH data $(\lambda,J)$ where $\ker \lambda = \xi$.
\end{lemma}
\begin{proof}
    First, observe that $\at^L(S^3, \lambda_{\text{irr}}, J_{\text{irr}}) = \infty$ for any  irrational ellipsoid contact form $\lambda_{\text{irr}}$ and any $\lambda_{\text{irr}}$-compatible $J_{\text{irr}}$. 
    Here, by irrational ellipsoid we mean to take $S^3 \cong \partial E(a,b)$ where 
    $$\partial E(a,b) = \left\{ (z_1,z_2) \in \C^2 \ | \ \pi \left( \frac{|z_1|^2}{a} + \frac{|z_2|^2}{b} \right)  =1 \right\},$$ with $a,b \geq 0$ and $a/b \in \R \setminus \Q$.   Then for $$\lambda_{\text{irr}}= \frac{i}{2}\left( \sum_{j=1}^2 z_j d\bar{z}_j - \bar{z}_jd z_j \right)$$ restricted to $ \partial E(a,b)$, the Reeb vector field is given by
\[
R_{\text{irr}} = 2\pi \left( \frac{1}{a}\frac{\partial}{\partial \theta_1} + \frac{1}{b}\frac{\partial}{\partial \theta_2}  \right).
\]
Since $a/b \in \R \setminus \Q$, there are exactly two embedded elliptic Reeb orbits, which are given by $\gamma_1$ in the $z_2=0$ plane with action $a$ and $\gamma_2$ in the $z_1=0$ plane with action $b$. Thus there are no generators with ECH index one by the Index Parity Property in Remark~\ref{rem:I}, hence $\at^L(S^3, \lambda_{\text{irr}}, J_{\text{irr}}) = \infty$.
   
     Now if $(Y, \xi)$ is exactly fillable, we know that for any $\lambda$ such that $\ker \lambda = \xi$, there is an exact cobordism from $(Y,e^c\lambda, J^c)$ to $(S^3, \lambda_{\text{irr}}, J_{\text{irr}})$ for a large positive $c$. Now by the ``scaling'' isomorphism  at the level of chain complexes induced by the bijection on Reeb orbits 
\eqref{eq:scale}, in combination with Theorem~\ref{thm:hutchings_A9}, we have that
    $$\at^{e^{-c}L}_{\simp}(Y, \lambda, J) = \at^L_{\simp}(Y, e^c\lambda, J^c)\geq \infty,$$
    where $L$ is arbitrary.
\end{proof}

\section{Algebraic torsion of contact toric manifolds with non-free action}
\label{section computing}
In this section we give our main ECH direct computation results. The more involved cases will rely on delicate arguments about pseudoholomorphic curves which we prove in \S\ref{section plane exists},~\ref{section constraints}, and~\ref{sec:asymptotic}. While we were able to analyze the geometric properties of almost toric contact structures directly to determine overtwistedness/fillability in \S\ref{section tight-ot}, in the remainder of the article, we wish to develop tools to calculate the ECH contact invariant and algebraic torsion directly and explicitly. Our motivation is to be able to extend these tools to more general examples which have portions that can be modeled with toric contact structures, but do not admit a global almost toric structure. As we will see, our analysis of pseudoholomorphic curves will apply to contact manifolds where a co-dimension $0$ submanifold with boundary has a contact toric structure satisfying certain properties. In this article, we will apply these results in the case of globally toric contact manifolds.

The main statement that we will prove explicitly is the following.

\begin{theorem} 
\label{theorem at}
 Let $(Y^3,\xi)$ be a compact connected contact toric manifold characterized by two real numbers $t_1,t_2$, which define the corresponding moment cone as in Remark~\ref{rem:cone}.  
 \begin{enumerate}[label=(\alph*)]
  \item \label{case less pi} If $t_2-t_1< \pi$ then
$c(\xi) \neq 0$ and $\at_{\simp}(Y, \lambda, J)= \infty$;
  
     \item \label{case more pi}
if  $t_2-t_1> \pi$ then $c(\xi)=0$ and $\at(Y, \lambda,J)=0$;   

  \item \label{case equal pi} if $t_2-t_1= \pi$ then
 $\at_{\simp}(Y, \lambda, J)>0$,
 \end{enumerate} 
    for all ECH data $(\lambda, J).$ 
\end{theorem}

Recall that if a contact manifold is exactly fillable then $c(\xi)\neq 0$ and $\at_{\simp} = \infty$, whereas any overtwisted contact manifold satisfies $c(\xi)=0$ and $\at = 0$. The results are what we expect in light of Theorems~\ref{theorem ot} and \ref{theorem tight}. The purpose of the remainder of the article is to develop the tools to detect these computations directly and independently, as a test case for more general examples.

\vskip2mm

Let $\lambda$ be an invariant contact form and $\mu_{\lambda}:Y\to\mathbb R^2$  the corresponding moment map defined in \S\ref{sec:contact_toric_manifolds}.

\begin{notation}\label{notation:Y_interval}

Let $P:[0,1]\rightarrow\R^2$ be a parametrization of the image $\mu_{
\lambda
}(Y)$. 
 We denote by $T_x:=\mu^{-1}(P(x))$ the toric fiber over $P(x)$ and 
$$Y_{[x_-,x_+]}=\{ y\in Y\mid \mu(y)=P(x) \textrm{ for some }x\in [x_-,x_+]\} $$
for an interval of toric fibers. We similarly define $Y_{(x_-,x_+)}$, $Y_{[x_-,x_+)}$, and $Y_{(x_-,x_+]},$ to include or exclude the endpoint fibers.
\end{notation}

Recall also that the moment map image determines an associated invariant contact form on the free toric part of $Y$. If $P(x)=\big(p_1(x),p_2(x)\big)$, then $\lambda = p_1(x)dq_1+p_2(x)dq_2$. 
Therefore, by choosing particular curve $P(x)$ we also have unique corresponding invariant contact form.
In each case of Theorem~\ref{theorem at}, we will take a curve  satisfying the following properties. 

\begin{setup}\label{setup well defined}
    Let $\lambda=p_1(x)dq_1+p_2(x)dq_2$ be an invariant contact form on a co-dimension zero submanifold (potentially with boundary) of a contact manifold $(Y^3,\xi)$, corresponding to the moment map image curve $P(x)=(p_1(x),p_2(x))$. In this set-up, we require the following properties of  $P$:
    \begin{enumerate}
        \item \label{item graphical} $P$ is everywhere positively transverse to the radial rays from the origin, namely $$\Big(\langle p_1(x),p_2(x) \rangle\ ,\  \langle p_1'(x),p_2'(x)\rangle \Big)\quad \text{is a positive basis for $\R^2$ for all $x\in [0,1]$}.$$
        Equivalently, $Q(x):=p_1(x)p_2'(x)-p_2(x)p_1'(x)>0$ for all $x\in (0,1)$.
        \item \label{item start} {There are rays $R_1$ and $R_2$ such that} $P(0)\in R_1$, and {either} $P(1)\in R_2$ or the fiber over $P(1)$ is the boundary of the co-dimension zero submanifold.
        \item \label{item singfibers} Near $0$ and $1$, $P$ is linear with irrational slope and $P'(0)=P'(1)=(0,0)$.
    \end{enumerate}
\end{setup}

Condition (\ref{item graphical}) ensures that $\ker(\lambda)$ is a contact structure. Condition (\ref{item start}) defines $R_1$ and $R_2$ from the moment image. Condition (\ref{item singfibers}) allows us to more easily understand the Reeb dynamics in the neighborhood of the singular fibers, since locally these neighborhoods will agree with neighborhoods of the singular fibers in an irrational ellipsoid.

We will first prove case~\ref{case less pi} of Theorem~\ref{theorem at}, as this is the simplest proof and we do not need to perturb the toric contact form.

\begin{proof}[Proof of Theorem~\ref{theorem at}\ref{case less pi}.]
When the angle $t_2-t_1$ traversed by the moment image is strictly less than $\pi$, we can choose the curve defining the toric contact form to be a linear path with irrational slope connecting points in the two rays. We call the induced contact form $\lambda_E$.

\begin{figure}
    \centering
    \includegraphics[scale=.5]{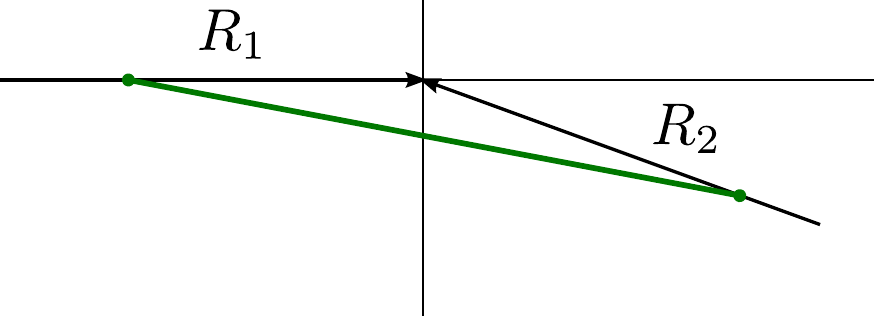}
    \caption{Curve defining the toric contact form $\lambda_E.$}
    \label{fig:ellipsoidal}
\end{figure}

As explained in \textsection\ref{subsec contact toric},  the Reeb direction in the regular torus orbit  that projects to a certain point $P(x)$ is perpendicular to the tangent vector of $P$ at  that point.
Since $P(x)$ is defined by an irrational slope, the only closed Reeb orbits for the Reeb vector field determined by $\lambda_E$ are the two singular toric fibers. Additionally, these Reeb orbits are elliptic, since they are surrounded by invariant tori.

    By the index parity property of ECH index in Remark~\ref{rem:I}, we know that any ECH index $1$ curve must contain an odd number of asymptotic ends at positive hyperbolic orbits. However, by the choice of our contact form, there are no hyperbolic orbits. Therefore, $c(\xi) \neq 0$ and $\at(Y,\lambda_E,J_E) = \infty$.

    For any ECH data $(\lambda, J)$ of $(Y, \xi)$, by Proposition~\ref{prop:cobordism}, there exists $C$ such that
    $$\at_{\simp}^{e^{-C}L}(Y,\lambda, J)\geq \at^L(Y,\lambda_E,J_E),$$
    which implies that $\at_{\simp}(Y,\lambda, J)=\infty$ since $L$ is arbitrary.
\end{proof}

Now we proceed towards cases~\ref{case more pi} and~\ref{case equal pi}. In these cases, we will need actual counts of pseudoholomorphic curves to compute the differential. To provide an explicit construction of a pseudoholomorphic plane asymptotic to a particular Reeb orbit, we will need to impose additional criteria on our toric contact form.

\begin{setup}\label{setup:construction}
    Let $\lambda=p_1(x)dq_1+p_2(x)dq_2$ be an invariant contact form with moment image curve $P(x)=(p_1(x),p_2(x))$ satisfying the conditions of Setup~\ref{setup well defined}. In this set-up, we require the following additional property of $P$:
    \begin{enumerate}
    \setcounter{enumi}{3}
        \item \label{item normalize R-} $R_1 = \{(-c,0)\mid c>0\}$.
        \item \label{item convexity} There exists $x_0\in (0,1)$ such that
\begin{itemize}
    \item {the oriented normal vector to the curve $P$ at $x_0$ is} $\nu_{x_0}=(0,-1)$ (i.e. $p_1'(x_0)>0$, $p_2'(x_0)=0$, and thus by (\ref{eqn:Qsign}) $p_2(x_0)<0$),
    \item $p_2''(x_0)>0$, and
    \item for all $x\in (0,x_0)$, we have $p_2(x)<0$, $p_2'(x)<0$, and $p_1'(x)>0$.
\end{itemize}
    \end{enumerate}
    See Figure~\ref{fig:contact_form_construct_disc}.
\end{setup}

\begin{figure}[h]
     \centering
     \includegraphics[width=0.4\linewidth]{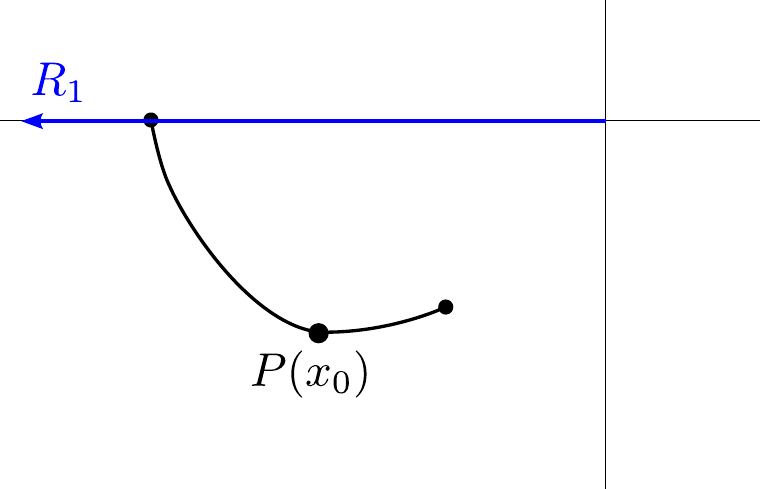}
     \caption{The moment image for a contact form satisfying Setup~\ref{setup:construction}.}
     \label{fig:contact_form_construct_disc}
 \end{figure}

Note that condition~(\ref{item normalize R-}) can always be achieved by an $SL(2,\Z)$ transformation, but it is convenient to perform this normalization for concreteness. In this case, the slope which is collapsed in this toric fiber is $(0,\pm 1)$. Note that this coincides with the slope of the closed Reeb orbits in $T_{x_0}$ which is $(0,-1)$.

Now we can prove part~\ref{case equal pi} of Theorem~\ref{theorem at}.

\begin{proof}[Proof of Theorem~\ref{theorem at}~\ref{case equal pi}]
Let $(Y,\lambda)$ be a contact toric manifold where the angle between the rays is $\pi$. We will require the contact form $\lambda$ to satisfy Setup~\ref{setup:construction}, with 
$$
R_1=\{(-c,0)\mid c>0\},\quad R_2=\{(c,0)\mid c>0\},\quad x_0=\frac{1}{2}
$$ and $P(x)$ symmetric across the $p_2$-axis, namely $(p_1(x),p_2(x)) = (-p_1(1-x),p_2(1-x))$ for $0\leq x\leq \frac{1}{2}$. See Figure~\ref{fig:S1S2graph}.

\begin{figure}
\centering
\includegraphics[scale=1]{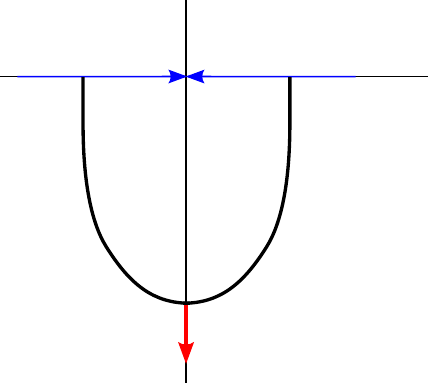}
\caption{Curve $P(x)$ defining the toric contact form $\lambda$ used to compute algebraic torsion in ECH when the angle between the rays is $\pi$.}
\label{fig:S1S2graph}
\end{figure}

{Observe that the Reeb vector field is periodic of slope $(0,-1)$ on the torus fiber $T_{x_0}$ and that these orbits have action equal to 1.} Fix an action bound $L>1$. Let $\lambdapert$ be the perturbed $L$-nondegenerate contact form obtained from the construction of \S\ref{ss:perturb}. Let $\Jpert$ be a compatible almost complex structure on $\R\times Y$. 

Let $\gamma_h$ denote the hyperbolic Reeb orbit in $T_{x_0}$ corresponding to the minimum of the Morse function used in the perturbation. Note that $\gamma_h$ has slope $(0,-1)$. Because of condition (\ref{item normalize R-}) of Setup~\ref{setup:construction}, $\gamma_h$ is contractible. Our strategy is as follows: We first show that $\gamma_h$ is the only generator that could have $\emptyset$ in its $0^{th}$ level differential, then we show that  $ \partial_0\gamma_h\neq \emptyset$. This guarantees that $\emptyset$ is not killed in the first page of the spectral sequence, and hence the algebraic torsion is positive.

\vspace{4pt}
\noindent\underline{1. $\partial_0\alpha = \emptyset\Rightarrow \alpha = \gamma_h$:} Suppose there were an ECH generator $\alpha$ such that $\partial_0\alpha = \emptyset$, we will show that $\alpha=\gamma_h$.  A $\Jpert$-holomorphic curve contributing to $\partial_0$ with $\emptyset$ in the image must have no negative ends. As a consequence, it follows from Proposition~\ref{prop:J0} that such a curve could have $\Jplus=0$ only if it has a single positive end and genus zero.
By the index parity statement of Remark~\ref{rem:I}, the positive end must be a positive hyperbolic orbit. Since the multiplicity of a hyperbolic orbit must be $1$ in an ECH generator, this implies $\alpha = \{(\gamma,1)\}$.

The existence of a $\Jpert$-holomorphic plane positively asymptotic to $\gamma$ with multiplicity $1$ implies that $\gamma$ is null-homotopic (in particular, null-homologous). In $Y$, any Reeb orbit of $\lambdapert$ with action $<L$ is either one of the two singular toric fibers or lies in a regular torus fiber with slope $(p,q)$ with $q<0$. For both of the singular fibers, and any Reeb orbit of slope $(p,q)$ with $p\neq 0$, such a Reeb orbit represents a non-zero class in $H_1(Y)$. Thus, the only possibility for $\gamma$ is $\gamma=\gamma_h$ in $T_{x_0}$.

\vspace{4pt}
\noindent\underline{2. $\partial_0\gamma_h\neq \emptyset$:} 
We start by showing that $\gamma_h$ bounds two pseudoholomorphic planes, and then show that no other curves contribute to the coefficient of $\emptyset$ in $\partial_0 \gamma_h$.

Theorem~\ref{theorem:construction} 
 below states that there exist a $\Jpert$-holomorphic plane $\upert^-$ in $Y_{[0,x_0]}$ positively asymptotic to $\gamma_h$, with ECH and Fredholm indices equal to 1, and $\Jplus$-index equal to zero. Since $Y_{[x_0,1]}$ is strictly contactomorphic to $Y_{[0,x_0]}$ by the symmetry, we conclude that there is another $\Jpert$-holomorphic plane $\upert^+$ in $Y_{[x_0,1]}$, with the same properties. So,
 \[
 \upert^\pm \text{ are positively asymptotic to }\gamma_h \text{ with }\ind(\upert^\pm) = I(\upert^\pm) = 1\text{ and }\Jplus(\upert^\pm)=0.
 \]
To show that $\partial_0\gamma_h\neq \emptyset$ it suffices to show that there are no other index $1$ $\Jpert$-holomorphic curves with no negative ends, whose positive end is exactly $\gamma_h$. Suppose for the sake of contradiction $u'$ is such a curve.
Theorem~\ref{theorem:asymptotics} below states that  any index $1$ $\Jpert$-holomorphic curves whose positive asymptotic end is exactly $\gamma_h$ is sided, that is, it approaches $\gamma_h$ from either $Y_{[0,x_0]}$ or from $Y_{[x_0,1]}$ in its asymptotic end. Moreover, if two curves $u_1$ and $u_2$ approach $\gamma_h$ from the same side, have no other asymptotic ends (in particular no negative ends), and have $Q_{\tau_0}(u_1,u_2) = 0$, then $u_1=u_2$ up to $\R$-translation. Therefore, to know that $u'$ coincides with one of $\upert^\pm$ up to translations, it is enough to prove that 
\begin{equation}\label{eq:Q_tau_u'}
    Q_{\tau_0}(u',\upert^{\pm}) = 0.
\end{equation}
Observe that $[\upert^+]-[\upert^-]$ generates the kernel of $H_2(Y,\gamma_h)\to H_1(\gamma_h)$ in the long exact sequence of the pair. Therefore $[u']=n[\upert^+]+(1-n)[\upert^-]$ so we can calculate 
$$Q_{\tau_0}(u',\upert^{\pm}) = nQ_{\tau_0}(\upert^+,\upert^{\pm})+(1-n)Q_{\tau_0}(\upert^-,\upert^{\pm}).$$
Lemma~\ref{lem:c_tau_shira_plane} below states that $Q_{\tau_0}(\upert^\pm)=0$. Hence, the relative intersection numbers $Q_{\tau_0}([\upert^\pm],[\upert^\pm]) = Q_{\tau_0}([\upert^+],[\upert^-])$ also vanish, and \eqref{eq:Q_tau_u'} holds.

We conclude that $\upert^\pm$ are the only curves counted for $\emptyset$ in $\partial_0\gamma_h$ and thus over $\Z/2\Z$ coefficients, $\partial_0(\gamma_h) \neq \emptyset$ with respect to the ECH data $(\lambdapert,\Jpert)$.

Since $(\gamma_h,1)$ was the only candidate to have image under $\partial_0$ be $\emptyset$, we conclude that $$\at^L(Y,\lambdapert,\Jpert)>0.$$
If $(\lambda',J')$ is any nondegenerate ECH data for the same contact manifold, by Proposition~\ref{prop:cobordism}, there exists $C$ such that
$$\at^{e^{-C}L}_{\simp}(Y,\lambda', J') \geq \at^L(Y,\lambdapert,\Jpert)>0.$$
Since we can choose perturbations for arbitrarily large values of $L$ for which $\at^L(Y,\lambdapert,\Jpert)>0$, (and $C$ does not depend on $L$), we conclude 
$\at^L_{\simp}(Y,\lambda',J')>0$ {for any $L$ and hence $\at_{\simp}(Y,\lambda',J')>0$}. 

\end{proof}

Finally, we consider the case when the angle between the rays is strictly greater than $\pi$.

\begin{proof}[Proof of Theorem~\ref{theorem at}~\ref{case more pi}]
When the angle between the rays is greater than $\pi$, we choose our invariant contact form in a way that heavily utilizes this larger angle.

\begin{setup}\label{setup:morethanpi}
Let $\lambda=p_1(x)dq_1+p_2(x)dq_2$ be an invariant contact form with the moment map image curve $P(x)=(p_1(x),p_2(x))$ satisfying the conditions of Setup~\ref{setup:construction}. In this set-up, we require the following additional property of $P$:
\begin{enumerate} 
\setcounter{enumi}{5}
\item \label{item rotatepi} There exists a point $x_1$ where the oriented vector normal to the curve $P$ at $x_1$ is  $\nu_{x_1} = (0,1)$.
\end{enumerate}
\end{setup}

Because $P(x)$ must be transverse to the radial rays, its tangent vector can get close to the slope of the ray, but not quite realize it. This is enough to ensure condition (\ref{item rotatepi})  holds because we can choose $P(x)$ so that its tangent vector rotates past $(-1,0)$, but does not quite reach the slope of $R_2$. See Figure~\ref{fig:OTgraph}.

\begin{figure}
\centering
\includegraphics[scale=1]{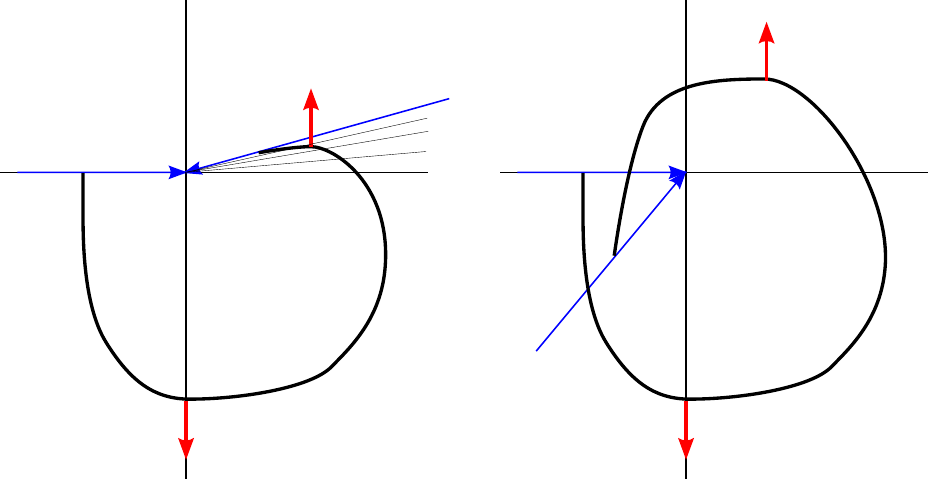}
\caption{Curve $P(x)$ defining the toric contact form $\lambda$ used to compute algebraic torsion in ECH when the angle between the rays is $>\pi$. The curve must have points where the normal vector is $(0,-1)$ and $(0,1)$, and the radial rays from the origin must be everywhere transverse. Observe (left) that when the angle is close to $\pi$, we have just enough space to satisfy both conditions. On the right is a case where the angle is greater than $2\pi$.}
\label{fig:OTgraph}
\end{figure}

Next, we fix $L>1$ and perturb the contact form $\lambda$ as in \S\ref{subsec:nondeg_contact_form} to obtain an $L$-nondegenerate contact form $\lambdapert$. Note that $T_{x_0}$ is a center of the perturbation (since it contains closed Reeb orbits of small action). In the construction of $\lambdapert$, we chose a parameter $\delta$ which determines the support of the perturbation. Note that we can always choose a smaller $\delta$ while maintaining the desired properties of the perturbation, stated in  Lemma~\ref{lem:morsebott_perturbation} below. 
\begin{setup}\label{setup:deltaperturbation}
We will shrink $\delta$ such that 
\begin{enumerate}
    \item for $x_-:=x_0+\delta$, $P'(x_-)$ is a rational slope, and 
    \item there exists an $x_+\in (x_-,1)$ such that $P'(x_+)$ is negatively proportionate to $P'(x_-)$.
    \item $T_{x_-}$ and $T_{x_+}$ are unperturbed tori foliated by closed Reeb orbits which have action greater than $L$.
\end{enumerate}
\end{setup}
\noindent Note that the Reeb vector field rotates by $\pi$ from $T_{x_-}$ to $T_{x_+}$.
Let $\gamma_h$ be the positive hyperbolic Reeb orbit in $T_{x_0}$. Let $\Jpert$ be an generic almost complex structure on $\R\times Y$ compatible with $\lambdapert$. 

Theorem~\ref{theorem:construction} states that there exists a $\Jpert$-holomorphic plane $\upert$ in $Y_{[0,x_0+\delta]}$ positively asymptotic to $\gamma_h$ with multiplicity $1$ such that $\ind(\upert)=I(\upert)=1$ and $\Jplus(\upert)=0$. Moreover Lemma~\ref{lem:c_tau_shira_plane} states that its self intersection number vanishes: $Q_{\tau_0}(\upert)=0$.

We would like to show that there is no other $\Jpert$ holomorphic curve $u'$ positively asymptotic to $\gamma_h$.

Theorem~\ref{thm:no curve from right} shows that any $\Jpert$-holomorphic curve $u'$ whose positive asymptotics are exactly one end at $\gamma_h$ has index $1$, has no negative ends, lies in the relative homology class $[u']=[\upert]\in H_2(Y,\gamma_h)$, and does not approach $\gamma_h$ from the right. 

Theorem~\ref{theorem:asymptotics} states that all Fredholm index one $\Jpert$-holomorphic curves with a positive asymptotic end on $\gamma_h$ approaches either from the left or from the right. 
In particular, $\upert$ and $u'$ both approach from the left. 

Next, we would like to show that $u'=\upert$ up to $\R$-translation. Theorem~\ref{theorem:asymptotics} states that if $u_1$ and $u_2$ are two Fredholm index one $\Jpert$-holomorphic curves approaching $\gamma_h$ from the same side and have $Q_\tau(u_1,u_2)=0$, where $u_1$ has no negative ends, then $u_1=u_2$ up to $\R$-translation.  

To conclude that $u'$ coincides with $\upert$ up to translation, it remains to notice that 
$$
Q_{\tau_0}(\upert,u') = Q_{\tau_0}(\upert,\upert) = Q_{\tau_0}(\upert)=0.
$$
Thus, up to $\R$-translation, $\upert$ is the unique $\Jpert$-holomorphic curve (of any index) whose positive asymptotic ends are precisely $\gamma_h$. Since $\upert$ has ECH index $1$ and $\Jplus(\upert)=0$, it contributes to $\partial_0$. From this we can conclude the following things:
\begin{enumerate}
\item The Reeb current $(\gamma_h,1)$ is a simple generator (Definition~\ref{def:simple_at}). 
\item In the ECH complex, $\partial \gamma_h = \emptyset$.
\item Using the $\Jplus$ splitting of the differential, $\partial_0(\gamma_h) = \emptyset$.
\end{enumerate}
It follows that $c(\xi) = 0$, $\at^L(Y,\lambdapert, \Jpert)=0$, and $\at^L_{\simp}(Y,\lambdapert,\Jpert)=0$.

If $(\lambda',J')$ is any nondegenerate ECH data for the same contact manifold, by Proposition~\ref{prop:cobordism}, there exists $C\geq0$ such that
$$0=\at_{\simp}^L(Y,\lambdapert, \Jpert) \geq \at^{e^{C}L}(Y,\lambda', J').$$
Since we can choose perturbations for which these calculations hold for arbitrarily large values of $L$, we conclude that $\at(Y,\lambda',J')=0$.
\end{proof}

\section{Perturbations of toric contact forms and properties of their Reeb orbits}\label{subsec:nondeg_contact_form}

In this section we will show how to perturb a toric invariant contact form $\lambda$ to make its Reeb orbits, up to a particular action bound, nondegenerate. We will give an explicit formula for the perturbed contact forms, describe the Reeb orbits of low action, and compute their Conley-Zehnder indices with respect to a natural trivialization coming from the contact toric structure.

\subsection{Morse-Bott toric contact form} \label{subsec contact toric}

We can calculate the Reeb vector field for $\lambda$ explicitly in toric coordinates.
As the toric action on the set of regular torus fibers is given by the rotation of every fiber, i.e. 
$X_k=\partial_{q_k}$, $k=1,2$, it follows that the corresponding invariant contact form
 $ \lambda$ is given by
\begin{equation}\label{eq:lambda}
\lambda = p_1(x) dq_1 + p_2(x) dq_2, \quad \mbox{with} \quad (q_1, q_2) \in T^2, 
\end{equation}
where $P(x)=(p_1(x),p_2(x)):[0,1]\to \R^2\setminus \{(0,0)\}$ is a curve winding counterclockwise about the origin. We assume the properties of Setup~\ref{setup well defined}.
Since
$$ d\lambda = p_1'(x) dx\wedge dq_1 +p_2'(x) dx\wedge dq_2,
$$
 the corresponding Reeb vector field is given by
\begin{equation}\label{eq:reeb}
R(x,q_1,q_2)=\frac{1}{Q(x)}\big( p_2'(x) \partial_{q_1} - p_1'(x) \partial_{q_2}\big)
,\quad \text{where}\quad Q(x):=p_1(x)p_2'(x) - p_2(x) p_1'(x).
\end{equation}
The transversality of the curve $P$ to radial rays ensures that $R$ is  well defined ($Q(x)\neq 0$). In fact, because $P$ is oriented counterclockwise,
\begin{equation}\label{eqn:Qsign}
    Q(x)=p_1(x) p_2'(x) - p_2(x) p_1'(x)>0
\end{equation}
Note that the standard inner product $\langle R,\dot{P} \rangle$ vanishes, where $\dot{P}(x) =p_1'(x)\partial_{q_1}
+p_2'(x)\partial_{q_2}$ is the tangent vector field of $P$.  Therefore,
the Reeb direction in the regular torus fiber over a certain point $P(x)$ is perpendicular to the tangent vector of $P$ at  that point. If the slope of $R$ is rational at a point $x \in (0,1)$, then the corresponding regular torus fiber is foliated by closed Reeb orbits of $\lambda$. In particular, these Reeb orbits are degenerate. 

In the ECH set-up, we need to work with an $L$-nondegenerate contact form, meaning that all Reeb orbits whose action is below a chosen bound $L$ are nondegenerate. In \S\ref{ss:perturb} we show how to perturb the toric contact form $\lambda$ to make it $L$-nondegenerate.

\subsection{Nondegenerate perturbation up to large action}\label{ss:perturb}
Fix an action bound $L>0$. Our goal is to modify $\lambda$ to an $L$-nondegenerate contact form, which will be denoted by $\lambdapert$.  There are finitely many torus fibers $\{T_{x_1},\dots, T_{x_N}\}$ with $x_1,\dots, x_N\in (0,1)$ containing all Reeb orbits of action less than $L$.
Note that the Reeb orbits in the singular fibers are already nondegenerate since we assume the slope of $P$ to be irrational there.

Let $(a_i,b_i)\in \Z^2$ be the primitive integral vector positively proportional to the Reeb vector field in $T_{x_i}$.  To construct the perturbation, we first need functions $f_i: T_{x_i} \to \R$ whose level sets are the curves of slope $(a_i,b_i)$ and have Morse-Bott families of critical points along exactly two such curves. Such a function is given by
\begin{equation}\label{cos}
    f_i(q_1,q_2):= \cos(2\pi(-b_iq_1+a_iq_2)).
\end{equation}

Next we choose a small $\delta>0$ such that $\delta$ neighborhoods of distinct $x_i$ are disjoint, and bump functions $\beta_i^\delta:[0,1]\to [0,1]$ such that
\[
\operatorname{supp}(\beta_i^\delta)\subset (x_i-\delta, x_i+\delta),\quad\text{ and }\quad \beta_i^\delta(x)=1-(x_i-x)^2 \text{ for $x$ near $x_i$.}
\]
Note that in this case  $\beta_i^\delta(x_i)=1$ and $(\beta_i^\delta)'(x_i)=0$. Let
$$F_L(x,q_1,q_2):= \sum_{i=1}^N \beta_i^{\delta}(x)f_i(q_1,q_2).$$
{Observe that the circle in $T_{x_i}$ corresponding to the  minimum of $f_i$ is a circle of saddle points of $F_L$ (here we refer to a saddle point of $F_L$ restricted to a transversal to $R$). The circle corresponding to the  maximum of $f_i$ is a circle of maxima of $F_L$.}
Using $F_L$ we define a perturbed contact form
$$\lambda_{\varepsilon,L} := e^{\varepsilon F_L}\lambda.$$

We will call $T_{x_1},\dots, T_{x_N}$ the \emph{centers of the perturbation}, and tori $T_x$ for $x$ outside the support of all $\beta_i$ \emph{unperturbed tori}.

\begin{remark}
    The dependence of $\lambdapert$ on $L$ comes from the set of points $\{x_1,\dots, x_N\}$. As $L$ becomes larger, this set gets larger. To ensure that $\delta$-neighborhoods of different points $x_i$ from this set are disjoint, $\delta$ will need to get smaller as $L$ becomes  large. However any $\delta$ that is sufficiently small for a large $L$ will also work perfectly well for a smaller $L$. For any fixed $L$, we can easily shrink $\delta$ to be smaller, while maintaining all the required properties.
    Observe that as $\varepsilon\to 0$, $\lambdapert\to \lambda$. In the definition of $\lambdapert$, the choice of $\varepsilon>0$ is independent from $L$ and $\delta$. However, to ensure that we do not introduce new closed Reeb orbits of action less than $L$, we must choose $\varepsilon$ very small. Thus, to achieve the properties of $\lambdapert$ that we want, the choice of sufficiently small $\varepsilon$ \emph{will} depend on $L$.
\end{remark}

In order to understand the dynamics of the Reeb vector field for the perturbed contact form, we calculate it, using the explicit formula for the perturbation.

\begin{lemma}[c.f. Lemma 2.3 of~\cite{Bourgeois}] \label{lemma:perturbedReeb}
The Reeb vector field $R_\varepsilon$ for $\lambdapert$ is given by
$$R_\varepsilon = e^{-\varepsilon F_L}(R-\varepsilon X_{F_L})$$
where $R$ is the Reeb vector field for $\lambda$ and $X_{F_L}\in \xi$ is the unique vector field determined by the condition
$$d\lambda(Z,X_{F_L}) = dF_L(Z) \qquad \textrm{for all }Z\in \xi. $$
\end{lemma}

\begin{remark} \label{remark perturbed reeb}
The equality $d\lambda(-,X_{F_L})=dF_L(-)$ holds on $\xi$. It extends to $T_pY$ only for $p\in T_{x_i}$, since there $dF_L(R)=0$, while $d\lambda (R,X_{F_L})=0$ holds everywhere.
\end{remark}

\begin{proof}[Proof of Lemma~\ref{lemma:perturbedReeb}]
We need to check that $d\lambdapert(R_\varepsilon,\cdot) =0$ and $\lambdapert(R_\varepsilon) =1$.
Indeed,
$$\lambdapert(R_\varepsilon) = \lambda(R) -\varepsilon \lambda(X_{F_L}) = \lambda(R) = 1\quad\text{since}\quad X_{F_L}\in \ker(\lambda).
$$
To check $R_\varepsilon$ is in the kernel of $d\lambdapert$, we compute:
$$d\lambdapert = e^{\varepsilon F_L}(d\lambda+\varepsilon dF_L\wedge \lambda) $$
so
\begin{align*}
    d\lambdapert (R_\varepsilon,\cdot):=&\ 
    d\lambdapert\big(e^{-\varepsilon F_L}(R-\varepsilon X_{F_L}), \cdot\big) \\
    =& -\varepsilon d\lambda(X_{F_L},\cdot) +\varepsilon dF_L(R)\lambda(\cdot)-\varepsilon^2 dF_L(X_{F_L})\lambda(\cdot)-\varepsilon dF_L(\cdot)+\varepsilon^2 \lambda(X_{F_L})dF_L(\cdot)\\
    =& -\varepsilon d\lambda(X_{F_L},\cdot) +\varepsilon dF_L(R)\lambda(\cdot)-\varepsilon^2 dF_L(X_{F_L})\lambda(\cdot)-\varepsilon dF_L(\cdot) 
\end{align*}
since $X_{F_L}\in \xi=\ker(\lambda)$. To check this is identically zero, we will show that it vanishes on $\xi$ and it vanishes on $R$. Since $TY = \xi\oplus \langle R \rangle$, this will prove the claim.

For vectors $Z\in \xi=\ker(\lambda)$, we have
\begin{align*}
   d\lambdapert (R_\varepsilon,Z):=&\ 
    -\varepsilon d\lambda(X_{F_L},Z) +\varepsilon dF_L(R)\lambda(Z)-\varepsilon^2 dF_L(X_{F_L})\lambda(Z)-\varepsilon dF_L(Z) \\
 =&\ \varepsilon dF_L(Z)+ 0-0-\varepsilon dF_L(Z) = 0. 
\end{align*}
On $R$, the Reeb vector field for $\lambda$, we have
\begin{align*}
    d\lambdapert (R_\varepsilon,R):=&\ 
    -\varepsilon d\lambda(X_{F_L},R) +\varepsilon dF_L(R)\lambda(R)-\varepsilon^2 dF_L(X_{F_L})\lambda(R)-\varepsilon dF_L(R)\\
    =&\ \varepsilon( 0 +dF_L(R)-\varepsilon dF_L(X_{F_L})-dF_L(R)) = -\varepsilon^2 dF_L(X_{F_L}) = \varepsilon^2 d\lambda(X_{F_L},X_{F_L}) = 0. \qedhere
\end{align*}
\end{proof}

Next we show that the Reeb vector field $R_\varepsilon$ of the perturbed contact form $\lambdapert$ satisfies the non-degeneracy properties we were aiming for and discuss its closed Reeb orbits below the action bound.

\begin{lemma}
\label{lem:morsebott_perturbation}
    For sufficiently small $0<\varepsilon<1$, the Reeb vector field of the contact form
    $$\lambdapert := e^{\varepsilon F_L}\lambda$$ satisfies the following properties:
    \begin{enumerate}
        \item It has exactly two closed orbits  of slope $(a_i,b_i)$ in $T_{x_i}$ for each $i=1,\dots, N$: $\gamma_e^i$ (corresponding to the maximum of $\eqref{cos}$) and $\gamma_h^i$ (corresponding to the minimum of $\eqref{cos}$). Moreover,  $\mathcal{A}(\gamma_e^i)>\mathcal{A}(\gamma_h^i)$.
        \item For each $i=1,\dots, N$, $\gamma_e^i$ is elliptic and $\gamma_h^i$ is positive hyperbolic. 
        \item These $2N$ orbits and the two singular fibers are the only closed Reeb orbits of action $<L$.
    \end{enumerate}
\end{lemma}
\begin{proof}

    \begin{enumerate}
        \item Using the description of $R_{\varepsilon}$ from Lemma~\ref{lemma:perturbedReeb}, we see that along the critical submanifolds of $F_L$, $X_{F_L}$ vanishes so $R_\varepsilon$ is proportional to $R$. Moreover, $dF_L=0$ along exactly two circles in $T_{x_i}$ (at $-b_iq_1+a_iq_2 = 0$ and $-b_iq_1+a_iq_2 = \frac{1}{2}$). Since $R$ is tangent to these two circles, they are both closed orbits of $R_\varepsilon$.

        Furthermore, in $T_{x_i}$, we have that $d\lambda(\cdot, X_{F_L}) = dF_L(\cdot)$ (see Remark~\ref{remark perturbed reeb}), $\beta_i'(x_i) = 0$, and $(p_1'(x),p_2'(x))$ is proportionate to $(-b,a)$. Therefore we have
        $$d\lambda(\cdot, X_{F_L}) = dF_L(\cdot) = df_i(\cdot) = -\sin(2\pi(-b_iq_1+a_iq_2))(-b_idq_1+a_idq_2).$$
        Writing out $d\lambda$ explicitly, we have
        $$dx\wedge k(-b_idq_1+a_idq_2)(\cdot,X_{F_L}) =  -\sin(2\pi(-b_iq_1+a_iq_2))(-b_idq_1+a_idq_2).$$
        Therefore we compute that in $T_{x_i}$ $X_{F_L}$ is positively proportionate to $-\sin(2\pi(-b_iq_1+a_iq_2))\partial_x$. Since this is non-vanishing and transverse to $T_{x_i}$ away from the two closed Reeb orbits mentioned above, there are no other closed Reeb orbits of $R_\varepsilon$ which are contained in $T_{x_i}$.

        \item (c.f. Lemma 10.3 in \cite{Bourgeois}.) The periodic orbits in $T_{x_i}$ are non-degenerate since $F_L$ is a Morse function.  As the orbits $\gamma_h^i$ and $\gamma_e^i$ are degenerate as orbits of $R$, their transverse behavior as orbits of $R_\varepsilon$ is determined by $X_{F_L}$. 
        
        {Consider an almost complex structure $J$ on $\xi$ compatible with $d\lambda$. Then on the manifold $Y$ we can define a Riemannian metric whose restriction to $\xi$  is $d\lambda(\cdot, J\cdot)$ and since the levels of $F_L$ are not tangent to $\xi$ we can also assume that for this metric $\nabla {F_L}\in \xi$.} Then, for $y\in Y$ and $u,v$ two vectors in $\xi_y$ we have that $d\lambda(u,Jv)=g(u,v)$. The metric allows us to write the equation $$d\lambda(\nabla F_L, X_{F_L})=dF_L(\nabla F_L)=g(\nabla F_L,\nabla F_L)=d\lambda(\nabla F_L, J\nabla F_L),$$
        and thus $X_{F_L}=J\nabla F_L$ and the two vector fields $X_{F_L}$ and $\nabla F_L$ are transverse except where they both vanish.
        
        Near $\gamma_h^i$, $F_L$ is a product of $\beta_i(x)$ and $f_i(q_1,q_2)$. Since $x_i$ is a local maximum of $\beta_i$, and $\gamma_h^i$ is a Morse-Bott minimum of $f_i$, the index of the $\gamma_h^i$ Morse-Bott family of zeros of $\nabla F_L$, and thus of $X_{F_L}$ is $1$; namely, if we project to a slice transverse to $\gamma_h^i$, the stable and unstable manifolds are each $1$-dimensional. Since $R_\varepsilon$ is a rescaling of $R-\varepsilon X_{F_L}$, if we project $R_\varepsilon$ to a slice transverse to $\gamma_h^i$, we will see a hyperbolic singularity. Thus $\gamma_h^i$ is a hyperbolic periodic orbit of $R_\varepsilon$. {To see that $\gamma_h^i$ is a positive hyperbolic orbit, observe first that $\nabla F_L$ is tangent to the torus $T_{x_i}$ and then $X_{F_L}$ is transverse to $T_{x_i}$ whenever it is non-zero. This implies that the stable and unstable manifolds of $R_\varepsilon$ can never be tangent to $T_{x_i}$ and that they do not intersect $T_{x_i}$. If $\gamma_h^i$ was negative hyperbolic, then the stable/unstable manifolds turn (with respect to the $T_{x_i}$) when one goes once around $\gamma_h^i$, which we just showed does not happen. Thus $\gamma_h^i$ is positive hyperbolic.}

        Likewise, since $\gamma_e^i$ is a maximum of $f_i(q_1,q_2)$, the Morse-Bott index of $\nabla F_L$ is 2, so in a slice transverse to $\gamma_e^i$, $\nabla F_L$ has a sink. Applying $-J$, we see that the vector field $-X_{F_L}$ (and thus the projection of $R_\varepsilon$ to the transverse slice) behaves as a rotation around $\gamma_e^i$. Thus $\gamma_e^i$ is an elliptic periodic orbit of $R_\varepsilon$.

         \item This follows from Lemma 2.3 of~\cite{Bourgeois}. The idea is to use Lemma~\ref{lemma:perturbedReeb}, $R_\varepsilon = e^{-\varepsilon F_L}(R-\varepsilon X_{F_L})$ and consider the case where $\varepsilon$ is very small. Note that $X_{F_L}$ is non-vanishing whenever $dF_L$ is non-zero.
         
         Observe that $X_{F_L}\in \xi$ and $R$ are linearly independent. In order for a Reeb orbit to close up, it must return to its original position in all three coordinate dimensions. For small $\varepsilon$, the flow of $R_\varepsilon$ moves very slowly in the $X_{F_L}$ direction because of the multiplicative term by $\varepsilon$.
        Thus, by choosing $\varepsilon$ very small, the closed orbits of $X_{F_L}$ will have very long period, (except for the zeros of $X_{F_L}$), so the closed orbits of $R_\varepsilon$ will also have very long period.
    \end{enumerate}
\end{proof}

It follows from Lemma~\ref{lem:morsebott_perturbation} that $\lambdapert$ is $L$-nondegenerate. If we look at the symplectization of $(Y,\lambda)$, given by $(\R\times Y,d(e^t\lambda))$, the pair  $(Y,\lambdapert)$ appears as the graph of $\varepsilon F_L$:
$$G_{\varepsilon}:=\{(\varepsilon F_L(y),y)\in \R\times Y\}  $$
where the contact form is obtained by restricting the primitive $e^t\lambda$ to the graph. 
Given any other contact form $\lambda'$  for the same contact structure ($\ker(\lambda') = \ker(\lambda)$), there exists some function $f:Y\to \R$ such that $\lambda'= e^f\lambda$. This allows us to build a symplectic cobordism as a subset of the symplectization, which will allow us to relate the simple algebraic torsion computed using $\lambdapert$ to the algebraic torsion computed using $\lambda'$ as follows.

\begin{proposition}\label{prop:cobordism}
    Let $(Y,\lambda)$ be a contact toric manifold with $\xi=\ker(\lambda)$. Let $\lambda'$ be any nondegenerate contact form for $\xi$. Then there exists a constant {$C > 1$} such that for any $L>0$ and $L$-nondegenerate perturbation, $\lambdapert$ as above, and compatible almost complex structure $\Jpert$, we have
    $$\at^{e^{-C}L}_{\simp}(Y,\lambda', J') \geq \at^L(Y,\lambdapert,\Jpert),$$
    and 
    $$\at_{\simp}^L(Y,\lambdapert, \Jpert) \geq \at^{e^{C}L}(Y,\lambda', J').$$
\end{proposition}

\begin{proof}
Since $\ker(\lambda')=\ker(\lambda)$, $\lambda'=e^f\lambda$ for some $f:Y\to \R$. Since $Y$ is compact, there exists  {$C > 1$} such that $-C+1<f<C-1$. Then there is a symplectic cobordism from $(Y,\lambdapert)$ to $(Y,e^{-C}\lambda')$ by taking the subset of the symplectization given by
$$W^-:= \{ (t,y) \in \R\times Y \mid  f(y)-C \leq t\leq \varepsilon F_L(y) \} $$
Note that $-1<-\varepsilon \leq \varepsilon F_L$ so the upper bound is strictly higher than the lower bound.
Since the primitive for the symplectic form on $\R\times Y$ is $e^t\lambda$, the restriction of this form to the upper boundary is $e^{\varepsilon F_L}\lambda=\lambdapert$ and the restriction to the lower boundary is $e^{f-C}\lambda = e^{-C}\lambda'$. 

Similarly, there is a  cobordism $W^+$ from $(Y,e^C\lambda')$ to $(Y,\lambdapert)$ 
$$W^+:= \{ (t,y) \in \R\times Y \mid  \varepsilon F_L(y) \leq t\leq f(y)+C \}. $$
Now, we apply Theorem~\ref{thm:hutchings_A9} to the cobordism $W^-$ and obtain that 
$$\at_{\simp}^L(Y,\lambdapert, \Jpert) \geq \at^L(Y,e^{-C}\lambda',J') = \at^{e^{C}L}(Y,\lambda', J'),$$
by the ``scaling'' isomorphism, equation (\ref{e:scale}).
Similarly, using the cobordism $W^+$ we get
$$\at^{e^{-C}L}_{\simp}(Y,\lambda', J')=\at_{\simp}^L(Y,e^{C}\lambda', J') \geq \at^L(Y,\lambdapert,\Jpert).$$
\end{proof}

\subsection{Positivity of the Morse-Bott torus and Robbin-Salamon index computation} \label{ss:MBindex}
Here we compute indices associated to Reeb orbits for both the toric contact form $\lambda$ and its perturbation $\lambdapert$. These indices will be measured relative to a trivialization of the contact structure near the Reeb orbits.

{We first recall the definition of a positive Morse-Bott torus, see also \cite[Prop.~2.3, 2.4]{Bourgeois}, \cite[\S 4]{CGH_obd}, \cite[\S 4]{Gutt_CZ}, \cite[\S 2]{yao2022cascades}.  Let $T_x$ be a Morse-Bott torus corresponding to a family of embedded Reeb orbits.  With respect to a natural trivialization given by an oriented basis $\{V_1, V_2\}$ for $\xi$ at any point $p\in T_x$, such that $V_1$ is transverse to $T_x$ and $V_2$ is tangent to $T_x$, the derivative of the first return map $\xi_p \to \xi_p$ of the Reeb flow is given by the matrix $\begin{pmatrix}
    1 & 0 \\a & 1
\end{pmatrix}$ (here a vector $V= a_1V_1 + a_2V_2$ is written as a column vector).  The Morse-Bott condition implies that $a\neq 0$.  We say that the Morse-Bott torus of Reeb orbits $T_x$ is \emph{positive} whenever $a>0$ and \emph{negative} whenever $a<0$.  }

{All the embedded orbits comprising a Morse-Bott torus $T_x$ are degenerate and admit a Robbin-Salamon index, which is half integer valued.  With respect to the above trivialization, any embedded orbit in a positive Morse-Bott torus has Robbin-Salamon index $1/2$ while any embedded orbit in a negative Morse-Bott torus has Robbin-Salamon index $-1/2$.  This means that after perturbation with respect to this trivialization, the resulting embedded elliptic orbit of a positive Morse-Bott torus has Conley-Zehnder index 1, while the resulting embedded elliptic orbit of a positive Morse-Bott torus has Conley-Zehnder index $-1$.  The resulting embedded hyperbolic orbit has Conley-Zehnder index 0 in either setting.}

We define a trivialization $\tau_0$ of the contact distribution $\xi=\ker\lambda=\ker(p_1(x)dq_1+p_2(x)dq_2)$ in $Y_{(0,1)}$ by
\begin{equation}\label{eqn:tau0}
    \left(V_1:=\partial_x, \quad V_2:=\frac{1}{Q(x)}\big(p_1(x)\partial_{q_2} - p_2(x) \partial_{q_1}\big)\right).
\end{equation}
Note that $d\lambda(V_1,V_2)>0$, so $(R,V_1,V_2)$ is a positively oriented basis for $Y_{(0,1)}$.
First, we compute the relevant index, the Robbin-Salamon index, in the degenerate toric setting.

\begin{lemma}\label{lem:RS_of_gammah}
    Suppose $x_0$ is a point on the curve $P$ such that
    $$W:=p_1'(x_0)p_2''(x_0)-p_1''(x_0)p_2'(x_0)>0.$$
    Then the Robbin-Salamon index of the Morse--Bott torus fiber $T_{x_0}$ over $P(x_0)$ with respect to $\tau_0$ is
    \[
    RS_{\tau_0}(T_{x_0}) = \frac{1}{2}.
    \]
\end{lemma}

\begin{remark}
We can think of the condition $W>0$ as a kind of convexity/positive curvature condition since
$$W = \langle -p_2'(x_0), p_1'(x_0)\rangle \cdot \langle p_1''(x_0),p_2''(x_0)\rangle = \nu_{x_0} \cdot P''(x_0). $$
\end{remark}

\begin{proof}[Proof of Lemma~\ref{lem:RS_of_gammah}]
Similar computations have been done in \cite{Bourgeois}. We include explicit computations here for completeness in our setting and to show that our Morse-Bott torus is positive.

Recall from (\ref{eq:reeb}) that on $TY$ we have
$$   R=\frac{1}{Q(x)}\left(p_2'(x)\partial_{q_1}-p_1'(x)\partial_{q_2}\right)$$
We now express the linearized flow at the fiber of $x_0$ in the coordinates $\langle R, V_1, V_2\rangle$ and show that it is a $3\times 3$-matrix preserving the first coordinate. 
Recall that $Q(x):=p_1(x) p_2'(x) - p_1'(x) p_2(x)$. Let $\phi^t$ denote the Reeb flow, then taking the spatial derivative of the differential equation
$\dot{\phi^t}=R\cdot \phi^t$ we have that $d\dot{\phi^t}=dR\cdot d\phi^t$. Calculating the partial derivatives, we see the matrix $dR|_{x_0}$ in coordinates $\langle \partial_{x}, \partial_{q_1}, \partial_{q_2}\rangle$ is given by
$$dR_{x_0}=\begin{pmatrix} 0 & 0 & 0\\ A& 0 & 0\\ B& 0 & 0\end{pmatrix},$$
where $$A = \frac{p_2(x_0)\left[-p_2''(x_0) p_1'(x_0)  + p_2'(x_0) p_1''(x_0)\right]}{(p_1(x_0)p_2'(x_0)-p_1'(x_0)p_2(x_0))^2} = -\frac{p_2(x_0)W}{(Q(x_0))^2}$$
and
$$B = \frac{-p_1(x_0)\left[-p_2''(x_0) p_1'(x_0)  + p_2'(x_0) p_1''(x_0)\right]}{(p_1(x_0)p_2'(x_0)-p_1'(x_0)p_2(x_0))^2} = \frac{p_1(x_0)W}{(Q(x_0))^2}.$$
From this we see that, 
$$dR_{x_0}(V_1)=A\partial_{q_1}+B\partial_{q_2}= \frac{W}{Q(x_0)} V_2 =: cV_2,$$
where $c>0$ because $W>0$ by assumption and $Q(x_0)>0$ by the contact condition.
Also, $dR_{x_0}(V_2)=0$. 
Thus the restriction to $\xi$ in coordinates $\langle V_1,V_2\rangle$ is the matrix 
$$\begin{pmatrix}
    0&0\\c&0
\end{pmatrix}$$
with $c>0$, which means that the Morse-Bott tori are positive. The action of the linearized flow at time $t$ on $\xi$ is then given by the matrix 
$$\begin{pmatrix}
    1 & 0 \\ct & 1
\end{pmatrix}.$$
Together with Proposition 4.9 in \cite{Gutt_CZ}, we have that 
\begin{equation}
\label{eq:RS}
    RS_{\tau_0}(T_0) = -\frac{1}{2}(0-1) = \frac{1}{2}.
\end{equation}
Note that our flow is lower triangular, whereas that in \cite{Gutt_CZ} is upper triangular. This means that our basis for the trivialization is of opposite orientation to that in \cite{Gutt_CZ} and therefore we have an extra negative sign in (\ref{eq:RS}).
\end{proof}

Now, we look at indices of the Reeb orbits with the perturbed contact form $\lambdapert$. For each $x_i$ which is a center of the perturbation, the torus $T_{x_i}$ contains two closed Reeb orbits, $\gamma_h^i$ and $\gamma_e^i$ as in Lemma~\ref{lem:morsebott_perturbation}. We can deduce the Conley-Zehnder indices of these Reeb orbits from the Robin-Salamon index computed with the unperturbed contact form.

\begin{cor} \label{cor:CZ}
    Let $x_0$ be a convex point on the moment curve $P$, which is a center of perturbation for the perturbed contact form $\lambdapert$. Let $\gamma_h:=\gamma_h^0$ and $\gamma_e:=\gamma_e^0$ be the resulting closed Reeb orbits in $T_{x_0}$ as in Lemma~\ref{lem:morsebott_perturbation}. Then their Conley-Zehnder indices are 
    \[
    CZ_{\tau_0}(\gamma_h) = 0 \quad \text{and}\quad CZ_{\tau_0}(\gamma_e) = 1.
    \]
\end{cor}
\begin{proof}
This follows from the Conley-Zehnder index formula in \cite[Lemma 2.4]{Bourgeois}:
\begin{align*}
    CZ_{\tau_0}(\gamma_h) &= RS_{\tau_0}(T) - \frac{1}{2} \dim(S^1) + \ind_{\text{Morse}}^{S^1}(\gamma_h) = 0,\\
    CZ_{\tau_0}(\gamma_e) &= RS_{\tau_0}(T) - \frac{1}{2} \dim(S^1) + \ind_{\text{Morse}}^{S^1}(\gamma_e) = 1.
\end{align*}

Note that $\ind_{\text{Morse}}^{S^1}$ denotes the index of the critical point of the Morse function $f_i$ on the quotient of $T_{x_i}$ by the Reeb orbits. This quotient is the $S^1$ in the notation. Thus $\ind_{\text{Morse}}^{S^1}(\gamma_h)=0$ and $\ind_{\text{Morse}}^{S^1}(\gamma_e)=1$.
\end{proof}

\section{Constructing pseudoholomorphic planes}\label{section plane exists}

In this section we prove existence of an embedded pseudoholomorphic plane positively asymptotic to a positive hyperbolic Reeb orbit. This will apply to any contact manifold $(Y,\xi)$ with a co-dimension 0 submanifold $N\subset Y$ which is a contact toric manifold with one boundary component and having one singular fiber, namely 
  $N=Y_{[0,k]}$, $0<k<1$,
 and satisfying the conditions of Setup~\ref{setup:construction}.

\begin{theorem} \label{theorem:construction}
    Let $(N,\lambda)$ be a contact toric manifold  with exactly one singular fiber and boundary a torus, satisfying the assumptions of Setup~\ref{setup:construction}. {There exists  a $\lambda$-compatible almost complex structure $J_0$ on $\R\times N$, such that for each Reeb orbit $\gamma_c$ in the torus fiber $T_{x_0}$ there is {an embedded} $J_0$-holomorphic plane $u_0:\C\to \R\times N$ which is positively asymptotic  to $\gamma_c$}. Additionally, there exist for any $L>1$:
    \begin{itemize}
        \item an $L$-nondegenerate perturbation $\lambdapert$ of $\lambda$,
        \item a closed positive hyperbolic Reeb orbit $\gamma_h$ of $\lambdapert$ in $T_{x_0}$,
        \item and a regular $\Jpert$-holomorphic plane $\upert$ positively asymptotic to $\gamma_h$ with $\ind(\upert)=I(\upert) = 1$ and $\Jplus(\upert) = 0$ for any $\lambdapert$-compatible almost complex structure $\Jpert$ on $\R\times N$ sufficiently close to $J_0$.
    \end{itemize}
\end{theorem}

The perturbation of the contact form $\lambdapert$ is the one constructed in \textsection\ref{ss:perturb}, and the Reeb orbit $\gamma_h$ is the hyperbolic Reeb orbit in $T_{x_0}$. The goal of this section is to specify an almost complex structure $J_0$ and construct the {$J_0$} and $\Jpert$-holomorphic planes for {any} $\lambdapert$-compatible almost complex structure $\Jpert$ sufficiently close to $J_0$.

We start with a construction of a $J_0$-holomorphic plane for an almost complex structure $J_0$ compatible with the degenerate Morse--Bott toric contact form $\lambda$, and then give a moduli space argument to prove existence for the perturbed data. 

\subsection{Explicit construction in a Morse--Bott setting} 
\label{subsection MorseBott plane}

Observe that $T_{x_0}$ is foliated by closed Reeb orbits for $\lambda$ of slope $(0,-1)=-\partial_{q_2}$ because $(p_1'(x_0),p_2'(x_0))$ is positively proportionate to $(1,0)$. In $(q_1,q_2,x)$ coordinates, these Reeb orbits are given as:
\[
\gamma_c(t) = (c,-t,x_0)\in Y\quad \text{for all }c\in S^1.
\]

\begin{proposition}\label{prop:explicit_construction}
Let $(N,\lambda)$ be a contact toric 3-manifold with boundary satisfying the conditions of Setup~\ref{setup:construction}.  There exists 
 a $\lambda$-compatible almost complex structure $J_0$ on $\R\times N$, such that for each Reeb orbit $\gamma_c$ in $T_{x_0}$ there is {an embedded} $J_0$-holomorphic plane $u_0:\C\to \R\times N$ which is positively asymptotic  to $\gamma_c$.
\end{proposition}

The outline of the proof is as follows: 
\begin{enumerate}
    \item Consider an $S^1$ sub-action of the 2-torus action, which on $T_{x_0}$ coincides with the Reeb flow;
    \item Lift this action to a Hamiltonian action on the symplectization $\R\times N$, generated by a Hamiltonian $H$;
    \item Show that $H$ is Morse--Bott and has a critical manifold over the singular toric fiber;
    \item Define a compatible $J_0$ that is invariant under the Hamiltonian $S^1$-action, and construct a $J_0$-holomorphic $\C^*$-action;
    \item Use the removal of singularities theorem to obtain holomorphic planes from the $\C^*$-action.
\end{enumerate}

\subsubsection{\bf Hamiltonian $S^1$-actions and holomorphic $\C^*$-actions}\label{subsec:Ham_C*_action}
We adapt a construction illustrated in \cite[Exercise 5.1.5]{mcduff2017introduction} to the setting of symplectizations instead of closed manifolds.
Given a toric action on a contact manifold $N$ (possibly with boundary) with an invariant contact form $\lambda$, let $X$ be a vector field generating an $S^1$ sub-action.
Then $X$ lifts to a Hamiltonian vector field in the symplectization
\[
\R_r\times N, \qquad \omega = d(e^r\lambda).
\]
Indeed {for $H:=e^r\cdot \iota_X\lambda$},  
$$\iota_X \omega = \iota_X d(e^r\lambda) = \mathcal{L}_X e^r\lambda - d\iota_X(e^r\lambda) = 0-d(e^r \cdot\iota_X\lambda) =-dH,\quad $$
Let $J_0$ be an almost complex structure that is invariant under this $S^1$-subaction and sends $\partial_r$ to the Reeb vector field $R$ of $\lambda$. We lift the $S^1$ action on $N$ to a $J_0$-holomorphic $\C^*$ action on the symplectization $\R\times N$, in the following way.
Denote by $\varphi^t$ the flow of $X$ 
lifted to $\R\times N$ in the obvious way, and denote by $\psi^s$ the  flow of $\nabla H := -J_0X_H = -J_0 X$. Since $J_0$ is invariant under the torus action, $\mathcal{L}_X J_0=0$ and thus $[X,\nabla H]=0$. In particular, the flows $\psi^s$ and $\varphi^t$ commute and therefore generate a $\C^*$ action, by
\[
e^{s+it}\in \C^* \qquad \mapsto \qquad \psi^s \varphi^t.
\]

\subsubsection{\bf $J_0$-holomorphic discs in $N$}
Now we apply the above construction of holomorphic $\C^*$ actions to construct holomorphic discs in the symplectization of a toric 3-manifold satisfying Setup~\ref{setup:construction}.

For context, recall that the Reeb vector field is given by
 \begin{equation*}
     R(q_1,q_2,x)=\frac{1}{Q(x)}\big( p_2'(x) \partial_{q_1} - p_1'(x)\partial_{q_2}\big),\quad \text{where}\quad Q(x):=p_1(x) p_2'(x) - p_2(x) p_1'(x) >0,
 \end{equation*}
(see equation (\ref{eqn:Qsign})). Consider the symplectization $\R_r\times N$ with the symplectic form
\[
\omega = d(e^r\lambda) = e^rdr\wedge(p_1(x)dq_1+ p_2(x)dq_2) + e^r(p_1'(x)dx\wedge dq_1 +p_2'(x)dx\wedge dq_2).
\]
We define a compatible almost complex structure  by 
\begin{equation}\label{eq:acs_for_construction}
    J_0\partial_r = R, \qquad J_0V_1 = V_2.
\end{equation}
Here $(V_1,V_2)$ is the trivialization of $\xi$ defined in equation (\ref{eqn:tau0}) in \textsection\ref{ss:MBindex}.
We will now use the $\C^*$-actions constructed in \S\ref{subsec:Ham_C*_action} to construct a $J_0$-holomorphic planes positively asymptotic to any Reeb orbit in $T_{x_0}$.
\begin{proof}[Proof of Proposition~\ref{prop:explicit_construction}]    
Using the notations of \S\ref{subsec:Ham_C*_action}, let $X=-\partial_{q_2}$, and note that $X$ agrees with the Reeb vector field on the regular toric fiber $T_{x_0}$. Then $X$ lifts to a Hamiltonian vector field on $\R_r\times N$, generated by the Hamiltonian
    \[
    H:\R_r\times N\rightarrow \R, \quad H(r,q_1,q_2,x) = -e^r p_2(x).
    \]
    In particular notice that $H$ is Morse--Bott, increasing in $x$, and has a single critical manifold of dimension 2 of index 0 at $\{x=0\}$. 
    The almost complex structure $J_0$ defined in (\ref{eq:acs_for_construction}) is invariant under the torus action and sends $\partial_r$ to $R$, so we have a $J_0$-holomorphic $\C^*$-action given by the Hamiltonian and gradient flows of $H$. 
    We calculate 
  \[
    \nabla H =  - p_2(x)\partial_r
- p_2'(x)\partial_x.
    \]    
    From this observe that $\nabla H$ is independent of the $q_1$ and $q_2$ coordinates, so its flowlines have constant $(q_1,q_2)$ values. Additionally observe that $\R\times T_{x_0}$ is invariant under the flow of $\nabla H$, on which it coincides with 
    $-p_2(x_0)\cdot \partial_r$ since $p_2'(x_0)=0$ and $p_2(x_0)<0$.

\begin{figure}
\centering
\includegraphics[scale=.3]{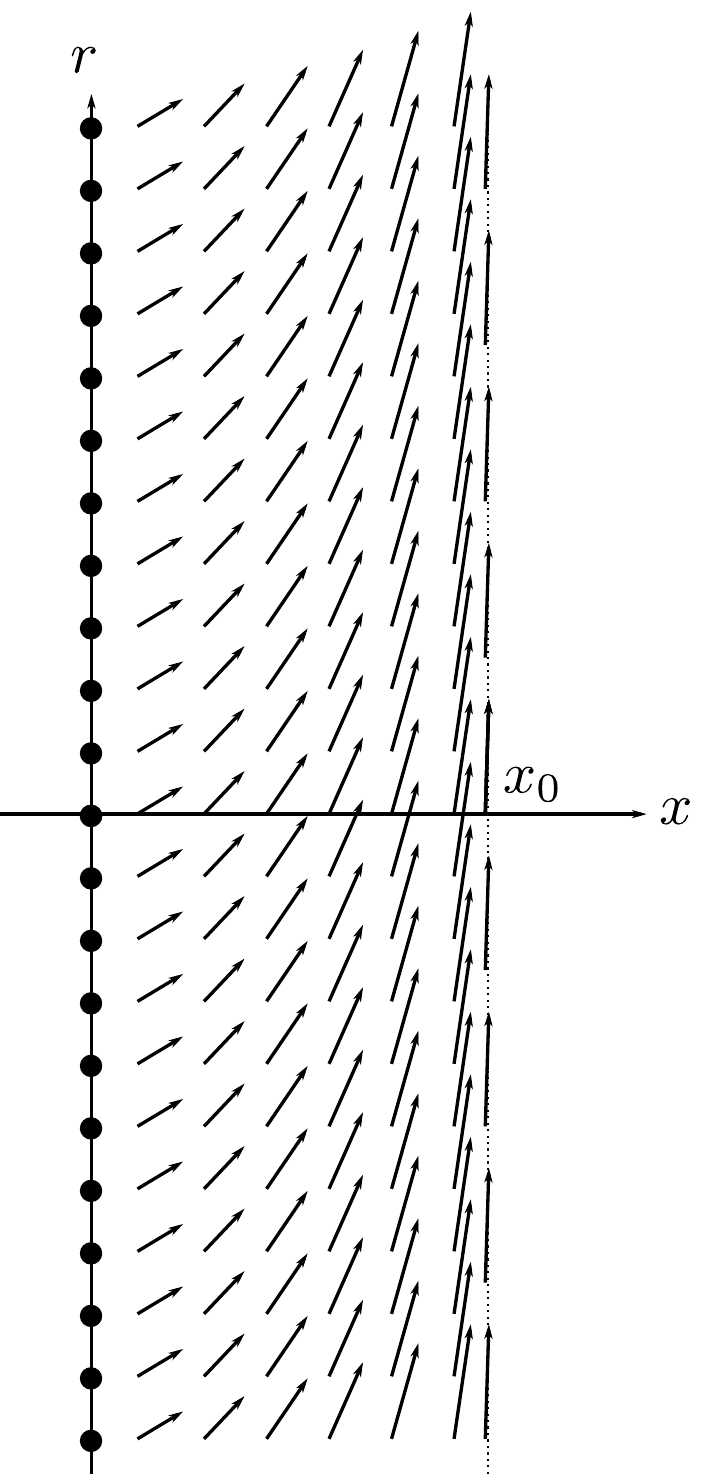}
\caption{The vector field $\nabla H$ in the $(x,r)$ plane.}
\label{fig:gradH}
\end{figure}

   Since $\R\times T_{x_0}$ is a separating hypersurface,
    for every point $(r,q_1,q_2,x)$ with $x\in (0,x_0)$, its $\mathbb C^*$ orbit
    \[
    u(s,t):= \psi^s\varphi^t (r,q_1,q_2,x)
    \]
    is contained in 
$\R\times T^2 \times (0,x_0) $. Observe that the gradient flowlines of $\nabla H$ limit in the positive direction towards $\{\infty\}\times T_{x_0}$. In the negative direction they limit towards a critical point on $\R\times T_0$. (See Figure~\ref{fig:gradH}). Recall also that $T_0$ is not a torus fiber, but a circle, since the fiber over $P(0)$ is singular. Furthermore, the slope that is collapsed in $T_0$ is the $\partial_{q_2}$ slope. Therefore each point in $T_0$ is a fixed point under the Hamiltonian circle action $X$. 

Fix a point $(r^\star,q_1^\star,q_2^\star,x^\star)  \in \R\times N.$
Then, the  $\mathbb C^*$ orbit is 
 \begin{equation} \label{eq:C*-orbit}     
u(s,t)=(r(s),q_1^\star,-t+q_2^\star,x(s)), \quad t\in S^1, \ s\in \mathbb R,
 \end{equation}
where $(r(0),x(0))=(r^\star,x^\star)$ and the functions $r(s), x(s)$ follow a flow line of $\nabla H$, which is determined by the function $p_2(x)$. In particular,
$$\lim_{s\rightarrow +\infty}r(s)=+ \infty, \quad \lim_{s\rightarrow -\infty}r(s)=r^-, \text{ for some }r^-\leq r^\star,$$
$$\lim_{s\rightarrow +\infty}x(s)=x_0, \quad \lim_{s\rightarrow -\infty}x(s)=0. $$
Therefore, 
as $s\rightarrow \infty$, $ u$ is positively asymptotic to a Reeb orbit $\gamma$ in $T_{x_0}$, while, 
as $s\rightarrow -\infty$, $ u$ converges to the critical point $(r^-,q_1^0,0)$ in the critical manifold $\{x=0\}$.\\

We would like to show that we can extend 
$ u$ to a $J_0$-holomorphic curve $u:\mathbb C\to \R\times N$ via removal of singularities. To apply the removal of singularities theorem, we restrict $u$ to the punctured unit disk $\mathbb{D}\setminus \{0\}$ and show that this restriction has finite energy. Indeed,
\[
    \begin{split}     
    \int_{\mathbb{D}\setminus\{0\}} u^*\omega  & = \int_{\mathbb{D}\setminus\{0\}} \omega(\partial_su,\partial_t u) \ ds\ dt= \int_{\mathbb{D}\setminus\{0\}}   \omega(\partial_su,X) \ ds\ dt \\
    &=     \int_{\mathbb{D}\setminus\{0\}}    dH(\partial_su)  \ ds\ dt = \int_{\mathbb{D}\setminus\{0\}} \frac{d}{ds}(H\circ u)\ ds\ dt \\
    & = \int_{S^1}  H(u(1,t)) - H(u(0,t)) dt =H(u(1,t_0)) - H(u(0,t_0))<\infty,  
    \end{split}
    \]
    where $t_0$ can be any choice of value in $[0,2\pi)$.

    Thus, by the removal of singularities theorem (see e.g. \cite[Theorem 4.1.2]{mcduff2012j}), 
    $ u|_{\mathbb{D}\setminus \{0\}}$ extends  to a smooth $J_0$-holomorphic $u|_{\mathbb{D}}: \mathbb{D} \rightarrow \R\times N$ which agrees with $u|_{\mathbb{D}\setminus \{0\}}$. Thus $u$ extends from a $J_0$-holomorphic map of $\C^*$ to a $J_0$-holomorphic plane
    $u:\C\rightarrow \R \times N$, which is positively asymptotic to the Reeb orbit $\{(q_1^\star,-t+q_2^\star)\mid t\in S^1\}$  in  $T_{x_0}$.
    
    Finally, we show that $u$ is embedded. Away from $(0,0)$ it has no critical points because $X$ and $\nabla H$ are non-vanishing. At $(0,0)$, we analyze a neighborhood of the curve and look at the link of the potential singularity.
    Based on the formulas $X=-\partial_{q_2}$ and $  \nabla H =  - p_2(x)\partial_r
- p_2'(x)\partial_x$, we have that the image of $u:\C\to \R\times N$ is the subset 
given by equation \eqref{eq:C*-orbit}, where $r(s),x(s)$ lie on the chosen compactified flowline of $\nabla H$.  
    Since this is simply the product of the compactified flowline with the circle (in the $q_2$ factor), the link of the potential singularity at $(r^-,q_1^\star,q_2^\star,0)$ is $(r(-K), q_1^\star, t, \varepsilon),$ {for some $K\gg 0$ such that $x(-K)=\varepsilon$}. This link is an unknot, viewed in the boundary of a small ball around $(r^-,q_1^\star,q_2^\star,0)$. Thus, $u$ is smoothly embedded at $(r^-,q_1^\star,q_2^\star,0)$ as well.
    
    Note that this construction holds for any value of $q_1^\star\in S^1$. Thus, we obtain an $S^1$ family of embedded $J_0$-holomorphic planes, each positively asymptotic to a different Reeb orbit in the $S^1$-family of Reeb orbits in $T_{x_0}$.
\end{proof}

\subsection{Regularity of the Morse-Bott $J_0$-holomorphic planes}
Now we would like to show that the $J_0$-holomorphic planes we have just constructed are regular. To do this, we will use an automatic transversality result. In order to apply this result, certain indices of our curve must satisfy a particular inequality. To compute the values involved in this inequality, we start by computing the relative Chern and self-intersection numbers, and then do further index calculations. These characteristic numbers will also be helpful later on.
 
Recall, from (\ref{eqn:tau0}), that on $Y_{(0,k]}=N-T_0$, where $N$ is diffeomorphic to a solid torus with a toric contact structure satisfying the conditions of Setup~\ref{setup:construction},  we have a trivialization $\tau_0$ of the contact planes via 
\begin{equation*}
    \tau_0:=\langle V_1=\partial_x,V_2=\frac{1}{Q(x)}(p_1(x)\partial_{q_2}-p_2(x)\partial_{q_1}\rangle.
\end{equation*}  

\begin{lemma}\label{lem:c_tau_shira_plane} The relative Chern number and self-intersection number with respect to the trivialization $\tau_0$ of each of the $J_0$-holomorphic planes $u$ constructed in Proposition~\ref{prop:explicit_construction} are
\begin{align}
 c_{\tau_0}(u) &= 1. \\
Q_{\tau_0}(u) &= 0 
    \end{align}
\end{lemma}

\begin{proof}
    To compute $c_{\tau_0}(u)$, we project the image of $u$ to $N$, and then count zeros (with signs and multiplicities) of a vector field $V$ along $\image(\pi_N\circ u)$ such that $V\in \xi$ and $V$ is a trivial non-vanishing section with respect to the trivialization $\tau_0$ when restricted to $\gamma$. Consider the vector field $V:= x \partial_x$, which is well defined on the image of $\pi_N\circ u$. Since $\lambda = p_1(x)dq_1+p_2(x)dq_2$, $V\in \xi=\ker(\lambda)$ and since it is a non-zero multiple of $V_1$, it is a trivial non-vanishing section of $\xi$ when restricted to $\gamma$. Observe that $V$ has an isolated zero only at $x=0$. Finally, $V$ is tangent to the projection $\pi_N\circ u$ of $u$, since $\partial_s u = -p_2(x)\partial_r
- p_2'(x)\partial_x.$ Therefore, we may apply the Poincar\'e-Hopf theorem and obtain that the index of the zero of $V$ at $x = 0$ is equal to the Euler characteristic of $u$, which is 1.

    To compute the self-intersection number $Q_{\tau_0}(u)$, we push $u$ off itself along the vector field 
    $$V'=\frac{1}{Q(x_0)}\left(p_1(x_0)\partial_{q_2}-p_2(x_0)\partial_{q_1}\right).$$ 
    Along the asymptotic boundary, $V'$ agrees with the vector field $V_2$ from $\tau_0$. Since $p_2(x_0)<0$ and $Q(x_0)>0$, $V'$ has a positive $\partial_{q_1}$ component. Since the image of $u$ was contained in a constant slice $\{q_1 = q_1^0\}$ (see the proof of Proposition~\ref{prop:explicit_construction}), the push-off of $u$ along $V'$ will be disjoint from $u$. Therefore $Q_{\tau_0}(u)=0$.
\end{proof}

\begin{definition}[Perturbed Conley-Zehnder indices] 
Let $\gamma$ be a degenerate orbit lying in a Morse-Bott family $T$ of orbits and let $\tau$ be a trivialization of the contact structure along $\gamma$. Following \cite[\S3.2]{wendl2010automatic}, we define perturbed Conley-Zehnder indices 
\[
CZ_\tau(\gamma + \epsilon):= CZ_\tau(-L_\gamma + \epsilon)
\]
where $L_\gamma$ is the asymptotic operator of $\gamma$, see Definition~\ref{def:asymptotic_operator}. (Note that our sign choice for the asymptotic operator is opposite than Wendl's.) Here by $CZ_\tau(-L_\gamma+\epsilon)$ we mean the Conley--Zehnder index of the path of symplectic matrices generated by the symmetric operator $-L_\gamma+\epsilon\cdot \operatorname{Id}$.
We also define the parity
\begin{equation}\label{e:parity}
    p(\gamma + \epsilon):= CZ_\tau(\gamma + \epsilon) \mod 2.
\end{equation}
By spectral flow arguments \cite{robbin1995spectral}, 
\begin{equation}\label{eq:diff_in_pert_CZ}
    CZ_\tau(\gamma - \epsilon) - CZ_\tau(\gamma + \epsilon) = \dim \ker (L_\gamma) = \dim {T} -1,
\end{equation}
and $CZ_\tau(\gamma \pm \epsilon)$ remains the same as we move $\gamma$ in the Morse-Bott family $T$. Note that here, the Morse-Bott family $T$ is that of parametrized orbits, so in our situation this is the torus $T_{x_0}$.

\end{definition}
\begin{definition}[Morse-Bott Fredholm index, (1.1) in \cite{wendl2010automatic}]
For a pseudoholomorphic curve $u: \dot{\Sigma} \to W^{2n}$ with asymptotic ends at Morse-Bott orbits,
\begin{equation}\label{eq:MB_Fredholm_ind}
    \ind(u; \mathbf{\delta}) = (n-3) \chi(\dot{\Sigma}) + 2c_\tau(u) + \mu_\tau(u;\mathbf{\delta}),
\end{equation}
where
\begin{equation}
    \mu_\tau(u;\mathbf{\delta}):= \Sigma_{z\in \Gamma^+} CZ_\tau(\gamma_z + \mathbf{\delta}_z) - \Sigma_{z\in \Gamma^-} CZ_\tau(\gamma_z - \mathbf{\delta}_z).
\end{equation}
and $\delta_z$ is a small non-zero number that is positive if $\gamma_z$ is a constrained end of $u$  and is  negative when $\gamma_z$ is  unconstrained \cite[\S3.2]{wendl2010automatic}.  Here $\Gamma_\pm$ denote the set of positive/negative punctures.
\end{definition}

\begin{lemma}\label{lem:shira_plane_index} The Morse--Bott Fredholm  index of each of the $J_0$-holomorphic planes $u$ constructed in Proposition~\ref{prop:explicit_construction} is
    \[ \ind(u) = 1.\]
\end{lemma}

\begin{proof}
  The plane $u$ is embedded with one positive constrained end on a degenerate orbit $\gamma$  in the Morse-Bott family $T_{x_0}$. Let $\tau_0$ be the trivialization defined in (\ref{eqn:tau0}).  By (\ref{eq:MB_Fredholm_ind}), 
  \[
  \ind(u) = (2-3)\cdot 1 + 2 c_{\tau_0}(u)+ CZ_{\tau_0}(\gamma +\delta) \quad \text{for}\quad \delta >0.
  \]
  By Lemma~\ref{lem:c_tau_shira_plane}, $c_{\tau_0}(u)=1$. To compute $ CZ_{\tau_0}(\gamma +\delta)$ we will use Lemma~\ref{lem:RS_of_gammah}, which states that the Robbin-Salamon index of $T_{x_0}$ is $\frac{1}{2}$. 
  Note that by Corollary~\ref{cor:CZ}, after a nondegenerate perturbation,  the Conley--Zehnder index of the hyperbolic orbit in $T_{x_0}$ is computed to be 0. We will see that the same holds for $CZ_{\tau_0}(\gamma+\delta)$.

  Indeed, since $RS_{\tau_0} (\gamma):=RS_{\tau_0}(-L_{\gamma})$, the difference between $RS_{\tau_0}(-L_\gamma)$ and any of the perturbations $CZ_{\tau_0}(-L_\gamma\pm\epsilon)$ is at most $\frac{1}{2}\dim\ker(L_\gamma)=\frac{1}{2}$ (see for example \cite[p.7]{ishikawa2016spectral} and \cite{robbin1995spectral}). As the Robbin--Salamon index is a half-integer, it follows from  (\ref{eq:diff_in_pert_CZ}) that indeed
  \begin{equation}
       CZ_{\tau_0}(\gamma+\delta)= 0.
  \end{equation} Overall,
  \[
  \ind(u_0) = -1+2+0=1.
  \]
\end{proof}

\begin{lemma}\label{lem:u_0_regular} 
     Each of the $J_0$-holomorphic planes $u$ from Proposition~\ref{prop:explicit_construction} is regular.
\end{lemma}
\begin{proof}
     We wish to apply Wendl's automatic transversality theorem \cite[Theorem 1]{wendl2010automatic} to the plane $u$, which guarantees that the linearized operator $D\bar \partial_{J_0}(u)$ is surjective, i.e. that $u$ is regular, provided that 
\begin{equation}\label{eq:Wendls_transversality_condition}
\ind(u)>c_N(u)+Z(du).
\end{equation}
In \eqref{eq:Wendls_transversality_condition} the first term is the Fredholm index, which is computed in Lemma~\ref{lem:shira_plane_index} to be
\[
\ind(u)= 1.
\]
The second term is the normal first  Chern number, which can be calculated by
    \begin{equation}\label{eq:c_N_vs_ind}
         2 c_N(u) = \ind(u) -2+2g(\mathbb{C})+\#\Gamma_{0}  = \ind(u) -1=0.
    \end{equation}
Here the set of Reeb orbits with even parity~\eqref{e:parity} is denoted by $\Gamma_{0}$, and the plane is asymptotic to one Reeb orbit with even parity, $\Gamma_{0}= \{\gamma\}$, c.f. \cite[\S1.1, (1.2), \S3.2]{wendl2010automatic}.  (In the case of an immersed closed pseudoholomorphic curve in a four dimensional symplectic manifold, the normal first Chern number agrees with the first Chern number of the normal bundle. To define it in the more general context of asymptotically cylindrical curves, the above numerical formula is typically used as the definition, see \cite[(1.2), \S3.5]{wendl2010automatic}, with the full derivation carried out in \cite[\S 3.4]{wendl-int}.)

The third term $Z(du)$ is a count with multiplicities of the zeros of $du$, which vanishes for immersed curves.  We know that $u$ is embedded (and hence immersed) by Proposition~\ref{prop:explicit_construction}.
Thus (\ref{eq:Wendls_transversality_condition}) is satisfied for $u$ and guarantees that $D\partial_{J_0} (u)$ is surjective.
\end{proof}

\subsection{Persistence under perturbations of the data} \label{subsection persistence}
Our goal for this section is to show that each of the $J_0$-holomorphic planes constructed in \S\ref{subsection MorseBott plane} survives under small deformations of the (non-generic) contact form and almost complex structure. 
\begin{remark}\label{rem:gammah}
    In this section we fix the Reeb orbit that persists after the perturbation of $\lambda$ to $\lambdapert$ in \S\ref{ss:perturb} corresponding to the minimum \eqref{cos}, in which case we denote it by $\gamma_h$, cf. Lemma~\ref{lem:morsebott_perturbation}.
    We stress the slight abuse of notation: $\gamma_h$ is not a positive hyperbolic orbit for $\lambda$, but it is the unique degenerate orbit from the family $T_{x_0}$ that becomes positive hyperbolic after the specific perturbation fixed in \S\ref{ss:perturb}. 
\end{remark}

\begin{proposition}\label{prop:existence_survives_pert}
    Let $\lambda$ be a toric invariant contact form as in Setup~\ref{setup:construction}. Let $J_0$ be the $\lambda$-compatible almost complex structure and $u_0$ be the $J_0$-holomorphic plane from Proposition~\ref{prop:explicit_construction} that is positively asymptotic to $\gamma_h$. Let $\lambdapert$ be the $L$-nondegenerate perturbation of $\lambda$ from \S\ref{ss:perturb}. Then for any $\lambdapert$-compatible almost complex structure $\Jpert$ close enough to $J_0$, there exists a regular $\Jpert$-holomorphic plane $\upert$ positively asymptotic to $\gamma_h$ and close to $u_0$.
\end{proposition}

\begin{remark}
    One can think of Proposition~\ref{prop:existence_survives_pert} as a special case of gluing of cascades, such as in \cite[Theorem 1]{yao2022cascades}, when the cascade is trivial and consists of a single holomorphic plane. Moreover, Proposition~\ref{prop:existence_survives_pert} shares some similarities with Wendl's strong implicit function theorem in \cite[Proposition 7]{wendl2010open}. The latter is a much stronger result than required for our purposes. For completeness we include a direct proof in our simpler setting.
\end{remark}

Our strategy for the proof of Proposition~\ref{prop:existence_survives_pert} is the following:
\begin{enumerate}
    \item Construct a parametric moduli space of planes in symplectizations corresponding to a path of contact forms and varying almost complex structures.
    \item Show that this moduli space is smooth at $(0,u_0)$ and has dimension $\ind(u_0)+1$.
    \item Deduce that $(0,u_0)$ is not isolated with respect to the parameter value, and thus sufficiently small perturbations of the data admit regular pseudoholomorphic planes close to $u_0$.
\end{enumerate}

Let $\lambda_s= e^{s\varepsilon F_L}\lambda$, and observe that $\lambda_0=\lambda$ and $\lambda_1=\lambdapert$.
Let $\underline{J}:=\{J_s\}_{s\in[0,1]}$ be a smooth path of almost complex structures on $\R\times Y$ such that $J_s$ is $\lambda_s$-compatible. Consider the parametric moduli space 
of homologous planes asymptotic $\gamma_h$ in the symplectizations of $\lambda_s$:
\[
\mathcal{M}(\{J_s\}):= \left\{ (s,u_s):\ \bar\partial_{J_s}u_s=0,\  u_s\text{ is  a plane positively asymptotic to }\gamma_h  \right\}.
\] 
For each $s\in [0,1]$ we consider also the ``sliced" moduli space 
\[
\mathcal{M}_s:= \left\{ u:\ \bar\partial_{J_s}u=0,\  u\text{ is a plane positively asymptotic to }\gamma_h  \right\}.
\]
\begin{remark}
We stress that our parametric moduli space differs from the standard scenario in the literature, since the almost complex structures $J_s$ are not all cylindrical with respect to a fixed contact form. In other words, the contact form $\lambda_s$ varies as well with the parameter. As mentioned earlier,  a similar setting appears in Wendl's strong implicit function theorem in \cite[Proposition 7]{wendl2010open}.

 As explained in \cite[Remark 8.7]{wendl2016lectures}, one can define the parametric moduli space $\mathcal{M}(\{J_s\})$ and the associated notion of parametric regularity, with an explicit description of the associated linearized operator $D\bar\partial_{\{J_s\}_{s\in P}}$.  In particular,  $D\bar\partial_{\{J_s\}}(j,u,s)$ is the sum of $D\bar\partial_{J_s}(j,u)$ with an additional term defined on $T_sP$, where $P=[0,1]$ is the parameter space. 
When these operators are surjective, the implicit function theorem guarantees that the parametric moduli space $\mathcal{M}$ and the slice $\mathcal{M}_s$ are smooth manifolds near $(s,u_s)$ and $u_s$ respectively, of dimensions given by the Fredholm indices.
\end{remark}

\begin{proof}[Proof of Proposition~\ref{prop:existence_survives_pert}]
    Our first step is to show that the parametric moduli space $\mathcal{M}(\{J_s\})$ is a smooth manifold in an open neighborhood of the point $(0,u_0)$. Indeed, this is an instance of ``honest regularity implies parametric regularity". More concretely, by Lemma~\ref{lem:u_0_regular}, the linearized operator $D\bar\partial_{J_0}(j,u_0)$ is surjective at $u_0$. Since the parametric operator $D\bar\partial_{\{J_s\}}(j,u,s)$ is the direct sum of $D\bar\partial_{J_s}(j,u)$ with an additional term and the co-domain is the same for the parametric and non-parametric operators, the parametric operator must be surjective at $(0,u_0)$ as well.  We conclude that  $\mathcal{M}(\{J_s\})$ is a smooth manifold near $(0,u_0)$ of dimension $$
    \dim(\ker(D\bar\partial_{\{J_s\}}(j,u_0,0))) = \dim (\ker(D\bar\partial_{J_0}(j,u_0)))+1 = \ind(u_0)+1.
    $$
    In particular, there exists $s$ arbitrarily close to 0 such that the  slice $\mathcal{M}_s$ is not empty and there exists $u_s\in \mathcal{M}_s$ that is $C^\infty$ close to $u_0$. 
    
   We now use automatic transversality to show that $u_s\in \mathcal{M}$ is regular for $s$ close enough to 0. 
   For that purpose, we first notice that $c_N(u_s)=c_N(u_0)$ since  $u_s$ is close to $u_0$ and has the same asymptotics. Now, it follow from \eqref{eq:c_N_vs_ind} that the indices coincide:  
    \[
    \ind(u_s)=\ind(u_0)\qquad \text{and}\qquad c_N(u_s)=c_N(u_0).
    \]
    Finally, observe that we can choose an embedded $u_s$, {since it is close to $u_0$, which is embedded}.

   First we observe that $u_s$ is embedded by way of the ECH index inequality, cf. Remark \ref{rem:I}.  In particular, $Z(du_s)=0$. This is because we have that both the Fredholm and ECH indices of $u_s$ are one. To deduce this, observe that 
    \[
    \ind(u_s)=\ind(u_0)\qquad \text{and}\qquad c_N(u_s)=c_N(u_0).
    \]
    This can be seen directly from the fact that $u_s$ is close to $u_0$ and has the same asymptotics.  Hence $c_N(u_s)=c_N(u_0)$ and then the equality for the Fredholm indices follow from (\ref{eq:c_N_vs_ind}).  To complete the computation that the ECH index is one, we observe that $Q_{\tau_0}(u_s)=0$ and $c_{\tau_0}(u_s)=1$ by Lemma \ref{lem:Q and Chern upert}.
    
    Thus, we have that Wendl's transversality condition (\ref{eq:Wendls_transversality_condition}) holds, hence his automatic transversality theorem \cite[Theorem 1]{wendl2010automatic} guarantees that $u_s$ is regular, similarly as in the proof of Lemma~\ref{lem:u_0_regular}.     
\end{proof}

\subsection{Index computations for the plane}
\label{sec:iplane}
We conclude by computing the relevant indices of the $\Jpert$-holomorphic plane constructed in \S\ref{subsection persistence}.

\begin{lemma}
\label{lem:Q and Chern upert}
    Let $N$ be diffeomorphic to a solid torus with a toric contact structure satisfying the conditions of Setup~\ref{setup:construction}.
    For any pseudoholomorphic curve $u$ in $\R\times N$ with a single positive asymptotic end at $\gamma_h$, we have that $Q_{\tau_0}(u) = 0$ and $\langle c_{\tau_0}(\omega), [u]\rangle = 1$, where $\omega$ is the symplectic form associated to the symplectization.    
\end{lemma}

\begin{proof}
Observe that $H_2(N)=0$. Therefore, the map $H_2(N, \gamma_h) \to H_1(\gamma_h)$ in the long exact sequence of a pair is injective, so any two relative homology classes in $H_2(N, \gamma_h)$ with the same simple positive asymptotic end at $\gamma_h$ and no negative ends are homologous.

Let $u_0$ be the pseudoholomorphic plane asymptotic to $\gamma_h$, from Proposition~\ref{prop:existence_survives_pert}. Since $Q_{\tau_0}(u)$ and $\langle c_{\tau_0}(\omega), [u]\rangle$ depend only on the underlying relative homology class of $u$, we have that $Q_{\tau_0}(u) = Q_{\tau_0}(u_0) = 0$ and $\langle c_{\tau_0}(\omega), [u]\rangle = \langle c_{\tau_0}(\omega), [u_0]\rangle=1$ by Lemma~\ref{lem:c_tau_shira_plane}.
\end{proof}

The following lemma is the missing part of the proof of Theorem~\ref{theorem:construction}.

\begin{lemma}\label{lem:perturbed_plane_index}
Let $\upert$ be the $\Jpert$-holomorphic plane from Proposition~\ref{prop:existence_survives_pert}, associated to a perturbed $L$-nondegenerate contact form $\lambdapert$. We have 
\begin{align}
        \ind(\upert) &= I(\upert) = 1 \\
        \Jplus(\upert) &= 0
    \end{align}

\end{lemma}
\begin{proof}
    We use the notation of \S\ref{subsec:AT_in_ECH}. By Lemma~\ref{lem:Q and Chern upert}, and Corollary~\ref{cor:CZ}, we can compute the indices as follows:
    \begin{equation}
        \begin{aligned}
            \ind(\upert) &= -\chi(\upert) + 2c_{\tau_0}(\upert) + CZ_{\tau_0}(\gamma_h) =  -1+2+0 = 1, \\
            I(\upert) &= c_{\tau_0}(\upert) + Q_{\tau_0}(\upert) + CZ_{\tau_0}(\gamma_h) =  1+0+0 = 1,\\
            \Jplus(\upert) &= 2(g-1) + m_{\gamma_h} + |\gamma_h| = 2(0-1)+1+1 = 0.
        \end{aligned}\qedhere
    \end{equation}
\end{proof}
{We conclude with the proof of the main theorem for this section, which is an immediate consequence of the above results.
\begin{proof}[Proof of Theorem~\ref{theorem:construction}]
    Propositions~\ref{prop:explicit_construction} and \ref{prop:existence_survives_pert} guarantee the existence of the planes $u_0$ and $\upert$ positively asymptotic to $\gamma_h$. The Fredholm, ECH and $\Jplus$ indices of $\upert$ are computed in Lemma~\ref{lem:perturbed_plane_index}.
\end{proof}
}
\section{Obstructing pseudoholomorphic curves from the right}  \label{section constraints}
In this section, we will obtain significant constraints on $\Jpert$-holomorphic curves in the case when the contact toric manifold has moment map image traversing an angle strictly greater than $\pi$.

In \S\ref{sec:asymptotic}, we will analyze the possible asymptotic behavior of $\Jpert$-holomorphic curves with a positive asymptotic end at $\gamma_h$ (with multiplicity $1$). One possibility of asymptotic behavior is that the end approaches entirely from one side of the torus $T_{x_0}$.

\begin{definition} \label{def:fromleftfromright}
    Let $(Y,\xi)$ be a contact toric manifold. Let $\lambda$ be any contact form and $\gamma$ a periodic Reeb orbit contained in a torus fiber $T_x$. Let $J$ be any $\lambda$-compatible almost complex structure on $\R\times Y$, and $u$ a $J$-holomorphic curve positively asymptotic to $\gamma$. Let $\delta>0$.  Recalling Notation~\ref{notation:Y_interval}, we say that:
    \begin{itemize}
        \item $u$ approaches $\gamma$ \emph{from the left} (resp. \emph{from the right}) if there exists $S\gg0$ such that the connected component of $u\cap [S,\infty)\times Y$ which approaches $\gamma$ is contained in $\R\times Y_{[x-\delta,x]}$ (resp. $\R\times Y_{[x,x+\delta]}$); 
        \item  $u$ approaches $\gamma$ \emph{from one side} if it approaches from the left or from the right; 
        \item $u$ and $u'$ approach $\gamma$ \emph{from the same side} if they both approach from the left or both approach from the right.
    \end{itemize}
\end{definition}

\begin{notation} \label{not:relhomclass} Given a solid torus portion of a contact toric manifold with one singular toric fiber and one boundary torus fiber, endowed with a Reeb vector field tangent to the fibers. If the collapsing slope of the singular fiber is realized as the slope of the Reeb orbits in a regular fiber $T_{x_0}$ as in Setup~\ref{setup:construction}, let $\gamma_h$ be the hyperbolic Reeb orbit in $T_{x_0}$ with respect to the perturbed contact form $\lambdapert$. Because the second homology of a solid torus is trivial, there is a unique relative homology class represented by a surface lying in the solid torus with (asymptotic) boundary on $\gamma_h$ with multiplicity $1$. We will denote this relative homology class in $H_2(Y,\gamma_h)$ by $E$.
\end{notation}

Note that the $J_0$-holomorphic and $\Jpert$-holomorphic curves constructed in section~\ref{section plane exists} both lie in this class $[u_0]=[\upert]=E$.

\begin{theorem} \label{thm:no curve from right}
Let $\lambda$ be a toric contact form satisfying Setup~\ref{setup:morethanpi}, $\lambdapert$ its $L$-nondegenerate perturbation satisfying Setup~\ref{setup:deltaperturbation}, $\Jpert$  any $\lambdapert$-compatible almost complex structure, and $u$ a connected $\Jpert$-holomorphic curve, with exactly one positive asymptotic end which is at the closed positive hyperbolic Reeb orbit $\gamma_h$ in the torus fiber $T_{x_0}$. 
Then $u$ has no negative ends, $[u]=E\in H_2(Y,\gamma_h)$, $u$ has Fredholm index $1$, and $u$ does not approach $\gamma_h$ from the right.
\end{theorem}

To obtain this obstruction, we will use positivity of intersections to show that the curve cannot go very far to the right (i.e. $\image(u)\subset \R\times Y_{[0,x_0+\delta)}$), and cannot have negative ends. Then we will use an action argument to show that the curve cannot approach from the right, to conclude Theorem~\ref{thm:no curve from right}. 

We start with constraints coming from the positivity of intersection argument.
We consider a contact toric manifold, possibly with boundary $(Y,\lambda)$, with moment map image curve $P(x)=(p_1(x),p_2(x))$ oriented counterclockwise and an $L$-nondegenerate perturbed contact form $\lambdapert$ as in \textsection\ref{subsec:nondeg_contact_form}. We will focus on torus fibers $T_x$ containing closed Reeb orbits. These include unperturbed tori which are foliated by degenerate Reeb orbits and tori which are at the center of the perturbation which contain two nondegenerate Reeb orbits. In both cases, recall that the slope of the Reeb orbits in the torus is proportionate to the normal vector to the curve $P$ in the moment map image. For Reeb orbits to be closed, this slope must be rational, and we will denote it by $(p,q)$.

Let $\Jpert$ be an almost complex structure on $\R\times Y$ compatible with $\lambdapert$. Our goal in this section will be to constrain $\Jpert$-holomorphic curves $u$ with a specified positive end.

Denote by $\pi_Y: \R\times Y\to Y$ the projection and fix the following notations:
\[\widetilde{C}_x:=(\R\times T_x)\cap \image(u) \quad\text{ and }\quad C_x:=\pi_Y(\widetilde{C}_x)\subset T_x.\]
When $u$ intersects $\R\times T_x$ transversely, $\widetilde{C}_x$ is a smooth curve. In general, $\widetilde{C}_x$ could be a graph. Denote the homology class $[C_x]\in H_1(T_x)=\Z^2$ by $(a,b)$. Note that there is an orientation ambiguity (between $(a,b)$ and $(-a,-b)$) which we resolve now by setting a convention.

\begin{convention}
 The orientation on $[C_x]$ is determined as follows.
 Orient the image of $u$ by $\Jpert$. Orient $\widetilde{C}_x$ as the boundary of $(\R\times Y_{[0,x]})\cap \image(u)$. This induces an orientation on $C_x$, which determines the sign for its homology class $(a,b)$.
\end{convention}

Observe that this convention depends on the orientation of the curve $P:[0,1]\to \R^2\setminus 0$, which we previously specified to be counterclockwise.

\begin{lemma}\label{l:posint}
Suppose $u$ is a $\Jpert$-holomorphic curve and $T_{x}$ is a torus fiber containing a Reeb orbit $\gamma$ of slope $(p,q)$. If $u$ intersects $\R\times T_x$ transversally at a point $A\in \R\times \gamma$ such that the oriented tangent vector to $C_x$ at $\pi_Y(A)$ in $T_x$ has slope $(\alpha,\beta)$, then $$p\beta-q\alpha > 0.$$ 
\end{lemma}

\begin{proof}
There exists the trivial $\Jpert$-holomorphic cylinder $v$ of the form $\R\times \gamma$ where $\gamma$ is a closed Reeb orbit contained in $T_{x}$ of slope $(p,q)$.
Since $u$ is also $\Jpert$-holomorphic, it must intersect $v$ positively. 
On $\R\times Y$, we have coordinates $(r,x,q_1,q_2)$. 
The positive orientation on $\R\times Y$, induced by the symplectic form $\omega = d(e^r\lambdapert)$, is represented by the ordered basis
$$(\partial_r, \partial_{q_1}, \partial_x, \partial_{q_2}).$$
Indeed, this can be checked by computing $\omega\wedge\omega =e^{2(r+\varepsilon F_L)}Q(x)dr\wedge dq_1\wedge dx\wedge dq_2$ where $Q(x)$ is positive as in Equation~\eqref{eqn:Qsign}.

The tangent space to $v$ is spanned by the oriented basis $(\partial_r, R=p\partial_{q_1}+q\partial_{q_2})$.
In order to write an oriented basis spanning the tangent space to $u$, we consider $\widetilde C_x$ oriented as the boundary of $(\R\times T_{[0,x]})\cap\image(u)$. By definition of boundary orientation, the outer normal to $(\R\times T_{[0,x]})\cap\image(u)$, which is of the form $\partial_x+c_1\partial_{q_1}+c_2\partial_{q_2}+c_3\partial_r$ and the oriented tangent to  $\widetilde{C}_{x}$, which is of the form $\alpha \partial_{q_1}+\beta\partial_{q_2}+k\partial_r$, give an ordered positive basis to the tangent space to $u$.  Overall we have positive bases given by 
\begin{align*}
    Tv &= \operatorname{span}\big(\partial_r\ ,\  R=p\partial_{q_1}+q\partial_{q_2}\big), \\
    Tu &= \operatorname{span}\big(\partial_x+c_1\partial_{q_1}+c_2\partial_{q_2}+c_3\partial_r\ ,\  \alpha \partial_{q_1}+\beta\partial_{q_2}+k\partial_r\big).
\end{align*}
Thus, by positivity of intersections, at any point where the images of $u$ and $v$ intersect,
$$(\partial_r\ ,\  p\partial_{q_1}+q\partial_{q_2}\ ,\ \partial_x+c_1\partial_{q_1}+c_2\partial_{q_2}+c_3\partial_r\ ,\  \alpha \partial_{q_1}+\beta\partial_{q_2}+k\partial_r) $$
must be related by a transformation of positive determinant to 
$$(\partial_r, \partial_{q_1},\partial_x,\partial_{q_2}).$$
Writing out this matrix,
$$\left[\begin{array}{cccc}  1&0&c_3&k\\0&p&c_1&\alpha \\ 0&0&1&0\\ 0&q&c_2&\beta\\ \end{array} \right]$$
we see its determinant is $p\beta-q\alpha$.
Thus we conclude that the slope $(\alpha,\beta)$ of the oriented tangent vector to $C_{x}$ at any point where $C_{x}$ intersects $\gamma$ (the Reeb orbit in $T_{x}$), must satisfy $p\beta-q\alpha> 0$.
\end{proof}

\begin{cor}\label{cor:posinthomology}
Suppose $u$ is a $\Jpert$-holomorphic curve and $T_{x}$ is a torus fiber containing a closed Reeb orbit $\gamma$ of slope $(p,q)$. If $u$ {has non-empty} intersection with $\R\times T_{x}$ and $[C_{x}]=(a,b)$, then $$pb-qa \geq0.$$ 
If $T_x$ is foliated by closed Reeb orbits, then $pb-qa=0$ if and only if $u$ is a trivial cylinder in $\R\times T_x$.
\end{cor}

\begin{proof}
    Suppose for contradiction that $pb-qa<0$. Note that $pb-qa$ is the oriented algebraic intersection number $I(\gamma,C_{x})$ in $T_{x}$, and hence it is the signed count of intersection points. To see $I((p,q),(a,b))=pb-qa$ one can homotope the curves to piecewise vertical and horizontal pieces as in Figure~\ref{fig:inttorus}. If $pb-qa<0$, there must exist a negatively signed intersection point between $\gamma$ and $C_{x}$. If the slope of $C_{x}$ at this point is $(\alpha,\beta)$, the sign of the intersection is the sign of $p\beta-q\alpha$. Thus there exists a point where $u$ intersects $v$ where $p\beta-q\alpha<0$ a contradiction to Lemma~\ref{l:posint}. 
    
    Similarly if $pb-qa=0$, then $C_{x}$ cannot intersect any Reeb orbit in $T_{x}$ transversely since any such intersection must have slope $(\alpha,\beta)$ with $p\alpha-q\beta >0$. If $T_x$ is foliated by Reeb orbits, the only way this is possible is if $C_x$ agrees precisely with a Reeb orbit of the foliation. Because $\Jpert$ is compatible with $\lambdapert$, the only $\Jpert$ holomorphic curve which intersects $T_x$ in a Reeb orbit is the trivial cylinder over that Reeb orbit.
\begin{figure}
    \centering
    \includegraphics[scale=.5]{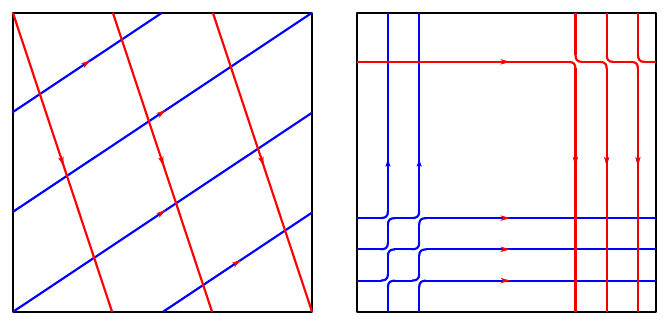}
    \caption{Computing the oriented intersection number of a curve of homology class $(p,q)$ and a curve of homology class $(a,b)$, via a homotopy to grid-like curves. }
    \label{fig:inttorus}
\end{figure}
\end{proof}

 The above corollary was previously stated in the proof of \cite[Prop.~10.12]{HutchingsSullivanT3}.

\begin{lemma}
\label{lem:sided_curve}
Suppose there is a closed interval $[x_-,x_+]\subset [0,1]$ such that $T_{x_-}$ and $T_{x_+}$ are regular torus fibers foliated by closed Reeb orbits for $\lambdapert$, and the Reeb vector field rotates by $\pi$ from $P(x_-)$ to $P(x_+)$. 
 For any connected $\Jpert$-holomorphic curve $u$ such that $\image(u)\cap Y_{[x_-,x_+]}$ contains no positive ends, and $\image(u)\cap (Y\setminus Y_{[x_-,x_+]})$ is non-empty, we have  $\image(u)\cap Y_{[x_-,x_+]}=\emptyset$.
\end{lemma}

\begin{proof}
Suppose for contradiction, $u$ has non-empty intersection with $Y_{[x_-,x_+]}$. Let $(a_-,b_-)$ denote the homology class of $C_{x_-}$ oriented as the boundary of $\image(u)\cap Y_{[0,x_-]}$, and $(a_+,b_+)$ the homology class of $C_{x_+}$ as the boundary of $\image(u)\cap Y_{[0,x_+]}$. 

Let $(p,q)$ denote the slope of the Reeb vector field in $T_{x_+}$, so $(-p,-q)$ is the Reeb slope in $T_{x_-}$. Let $\nu=\nu_{x_+}=(q,-p)$ be the oriented normal vector to $P$ at $x_+$. Then for every $x\in [x_-,x_+]$, the Reeb vector field in $T_x$ has a slope $(\sigma,\tau)$ such that 
\begin{equation} \label{eqnR}
(\sigma,\tau)\cdot \nu \geq 0
\end{equation}

Note that $u$ is not a trivial cylinder in $Y_{[x_-,x_+]}$ since it has no positive ends in $Y_{[x_-,x_+]}$.
Thus, by Corollary~\ref{cor:posinthomology}, 
\begin{equation}\label{eqn+}
-pb_-+qa_-\geq 0 \qquad\text{and}\qquad pb_+-qa_+ \geq 0,
\end{equation} and the above inequalities are strict unless {$u$ is disjoint from $\R\times T_{x_\pm}$}. Therefore, 
\begin{equation} \label{eqn-}
(a_-,b_-)\cdot \nu\geq 0 \qquad\text{and}\qquad(a_+,b_+)\cdot \nu\leq 0,
\end{equation}
where again the inequality is strict unless the intersection is empty.

In homology, $[\partial(\image(u)\cap Y_{[x_-,x_+]})]=0\in H_1(\R\times Y_{[x_-,x_+]})\cong H_1(T^2)\cong \Z^2$. Denote by $\{(\sigma_i,\tau_i)\}$ the slopes of the Reeb orbits of negative asymptotic ends of $u$ in $Y_{[x_-,x_+]}$, and by $m_i>0$  their multiplicity. Then in $H_1(\R\times Y_{[x_-,x_+]})$, we have
$$0 = (a_+,b_+)-(a_-,b_-) - \sum_i m_i(\sigma_i,\tau_i). $$
Taking the inner product with $\nu$ we obtain
$$0 = (a_+,b_+)\cdot \nu - (a_-,b_-)\cdot \nu - \sum_i m_i(\sigma_i,\tau_i)\cdot \nu.$$
By equations (\ref{eqnR}), (\ref{eqn+}), and (\ref{eqn-}), the right hand side of this equation is strictly negative unless $C_{x_-}$ and $C_{x_+}$ are empty. Since $u$ is connected, if $C_{x_-}$ and $C_{x_+}$ are empty, $\image(u)\cap Y_{[x_-,x_+]}$ must be empty.
\end{proof}

Now we return to the set-up of \textsection\ref{subsec:nondeg_contact_form} in the case that our contact toric manifold has angle greater than $\pi$.

\begin{proposition} 
\label{prop:close to left}
Let $\lambda$ be a toric contact form satisfying Setup~\ref{setup:morethanpi}, $\lambdapert$ its perturbation satisfying Setup~\ref{setup:deltaperturbation}, $\Jpert$ a compatible almost complex structure, and $u$ a connected $\Jpert$-holomorphic curve that has a single positive end asymptotic to the hyperbolic Reeb orbit $\gamma_h$ with multiplicity $1$ in the torus fiber $T_{x_0}$. Assume $u$ is not a trivial cylinder. Let $\upert$ be the $\Jpert$-holomorphic curve constructed in \textsection\ref{section plane exists}.

    Then 
    \begin{itemize}
        \item $\image(u)\subset \R\times Y_{[0,x_0+\delta)}$,
        \item $u$ has no negative ends,
        \item the relative homology class of $u$ is $E$ (see Notation~\ref{not:relhomclass}), and
        \item $\ind(u)=I(u)=1$
    \end{itemize}
\end{proposition} 

\begin{proof}
    To see that $\image(u)\subset \R\times Y_{[0,x_0+\delta)}$, simply apply Lemma~\ref{lem:sided_curve} where $x_-=x_0+\delta$ and $x_+$ are as in Setup~\ref{setup:deltaperturbation} to conclude $u$ is disjoint from $\R\times Y_{[x_0+\delta,x_+]}$. Since $u$ is connected and intersects $\R\times Y_{[0,x_0+\delta)}$ nontrivially near its positive asymptotic end, $\image(u)\subset \R\times Y_{[0,x_0+\delta)}$.

    Now we assume $u$ is not a trivial cylinder, and show that $u$ has no negative ends. Note that because every Reeb orbit in $Y_{(x_0-\delta,x_0+\delta)}$ has action larger than $\mathcal{A}(\gamma_h)$, $u$ cannot have negative ends in $Y_{(x_0-\delta,x_0+\delta)}$. Thus any potential negative end is in $Y_{[0,x_0-\delta]}$. Because of the assumptions of Setup~\ref{setup:construction}, all Reeb orbits in $Y_{[0,x_0-\delta]}$ have slope $(a_i,b_i)$ with $a_i<0$. Let $(a_i,b_i)$ be the negative ends of $u$ with multiplicity $m_i>0$. Since $[\gamma_h]=(0,-1)$, the total homology class of the negative ends satisfies
    $$ \sum_i m_i(a_i,b_i)  =(0,-1)+\ell(0,1) \quad \text{for some } \ell\in\Z.$$
    However, $a_i<0$ for all $i$ and $m_i>0$, and thus the above equality can only be satisfied if there are no negative ends.

    Finally, we show that $[u]=E\in H_2(Y,\gamma_h)$ and $u$ has {Fredholm and ECH} index $1$. Observe that $\image(u)\subset \R\times Y_{[0,x_0+\delta)}$ and $Y_{[0,x_0+\delta)}$ is a solid torus so $H_2(Y_{[0,x_0+\delta)})=0$. Therefore the boundary map in the long exact sequence of a pair is injective:
    $$H_2(Y_{[0,x_0+\delta)})=0\to H_2(Y_{[0,x_0+\delta)};\gamma_h) \xrightarrow{\partial} H_1(\gamma_h).$$
    In particular, there is a unique relative homology class for such curves $u$ which have a single positive asymptotic end at $\gamma_h$ with multiplicity $1$ (and no negative ends as we just proved). Thus $[u]=E\in H_1(Y,\gamma_h)$. Since ECH index only depends on the relative homology class, $[u]=E=[\upert]$, and $\ind(\upert)=I(\upert)=1$, we conclude that $I(u)=1$, which implies that $\ind(u)=1$ by Remark~\ref{rem:I} (Low Index Classifications).
\end{proof}

Proposition~\ref{prop:close to left}, establishes most of the results stated in Theorem~\ref{thm:no curve from right}, except that we  need to upgrade the result that $\image(u)\subset \R\times Y_{[0,x_0+\delta)}$ to show that $u$ cannot approach $\gamma_h$ asymptotically from the right. This will conclude the proof of Theorem~\ref{thm:no curve from right}.

The following lemma will be necessary input in the proof of the next proposition, which will provide this last piece.

\begin{lemma} \label{lem:disjointcylinder}
    Let $Y_{[0,x_0+\delta)}$ satisfy the assumptions of Setup~\ref{setup:construction}.
    Let $u$ be a $\Jpert$-holomorphic curve in $\R\times Y_{[0,x_0+\delta)}$ which has exactly one positive asymptotic end and no negative ends, where the positive end is asymptotic to $\gamma_h$ the positive hyperbolic Reeb orbit in $T_{x_0}$. If $u$ approaches $\gamma_h$ from the right, then it is disjoint from $\R\times \gamma_e$ and $\R\times \gamma_h$.
\end{lemma} 

\begin{proof}
    Let $u_0$ be the $J_0$-holomorphic curve of Proposition~\ref{prop:explicit_construction}. Then since $u$ and $u_0$ are both contained in $\R\times Y_{[0,x_0+\delta)}$ with the same asymptotic ends, they are in the same relative homology class by the same argument in {Lemma~\ref{lem:Q and Chern upert}}. Let $\mathcal{T}$ denote the trivial cylinder $\R\times \gamma_h$ or $\R\times \gamma_e$. {Observe that $u_0$ and $\mathcal{T}$ are disjoint because $u_0\subset \R\times Y_{[0,x_0)}$ which is disjoint from $\mathcal{T}\subset \R\times T_{x_0}$.} Thus, the relative intersection number is
    $$Q_{\tau_0}(u,\mathcal{T}) = Q_{\tau_0}(u_0,\mathcal{T}) = 0.$$
    
    If $u$ approaches $\gamma_h$ from the right, {$\zeta_1^+:=\image(u)\cap (\{s\}\times \mathcal{N}(\gamma_h))$ will lie to the right of $\gamma_h=\zeta_2^+:=\mathcal{T}\cap \{s\}\times \mathcal{N}(\gamma_h)$. Namely, $\zeta_1^+$ is pushed in the positive $V_1$ direction at all points along $\gamma_h$. {(Recall that $V_1=\partial_x$ as defined in \S~\ref{ss:MBindex}, \eqref{eqn:tau0}.)} With respect to the trivialization $\tau_0=(V_1,V_2)$, $\zeta_1^+$ gives a zero framing along $\gamma_h=\zeta_2^+$. Thus, the linking number of $\zeta_1^+$ with $\zeta_2^+$ is zero, so the asymptotic linking between $u$ and $\mathcal{T}$, $l_{\tau_0}(u,\mathcal{T})$ is zero.}

    In the case of $\mathcal{T}=\R\times \gamma_e$, we have $l_{\tau_0}(u,\mathcal{T})=0$ simply because $u$ has no asymptotic end at $\gamma_e$.
    
    By Lemma~\ref{lem:Q_tau_lk_def} the number of intersection points is given by $Q_{\tau_0}(u,\mathcal{T})+l_{\tau_0}(u,\mathcal{T})=0$ so $u$ and $\mathcal{T}$ are disjoint.
\end{proof}

\begin{proposition}
    Let $Y_{[0,x_0+\delta)}$ satisfy the assumptions of Setup~\ref{setup:construction}.
    Let $u$ be a connected $\Jpert$-holomorphic curve whose image is contained in $\R\times Y_{[0,x_0+\delta)}$ with no negative ends, whose unique positive end is asymptotic to the hyperbolic Reeb orbit $\gamma_h$ in $T_{x_0}$. Then $u$ cannot approach $\gamma_h$ from the right.
\end{proposition}

\begin{proof}
Suppose for sake of contradiction that $u$ approaches $\gamma_h$ from the right.
Let $\widetilde{B}$ be the connected component of $\image(u)\cap (\R\times Y_{[x_0,x_0+\delta)})$ containing the positive asymptotic end to $\gamma_h$ and let $B=\pi_Y(\widetilde{B})$ be its projection. Let $\widetilde{G}=-\partial \widetilde{B} = \widetilde{B}\cap (\R\times T_{x_0})$ and $G=\pi_Y(\widetilde{G})\subset T_{x_0}$, see Figure~\ref{fig:actionB}. Note $\widetilde{G}$ is oriented as the boundary of $\image(u)\cap (\R\times Y_{[0,x_0]})$, so as the negative boundary of $\widetilde{B}$. $G$ inherits an orientation from $\widetilde{G}$. 

\begin{figure}
    \centering
    \includegraphics[scale=.5]{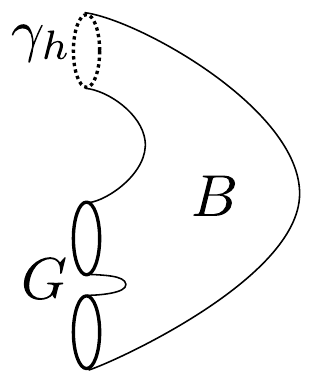}
    \caption{The projection of the component of $u$ in $Y_{[x_0,x_0+\delta)}$ with asymptotic boundary on $\gamma_h$ and actual boundary $G$.}
    \label{fig:actionB}
\end{figure}

Since $B$ is the projection of a portion of $u$,
$$0<\int_B d\lambdapert = \int_{\gamma_h} \lambdapert - \int_{G} \lambdapert = \mathcal{A}(\gamma_h)-\mathcal{A}(G).$$
Thus 
\begin{equation}\label{eqn:actionineq}
\mathcal{A}(\gamma_h) >\mathcal{A}(G).
\end{equation}

We will show that the opposite inequality holds to get a contradiction. Intuitively, $\mathcal{A}(\gamma_h)$ has the smallest action with respect to $\lambdapert$ of any cycle in $T_{x_0}$ in the same homology class because $\gamma_h$ is the minimum of the Morse function $f_0$ we used to perturb the contact form and with the unperturbed contact form, all cycles homologous to $\gamma_h$ had the same action. To prove this precisely, we will use a Stokes' theorem argument on $T_{x_0}$.

By Lemma~\ref{lem:disjointcylinder}, $G\subseteq \pi_Y(\image(u))$ is disjoint from $\gamma_h$ and $\gamma_e$. Therefore $G\subset T_{x_0}\setminus (\gamma_h\cup \gamma_e)$. $G$ may have singularities due to the projection or due to a non-transverse intersection of $\image(u)$ with $\R\times T_{x_0}$, but it still gives a $1$-cycle representing a class in $H_1(T_{x_0})$. Since $Y_{[x_0,x_0+\delta)}$ deformation retracts to $T_{x_0}$, $B$ gives a homology showing that $[G]=[\gamma_h]\in H_1(T_{x_0})$. Notice that $\gamma_h$  is oriented by the Reeb vector field orientation.

Let $A_1$ and $A_2$ be the connected components of $T_{x_0}\setminus (\gamma_h\cup \gamma_e)$, so $A_1\cup A_2 = T_{x_0}\setminus (\gamma_h\cup \gamma_e)$. Orient $A_1\cup A_2$ by $d\lambdapert|_{A_1\cup A_2} = \varepsilon e^{\varepsilon F_L} dF_L\wedge \lambda|_{A_1\cup A_2}$. Equivalently, we declare $(\nabla F_L, R)$ to be an oriented basis. Note that the orientations on $A_1$ and $A_2$ disagree as subsets of $T_{x_0}$ since $dF_L$ and $\nabla F_L$ switch sign when passing $\gamma_h$ or $\gamma_e$. Because of this, viewing $\gamma_h$ as part of the boundary of $A_1$ and $A_2$, the boundary orientation it inherits is the same as a subset of the oriented boundary of $A_1$ and the oriented boundary of $A_2$. {Here we use the convention that the boundary orientation is that which concatenates after an outward normal vector to a positive orientation. Furthermore, the boundary orientations that $A_1$ and $A_2$ both induce on $\gamma_h$ } are {opposite from the Reeb orientation on $\gamma_h$ because $\nabla F_L$ points into the interior of each $A_i$ along $\gamma_h$ since it points away from the minimum $\gamma_h$. See Figure~\ref{fig:actionA}.}

\begin{figure}
    \centering
    \includegraphics[scale=.5]{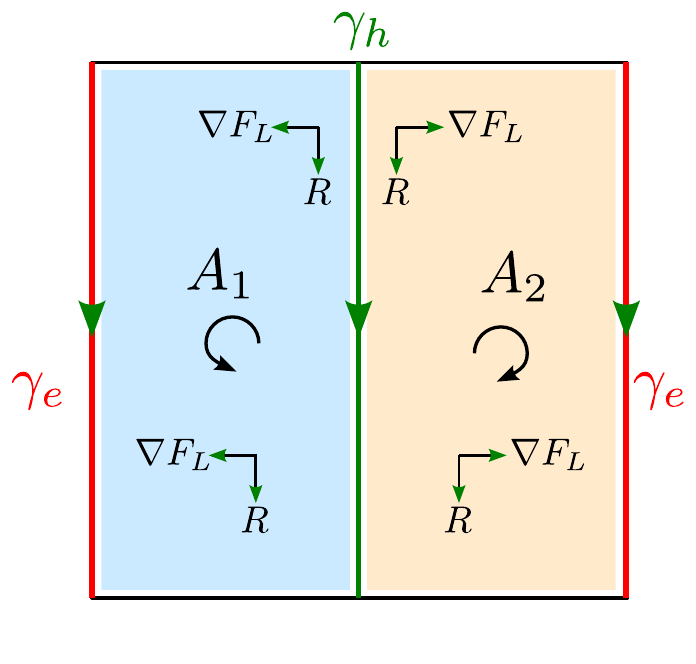}
    \caption{The orientation on $A_1\sqcup A_2$ is represented by $(\nabla F_L, R)$. The orientations on $A_1$ and $A_2$ disagree as subsets of $T_{x_0}$. The boundary orientation that $A_1$ and $A_2$ induce on $\gamma_h$ is negative the Reeb orientation.}
    \label{fig:actionA}
\end{figure}

\begin{figure}
    \centering
    \includegraphics[scale=.5]{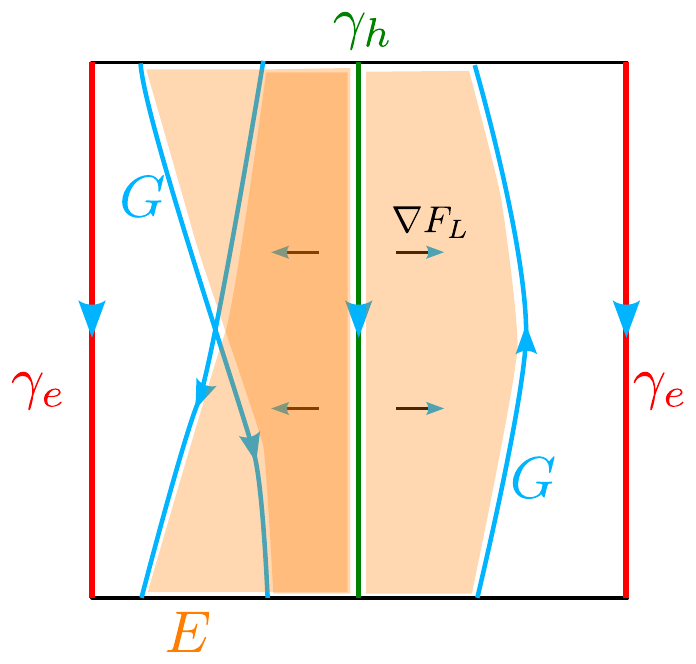}
    \caption{The torus $T_{x_0}$ with Reeb orbits $\gamma_h$ and $\gamma_e$ pointing downwards and cutting $T_{x_0}$ into two annuli oriented by $(\nabla F_L, R)$. $G$ is contained in these two annuli, and there is an oriented $2$-cycle $E$ which co-bounds $G$ and $\gamma_h$.}
    \label{fig:action}
\end{figure}

Because $[G]=[\gamma_h]\in H_1(T_{x_0})$ and $G$ is disjoint from $\gamma_h$ and $\gamma_e$, there exists an oriented $2$-chain $E\subset A_1\cup A_2$ {(oriented using the orientations of $A_1$ and $A_2$ specified above)} with $\partial E = G-\gamma_h$, see Figure~\ref{fig:action}. Hence,
$$0<\int_E d\lambdapert = \int_{G} \lambdapert - \int_{\gamma_h} \lambdapert.$$
Thus, $\mathcal{A}(G)>\mathcal{A}(\gamma_h)$, a contradiction to the opposite inequality above. 
\end{proof}

\section{Asymptotic constraints on $\Jpert$-holomorphic curves}
\label{sec:asymptotic}
In this section, we will analyze the asymptotic ends of $\Jpert$-holomorphic curves, using key results of Hofer-Wysocki-Zehnder and Siefring.  We obtain constraints on the ways such a curve can {asymptotically approach the closed positive hyperbolic orbit $\gamma_h$ at a positive end} and derive a uniqueness result under a homological condition, which will be satisfied in our setting.
Our main result of this section will be the following.

\begin{theorem} \label{theorem:asymptotics}
    For sufficiently small positive $\varepsilon$, let $\lambdapert$ be the perturbed contact form constructed in \S\ref{ss:perturb} and let $\Jpert$  be a compatible almost complex structure on $\R\times Y$. Let $\gamma_h$ be a closed positive hyperbolic Reeb orbit lying in the perturbation of a positive Morse-Bott torus at a center of the perturbation, and $\tau_0$ the trivialization from (\ref{eqn:tau0}).  
    
    Then any Fredholm index one $\Jpert$-holomorphic curve $u$ with a positive asymptotic end at $\gamma_h$ approaches from the left or from the right. 
    
    Moreover, if two Fredholm index one, $\Jpert$-holomorphic curves have $Q_{\tau_0}(u_1,u_2)=0$, no negative ends, the positive end is exactly $\gamma_h$, and $u_1$ and $u_2$ approach $\gamma_h$ from the same side then $u_1=u_2$ up to $\R$-translation.
\end{theorem}

The end of a curve that asymptotically approaches a Reeb orbit $\gamma$ lies in a neighborhood of $\R\times \gamma$. We can choose coordinates identifying a neighborhood of $\R\times \gamma$ with $\R\times S^1\times D^2$ by choosing a trivialization of $\xi$ along $\gamma$ (since $\xi$ gives a normal bundle to $\gamma$ in the contact manifold). 

In our setting of a contact toric manifold, we use the trivialization $\tau_0=\langle V_1, V_2\rangle$ given by equation (\ref{eqn:tau0}) in \textsection\ref{ss:MBindex}. Recall $V_1=\partial_x$ is transverse to the torus fibers, $V_2\in\xi\cap T(T_x)$ is tangent to the torus fibers, and $\langle V_1,V_2\rangle$ define a positive basis of $\xi$. Also, note that $CZ_{\tau_0}(\gamma_h)=0$ by Corollary~\ref{cor:CZ}.

We now analyze properties of the stable and unstable manifolds of $R_\varepsilon$ projected to a transversal slice along $\gamma_h$. 
We are interested in comparing the trivialization $\tau_0$ with the trivialization induced by these stable and unstable manifolds. We also are interested in understanding the four quadrants of the transversal plane formed in the complement of these stable and unstable manifolds. In particular, we will analyze in Lemma~\ref{lemma:quadrants} through which quadrants the projection of a pseudoholomorphic curve can approach $\gamma_h$ \emph{positively} asymptotically, and how these quadrants sit with respect to the splitting of the transversal by $T_{x_0}$.

  Let 
  \[
  v_1 = \partial_x, \quad v_2 = \frac{1}{Q(x_0)}\big(p_1(x_0)\partial_{q_2}-p_2(x_0)\partial_{q_1}\big)\quad \text{and}\quad v_3 = -\partial_{q_2}.
  \] 
  Then $(v_1,v_2,v_3)$ agrees with $(V_1,V_2,R)$ along $\gamma_h$ and give a positively oriented frame. We can choose a local coordinate system $(x_1=x-x_0,x_2,x_3)$ such that the coordinate vector fields $(\partial_{x_1},\partial_{x_2},\partial_{x_3})$ are given by 
 $(v_1,v_2,v_3)$. Then slices where $x_3$ is constant provide local transversals, $S$, to $\gamma_h$, with coordinates $(x_1,x_2)$. Note that the $x_2$-axis in such a slice is where $x=x_0$, which is exactly $S\cap T_{x_0}$. We will establish some properties of the stable and unstable manifolds of $R_\varepsilon$ projected to these transversals along the orbits of $R_\varepsilon$.

\begin{lemma}\label{lem:stableunstable}
    \begin{enumerate}
    \item The frame given by the stable and unstable manifolds of $\gamma_h$ as an orbit of $R_\varepsilon$ do not rotate with respect to the frame $V_1,V_2$. 
    \item The stable and unstable manifolds divide a transverse slice into quadrants $Q_1,Q_2,Q_3,Q_4$ such that close to $(x_1,x_2)=(0,0)$,
   \begin{itemize}
        \item       $Q_1\subset \{x_1>0\}$, 
        \item $Q_2\subset \{x_2>0\}$, 
        \item $Q_3\subset \{x_1<0\}$,  
        \item $Q_4\subset\{x_2<0\}.$ 
          \end{itemize}
          Furthermore, if we write $R_\varepsilon = aV_1+bV_2+cR$ then 
      \begin{itemize}
        \item    $b>0$ in $Q_1$, 
         \item $a>0$ in $Q_2$, 
           \item $b<0$ in $Q_3$,  
             \item$a<0$ in $Q_4$.
               \end{itemize}
   (See Figure~\ref{fig:stableunstable}.)
\end{enumerate}
\end{lemma}

\begin{proof}    
    Recall that by Lemma~\ref{lemma:perturbedReeb}, $R_\varepsilon=e^{-\varepsilon F_L}(R-\varepsilon X_{F_L})$. As in the proof of Lemma~\ref{lem:morsebott_perturbation}, we can choose an auxiliary Riemannian metric $g$ as in \S\ref{ss:perturb} such that $g(u,v)=d\lambda(u,\Jpert v)$ for $u,v\in \xi$ and $\nabla F_L\in \xi$. Then $\nabla F_L$ and $X_{F_L}=J\nabla F_L$ form a positive basis of $\xi$.
   
  The stable and unstable manifolds of $\gamma_h$, as orbits of $R_\varepsilon$, are close to the stable and unstable manifolds of $-X_{F_L}$ and they $C^1$ agree along $\gamma_h$, because the $R$ direction is projected out in the transverse slice. Thus it suffices to analyze the properties of the stable and unstable manifolds of $-X_{F_L}$ projected to the transversal. Let $-Y_{F_L}$ be the projected vector field. Additionally, in the transversal, sufficiently close to $\gamma_h$, $V_2$ is close to $v_2$ (and $V_1=v_1$ by definition). Using the relation $X_{F_L} = \Jpert(\nabla F_L)$ as in the proof of Lemma~\ref{lem:morsebott_perturbation} (2), or simply computing $X_{F_L}$ directly from the conditions $X_{F_L}\in \ker(\lambda)$, $d\lambda(\cdot, X_{F_L}) = dF_L(\cdot)$ as in Lemma~\ref{lemma:perturbedReeb}, we find the following properties as illustrated in Figure~\ref{fig:stableunstable}. 

  \begin{figure}
    \centering
    \includegraphics[scale=.5]{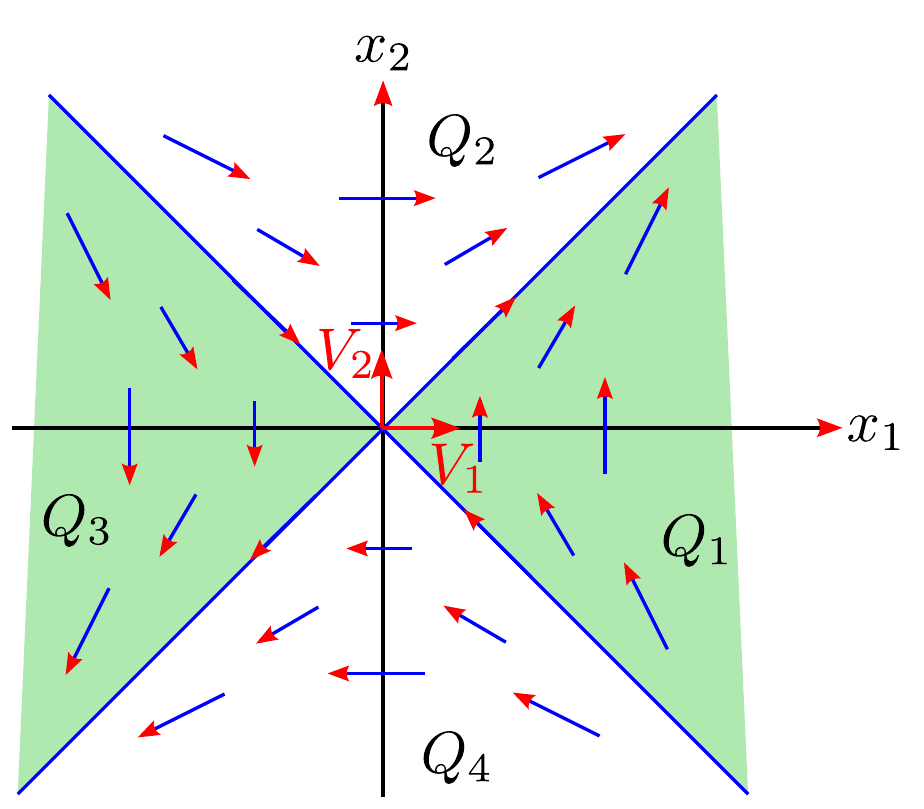}
    \caption{Projection of $R_\varepsilon$ to a $2$-dimensional slice transverse to $R_\varepsilon$ near $\gamma_h$ with the stable and unstable manifolds indicated as well as their positions relative to $V_1$ and $V_2$. The stable and unstable manifolds delimit the quadrants $Q_1$, $Q_2$, $Q_3$ and $Q_4$. The quadrants $Q_1$ and $Q_3$ are colored in green.}
    \label{fig:stableunstable}
    \end{figure}
  
  Writing $-Y_{F_L} = aV_1+bV_2$ where $a,b$ are real valued functions, one can compute:
  \begin{itemize}
        \item $b<0$ when $x<x_0$ {(equivalently, $x_1<0)$},
        \item $b>0$ when $x>x_0$ {(equivalently, $x_1>0)$},
        \item $a>0$ when $x_2>0$,
        \item $a<0$ when $x_2<0$.
  \end{itemize}

  The stable manifold is where $-Y_{F_L}$ is attracted towards $(0,0)$, so the vector field along the stable manifold must point inwards near $(0,0)$. This means that the signs of $(a,b)$ should be the opposite of the signs of $(x_1,x_2)$. Similarly, for the unstable manifold, near $(0,0)$, the signs of $(a,b)$ should be the same as the signs of $(x_1,x_2)$. Analyzing the signs in each of the $(x_1,x_2)$ quadrants, we see that the stable manifold of $-Y_{F_L}$ lies in the cone where $x_1x_2<0$, and the unstable manifold lies in the cone where $x_1x_2>0$. Note that this is truly independent of which point we are on $\gamma_h$, therefore the framing by the stable and unstable directions of $R_\varepsilon$ along a transversal agrees with the framing of $\gamma_h$ by $\tau_0=(V_1,V_2)$, that is, they do not rotate with respect to each other.

  The stable and unstable manifolds of $-Y_{F_L}$ split the transversal $S$ into four quadrants, $Q_1$, $Q_2$, $Q_3$, and $Q_4$. Observe that exactly one of these, $Q_1$, lies entirely on the right of $T_{x_0}$, and exactly one, $Q_3$, lies on the left, and the sign of $b$ is as stated.
\end{proof}

Results in~\cite{HWZI,Siefring} relate the pseudoholomorphic curve equation in an asymptotic end to a particular operator whose eigenfunctions and eigenvalues govern the possible behaviors of pseudoholomorphic curves in their asymptotic ends. 

\begin{definition}
\label{def:asymptotic_operator}
Let $\gamma$ be a nondegenerate embedded closed Reeb orbit and fix a unitary trivialization $\tau$ of the bundle $\gamma^*\xi \to S^1$. By rescaling the $t$ coordinate we may assume that $\gamma$ is parametrized by $S^1:=\R/\Z$ and has period 1. The  asymptotic operator\footnote{Our asymptotic operator is defined with an opposite sign from others in the literature, e.g. \cite{HWZI}.} associated to $\gamma$ is defined by
$$L_\gamma: = J \nabla^R_t: C^\infty(S^1, \gamma^* \xi) \to C^\infty(S^1, \gamma^* \xi),$$ 
where $\nabla^R$ denotes the symplectic connection on $\xi|_{\gamma}$ defined by the linearized Reeb flow. Nondegeneracy of the Reeb orbit $\gamma$ implies that $\gamma^*\xi$ does not have any nonzero section which is parallel with respect to $\nabla^R$, and so zero is not an eigenvalue of $L_\gamma$.

Using the unitary trivialization identifying $d\lambda|_\xi$ and $J|_\xi$ with the standard symplectic and complex structures on $\R^2 \cong \C$ with $J_0$ the associated standard complex structure, this operator can be identified with the first order differential operator $$L = J_0 \partial_t + S_t,$$ 
where $S_t= -J_0\dot \Psi(t)(\Psi(t))^{-1}$ is a symmetric matrix for each $t\in S^1$ defined in terms of the linearized Reeb flow $\Psi(t)$ along $\gamma$ on $\xi|_\gamma \cong \R^2$. Thus the operator $L$ is self-adjoint, and hence has real eigenvalues.
\end{definition}

Next, let $u$ be a pseudoholomorphic curve in $\R\times Y$. Suppose that $u$ has a positive end asymptotic to $\gamma$ and let $\mathcal{N}(\gamma)$ be a tubular neighborhood of $\gamma$.  Choose an identification $\mathcal{N}(\gamma)\simeq (\R/\Z)\times D^2$ compatible with the unitary trivialization $\tau$. The asymptotic behavior of a pseudoholomorphic curve is controlled  as follows, as elucidated by \cite{HWZI}, \cite[Theorems 2.2 and 2.3]{Siefring}, and \cite[Proposition 2.4]{obg2}.  For $s_0\gg0$,  the intersection of this positive end of $u$ with $[s_0,\infty)\times \mathcal{N}(\gamma)\subset [s_0,\infty)\times Y$ can be described as the image of a map
\begin{align*}
[s_0,\infty)\times (\R/\Z) &\longrightarrow \R \times D^2,\\
(s,t) &\longmapsto (s,f_u(s,t)),
\end{align*}
where $f_u$ at the positive end in the asymptotic limit near $\gamma$ takes on the following form (by some choice of parametrization of the domain) via \cite[Theorems 2.2 and 2.3]{Siefring}, a refinement of the asymptotic analysis appearing in \cite{HWZI}.
    \begin{equation}
    \label{eq:expansion}
    f_u(s,t) = \sum_{i\geq 0}^N e^{-\eta_i s} a_i \phi_i(t) + o_N(s,t),
    \end{equation} 
where
\begin{itemize}
    \item $\eta_i>0$ is an eigenvalue of the asymptotic operator $L$,
    \item $\phi_i:S^1\to R^2$ is the eigenfunction corresponding to $\eta_i$,
    \item $a_i \in \R$ are scalars depending on $u$, and
    \item $o_N: [R, \infty) \times S^1 \to \mathbb{R}^2$ is a function satisfying the decay estimate 
    \begin{equation}
    \label{eq:error_estimate}
    |\nabla_s^i \nabla_t^j o_N|\leq M_{ij} e^{-C s},
    \end{equation}
    for some positive constants $M_{ij}$ and $C$.
    \end{itemize}
(A negative end may be described analogously, with the sign of the eigenvalue reversed.) From this normal form, we see that the behavior of the curve as $s\to \infty$ is governed by the eigenvalues and  eigenfunctions associated to the asymptotic operator $L$. 

We call an eigenfunction $\phi_i$  \emph{leading} if its eigenvalue $\eta_i$ is the closest to zero and its coefficient $a_i \neq 0$.

Next, we recall some facts pertaining to the spectral properties of the asymptotic operator and the winding numbers of its eigenfunctions, going back to \cite[\S3]{HWZII}; see also the exposition in \cite[\S 3]{wendl2016lectures}. 

\begin{lemma}\cite[\S 3]{HWZII}, \cite[Lemma 6.4]{index} and \cite[Lemma 3.2]{Hutchings_Nelson_cyl} 
\label{lem:eigen_lemmas}
Let $\gamma$ be a nondegenerate Reeb orbit associated to a contact 3-manifold $(Y,\xi:=\ker \lambda)$.  Let $\sigma(L)$ be the spectrum of $L$. 
\begin{enumerate}
\item{Every (nonzero) eigenfunction $\varphi$ of $L$ is a nonvanishing section of $\gamma^*\xi$, and hence has a well defined winding number (with respect to a chosen unitary trivialization of $\gamma^*\xi$), which we denote by $\wind(\varphi)$.  Any two nontrivial eigenfunctions in the same eigenspace have the same winding number.  \cite[Lemma 3.5]{HWZII}}
\item Let $\varphi_1$ and $\varphi_2$ be eigenfunctions of $L$ with the same winding number and different eigenvalues. Then any nontrivial linear combination of $\varphi_1$ and $\varphi_2$ is a nonvanishing section of $\gamma^*\xi$  \cite[Lemma 3.5]{HWZII}.
\item  If $\phi_1$ and $\phi_2$ are eigenfunctions of $L$ with eigenvalues $\eta_1 \leq \eta_2$ then $\wind_{\tau}(\phi_1) \geq \wind_{\tau}(\phi_2)$. \cite[Lemma 3.7]{HWZII}\footnote{ Note the change of sign according to our definition of the asymptotic operator.}  

\item For every $w \in \Z$, $L$ has exactly two eigenvalues (counting multiplicity) for which the corresponding eigenfunctions have winding number equal to $w$. \cite[Lemma 3.6]{HWZII}.

\item \label{windCZineq} 
If $\gamma$ is a nondegenerate embedded orbit and a positive end, then for any eigenfunction $\phi_i$ with eigenvalue $\eta_i>0$, we have that
$$\wind_{\tau} (\phi_i) \leq \alpha_{\tau}^+(\gamma):=\min\{\wind_{\tau}(\phi_i) \ | \ \eta_i\in \sigma(L)\cap (0,\infty)\} = \floor{CZ_\tau(\gamma)/2}.$$ 
 If $\gamma $ is a nondegenerate embedded orbit and a negative end, then for any eigenfunction $\phi_i$ with eigenvalue $\eta_i<0$, we have that
$$\wind_{\tau} (\phi_i) \geq \alpha_{\tau}^- (\gamma) :=\max\{\wind_{\tau}(\phi_i) \ | \ \eta_i\in \sigma(L)\cap (-\infty, 0)\}= \ceil{CZ_\tau (\gamma)/2}.$$
\cite[Theorem 3.10]{HWZII}.
\end{enumerate}
\end{lemma}
The quantities $\alpha_{\tau}^+(\gamma)$ and $\alpha^{-}_{\tau}(\gamma)$ are  respectively defined to be the \emph{positive and negative extremal winding numbers}.

\begin{proposition}\cite[Proposition 3.2]{obg2} 
\label{prop:maxwinding}
If the $\lambda$-compatible almost complex structure $J$ on $\R \times Y$ is generic, then for any Fredholm index $1$, connected, non-multiply covered $J$-holomorphic curve $C$ having a positive end at $\gamma$, the winding number of the leading eigenfunction of the asymptotic expansion of $C$ achieves the equality in Lemma~\ref{lem:eigen_lemmas}(\ref{windCZineq}).
\end{proposition}

\begin{remark}
We have reformulated the above proposition, as Hutchings--Taubes observe that under these assumptions any such curve has a non-zero coefficient for an eigenfunction with extremal winding number. In particular, the leading eigenfunction of such a curve achieves the extremal winding number, as we have stated. See also the statement and proof of \cite[Lemma 3.2(c)]{Hutchings_Nelson_cyl}.  Alternatively, by Proposition 4.1 in \cite{Wendl_compactness}, any $\ind = 1$ immersed planes with a single positive asymptotic end at $\gamma_h$ must achieve extremal eigenfunctions since such planes have $c_N = 0$. This argument, due to Wendl and Hofer--Wysocki--Zehnder \cite{HWZII} does not assume genericity of $J$.
\end{remark}

We number the four quadrants by $Q_i$ for $i=1,2,3,4$, as in Figure~\ref{fig:stableunstable}, in circular order around $\gamma_h$ starting by the quadrant contained in the right side of $T_x$.
\begin{lemma}\label{lem: onesided}
If $u$ is a Fredholm index one $\Jpert$-holomorphic curve  that is not a trivial cylinder and  has a positive asymptotic end to $\gamma_h$, then:
\begin{enumerate}
    \item the leading winding number with respect to $\tau_0$,  the trivialization from (\ref{eqn:tau0}),   of the  positive asymptotic end of $u$ approaching $\gamma_h$ is $\wind_{\tau_0} (\phi_0)=0$;
    \item the image of the corresponding leading eigenfunction $\phi_0$ is contained in a single quadrant $Q_j$ of the contact plane (see Figure~\ref{fig:stableunstable}). 
\end{enumerate}
\end{lemma}

\begin{proof}
  By Corollary~\ref{cor:CZ}, $CZ_{\tau_0}(\gamma_h)=0$. Thus, by Lemma~\ref{lem:eigen_lemmas}(5) and Proposition~\ref{prop:maxwinding}, if $\phi$ is the leading asymptotic eigenfunction for the positive asymptotic end of $u$ then $\wind_{\tau_0}(\phi) =0$. This in turn implies that the winding of $u$ with respect to $\tau_0$ is equal to zero (see \cite[\S~6.2]{index}). 

    Observe that by Lemma~\ref{lem:stableunstable}, the winding number with respect to the trivialization given by the stable and unstable directions of $\gamma_h$ of a positive asymptotic end of $u$ approaching $\gamma_h$ is zero. Lemma~3.5 in \cite{CDR} guarantees that in this case, the  asymptotic end cannot intersect the stable or unstable manifolds of $\gamma_h$. Thus, the asymptotic end, must lie in a fixed quadrant defined by the stable and unstable manifolds. 
\end{proof}

    \begin{lemma}\label{lemma:quadrants}
    Let $C$ be an oriented cylinder contained in a tubular neighborhood $U$ of $\gamma_h$, such that one boundary component is $\gamma_h$ and the other one is in $\partial U$. Assume that $C\subset Q_i$ for some $1\leq i\leq 4$ and that $C$ is positively transverse to $R_\varepsilon$ (thus $d\lambdapert|_{TC}>0$). Then the boundary orientation on $\gamma_h$ induced by $C$ agrees with the orientation on $\gamma_h$ generated by $R_\varepsilon$ if and only if $C$ is in $Q_1$ or $Q_3$.
    \end{lemma}

    \begin{proof}
       Observe first that given two cylinders $C_1$, $C_2$ satisfying the hypotheses of the lemma that are contained in the same quadrant, we can isotope one to the other using the flow, preserving the orientation by $d\lambdapert$. Since $C_1$ and $C_2$ are positively transverse to $R_\varepsilon$, they induce the same orientation on $\gamma_h$. If the boundary orientation on $\gamma_h$ induced by the orientation on the cylinder agrees with the direction of $R_\varepsilon$, we say that the oriented cylinder positively orients  $\gamma_h$. 

       For each quadrant, we choose a model cylinder and determine whether it positively or negatively orients $\gamma_h$. By the isotopy described above, any other cylinder will induce the same orientation on $\gamma_h$ as the model cylinder in the same quadrant.
       
       Let $X=a_i\partial_{q_1}+b_i\partial_{q_2}$ be the vector field on $U$ that is parallel to $R|_{T_x}$ (i.e. on each torus fiber it has the slope of $\gamma_h$). In what follows, each cylinder is oriented by a vector field that points in the outward normal direction along $\gamma_h$ and $\pm X$.

       In $Q_1$ consider the cylinder $C_1$ tangent to $V_1$ and $X$, then $-V_1$ is an outward normal along $\gamma_h$. By Lemma~\ref{lem:stableunstable}, in $Q_1$, $R_\varepsilon$ has a positive component in the $V_2$ direction along $C_1$. The basis $\langle -V_1, X\rangle$ of $C_1$ orients $\gamma_h$ positively and the basis $\langle R_\varepsilon,-V_1, X\rangle$ is a positive basis of $Y$. Hence cylinders in $Q_1$ positively orient $\gamma_h$.

       In $Q_3$ consider the cylinder $C_3$ tangent to $V_1$ and $X$, then $V_1$ is  outer normal along $\gamma_h$. Here $R_\varepsilon$ has a negative component in the $V_2$ direction along $C_3$, by Lemma~\ref{lem:stableunstable}. The basis $\langle V_1, X\rangle$ of $C_3$ orients $\gamma_h$ positively and the basis $\langle R_\varepsilon,V_1, X\rangle$ is a positive basis of $Y$. Hence cylinders in $Q_3$  positively orient $\gamma_h$.

       In $Q_2$ consider the cylinder $C_2$ tangent to $V_2$ and $X$, then $-V_2$ is outer normal along $\gamma_h$. Here $R_\varepsilon$ has a positive component in the $V_1$ direction along $C_2$ by Lemma~\ref{lem:stableunstable}. The basis $\langle -V_2, -X\rangle$ of $C_2$ orients $\gamma_h$ negatively and the basis $\langle R_\varepsilon,-V_2, -X\rangle$ is a positive basis of $Y$. Hence cylinders in $Q_2$  negatively orient $\gamma_h$.

       Finally, in $Q_4$ consider the cylinder $C_4$ tangent to $V_2$ and $X$, then $V_2$ is outer normal along $\gamma_h$. Here $R_\varepsilon$ has a negative component in the $V_1$ direction along $C_4$ by Lemma~\ref{lem:stableunstable}. The basis $\langle V_2, -X\rangle$ of $C_2$ orients $\gamma_h$ negatively and the basis $\langle R_\varepsilon,V_2, -X\rangle$ is a positive basis of $Y$. Hence cylinders in $Q_2$ negatively orient $\gamma_h$.
    \end{proof}

    The lemma implies that {any} cylinder $C$ that is oriented so that it is positively transverse to $R_\varepsilon$ and that orients $\gamma_h$ positively, {must be} contained in either $Q_1$ or $Q_3$. Together with Lemma~\ref{lem: onesided}, this proves that $u$ approaches $\gamma_h$ from the left or from the right.

\begin{cor}
    If $u$ is a Fredholm index one $\Jpert$-holomorphic curve that is not a trivial cylinder and has a positive asymptotic end to $\gamma_h$ (with multiplicity $1$) then $u$ approaches $\gamma_h$ from the quadrant $Q_1$ or $Q_3$. In particular, $u$ is sided.
\end{cor}
\begin{proof}
    By Lemma~\ref{lem: onesided}(2), the image of the leading eigenfunction of $u$ is contained in a single quadrant $Q_i$. Since $u$ has a positive end on $\gamma_h$, it orients it positively and by Lemma~\ref{lemma:quadrants}, the image must be contained in either $Q_1$ or $Q_3$.
\end{proof}

The following corollary shows that if two curves are asymptotic to the same positive hyperbolic orbit after a perturbation on the Morse-Bott torus $T_x$, and they are on the same side of $T_x$, then they lie in the same quadrant given by the stable and unstable direction of the hyperbolic orbit $\gamma_h$.

\begin{cor}
\label{cor:quadrant}  
    If $u_1$ and $u_2$ are two connected Fredholm index one $\Jpert$-holomorphic curves positively asymptotic to $\gamma_h$, which approach from the same side of $T_{x}$, then their asymptotic ends are strictly contained in the same quadrant of $\gamma_h$. 
\end{cor}

\begin{proof}
    We consider one asymptotic end $\widetilde{C}_i\subset u_i$ of each curve limiting positively to $\gamma_h$. Since $u_1$ and $u_2$ lie on the same side of $T_x$, their asymptotic ends $\widetilde{C}_1$ and $\widetilde{C}_2$ approach $\gamma_h$ without winding with respect to the trivialization $\tau_0$. By Lemma~\ref{lem:stableunstable}, the winding is also zero with respect to the trivialization given by the stable and unstable manifolds of $\gamma_h$. We now project to $Y$ the two asymptotic ends to obtain near $\gamma_h$ two cylinders $C_i$. By Lemma~\ref{lemma:quadrants}, each one of these cylinders is contained in $Q_1$ or $Q_3$. Since they are on the same side of $T_x$, either the two are in $Q_1$ or the two are in $Q_3$, proving the claim.
\end{proof}

Now we show that a $\Jpert$-holomorphic curve positively asymptotically approaching $\gamma_h$ from the left (resp. right) is unique. Thus there are at most two $\Jpert$-holomorphic curves whose positive end is asymptotic to $\gamma_h$. Notice that a priori, there is no topological assumption on the curves and this is necessary for analyzing simpleness of the Reeb current $(\gamma_h,1)$ in the proof of Theorem~\ref{theorem at}~\ref{case more pi}. Similar asymptotic arguments have been used in \cite[Proposition 5.4]{Luya}. 

\begin{proposition}
\label{prop:uniqueness}
   Let $u_1$ and $u_2$ be index one pseudoholomorphic curves each with a positive asymptotic end at $\gamma_h$ and which approach $\gamma_h$ from the same side. Assume $u_1$ has no other asymptotic ends (positive or negative). Suppose that $Q_{\tau_0}(u_1, u_2) = 0$. Then, $u_1 = u_2$ up to $\R$-translation.
\end{proposition}

\begin{proof}
    Recall that we have a normal form of $u$ near the asymptotic limit at $\gamma_h$ as 
    \begin{equation}
    f_u(s,t) = \sum_{i\geq 0}^N e^{-\eta_i s} a_i \phi_i(t) + o_N(s,t)
    \end{equation} 
    satisfying Equation (\ref{eq:error_estimate}) as explained above.

    The proof is by contradiction as follows.  Suppose that we have two pseudoholomorphic curves $u_1$ and $u_2$ up to $\R$-translation asymptotic to $\gamma_h$ as a positive end and that $u_1$ and $u_2$ lie in the same quadrant. 
    We write the asymptotic expansions of $u_1$ and $u_2$ at $\gamma_h$ as:
    \begin{equation}\label{e:u1}
    f_{u_1} = \sum_{i\geq 0}^N e^{-\eta_i s} c_i \phi_i(t) +o_N(s,t)
    \end{equation}

    \begin{equation}\label{e:u2}
    f_{u_2} = \sum_{i\geq 0}^N e^{-\eta_i s} b_i \phi_i(t) +o_N(s,t)
    \end{equation}

  By Lemma \ref{lem: onesided}(1), we have that $\wind_{\tau_0}(\phi_0) = 0$, i.e., the leading eigenfunction is a vector in the plane given by the trivialization $\tau_0$.
  By Corollary~\ref{cor:quadrant}, the asymptotic ends of $u_1$ and $u_2$ approach $\gamma_h$ from the same quadrant determined by the coordinates of $\tau_0$, which by Lemma \ref{lem:stableunstable} identify with the stable and unstable directions of the linearized return map of the Reeb flow along $\gamma_h$ as in Figure~\ref{fig:stableunstable}. 
    This implies that the leading eigenterms of $u_1$ and $u_2$ are positive scalar multiple of one another. In particular, the signs of the coefficients of their leading eigenfunctions are the same. Thus in \eqref{e:u1} and \eqref{e:u2}, $c_0$ and $b_0$ must have the same sign.

    If $c_0= b_0$, then we may find the first $i>0$ such that $c_i \neq b_i$.      
    Now, by Lemma~\ref{lem:eigen_lemmas}(3), we have for any $\eta_i>\eta_0$ that $$\wind_{\tau_0}(\phi_{i}) \leq \wind_{\tau_0}(\phi_{0}).$$ 
    Recall that $\wind_{\tau_0}(\phi_0) = 0$ as above. Therefore, now by Lemma~\ref{lem:eigen_lemmas}(4), $$\wind_{\tau_0}(\phi_{i}) < \wind_{\tau_0}(\phi_{0}) = 0,$$ 
    since the winding number that is $0$ is achieved by the corresponding eigenfunctions of both the maximal negative and minimal positive eigenvalues.

    By the above discussion on winding numbers, we have that $$l_{\tau_0}(u_1, u_2) \leq \wind_{\tau_0}(\phi_{i}) \leq -1.$$

    If $c_0\neq b_0$, we may shift $u_2$ in the $\partial_s$ direction to some $s'\neq s$ so that $c_0 e^{-\eta_0 s} = b_0 e^{-\eta_0 s'}$. Let $u_2'$ denote the shifted curve. Again, this leads to $$l_{\tau_0}(u_1, u_2') \leq \wind_{\tau_0}(\phi_{i}) \leq -1.$$

    Now, since $Q_{\tau_0}(u_1, u_2) = Q_{\tau_0}(u_1, u_2') = 0$, we have a contradiction to Lemma~\ref{lem:Q_tau_lk_def} by intersection positivity for pseudoholomorphic curves, since
    \begin{equation}
        Q_{\tau_0}(u_1, u_2) + l_{\tau_0}(u_1, u_2) = \#(u_1, u_2) \geq 0,
    \end{equation}   
     and similarly for the pair $(u_1, u_2')$.

\end{proof}

\bibliographystyle{amsalpha}
\bibliography{reference.bib}

\end{document}